\documentclass[a4paper,11pt,draft]{article}
\usepackage{amsmath,amssymb,enumerate,color}
\usepackage{amsthm,color}
\usepackage{txfonts}
\usepackage{bigints}
\usepackage{comment}

\newcommand{\R}{\mathbb{R}}
\newcommand{\C}{\mathbb{C}}
\newcommand{\Z}{\mathbb{Z}}
\newcommand{\N}{\mathbb{N}}

\newcommand{\pa}{\partial}
\newcommand{\F}{\mathcal{F}}
\newcommand{\ee}{\mbox{\boldmath $1$}}

\DeclareMathOperator{\supp}{supp}
\DeclareMathOperator{\Id}{Id}

\newtheorem{theorem}{Theorem}[section]
\newtheorem{lemma}[theorem]{Lemma}
\newtheorem{proposition}[theorem]{Proposition}
\newtheorem{corollary}[theorem]{Corollary}
\theoremstyle{remark}
\newtheorem{remark}{Remark}[section]
\theoremstyle{definition}

\newtheorem{definition}{Definition}[section]

\numberwithin{equation}{section}

\makeatletter
\def\@cite#1#2{[{{\bfseries #1}\if@tempswa , #2\fi}]}
\makeatother                  
\begin{document}
\begin{center}
\Large{{\bf
Well-posedness for the fourth-order Schr\"odinger equation with third order derivative nonlinearities}}
\end{center}

\vspace{5pt}

\begin{center}
Hiroyuki Hirayama%
\footnote{
Institute for Tenure Track Promotion, University of Miyazaki, 1-1, Gakuenkibanadai-nishi, Miyazaki, 889-2192 Japan,  
E-mail:\ {\tt h.hirayama@cc.miyazaki-u.ac.jp}}, Masahiro Ikeda%
\footnote{
Department of Mathematics, Faculty of Science and Technology, Keio University, 3-14-1 Hiyoshi, Kohoku-ku, Yokohama, 223-8522, Japan/Center for Advanced Intelligence Project, RIKEN, Japan, 
E-mail:\ {\tt masahiro.ikeda@keio.jp/masahiro.ikeda@riken.jp}}
and
Tomoyuki Tanaka
\footnote{
Graduate school of Mathematics, Nagoya University, Chikusa-Ku, Nagoya, 464-8602, Japan, 
E-mail:\ {\tt d18003s@math.nagoya-u.ac.jp}}
\end{center}

\newenvironment{summary}{\vspace{.5\baselineskip}\begin{list}{}{%
     \setlength{\baselineskip}{0.85\baselineskip}
     \setlength{\topsep}{0pt}
     \setlength{\leftmargin}{12mm}
     \setlength{\rightmargin}{12mm}
     \setlength{\listparindent}{0mm}
     \setlength{\itemindent}{\listparindent}
     \setlength{\parsep}{0pt}
     \item\relax}}{\end{list}\vspace{.5\baselineskip}}
\begin{summary}
{\footnotesize {\bf Abstract.}
We study the Cauchy problem to the semilinear fourth-order Schr\"odinger equations:
\begin{equation}\label{0-1}\tag{4NLS}
\begin{cases}
i\pa_t u+\pa_x^4u=G\left(\left\{\partial_x^{k}u\right\}_{k\le \gamma},\left\{\partial_x^{k}\bar{u}\right\}_{k\le \gamma}\right), 
& t>0,\ x\in \R,
\\
\ \ \ u|_{t=0}=u_0\in H^s(\R), 
\end{cases}
\end{equation}
where $\gamma\in \{1,2,3\}$ 
and the unknown function $u=u(t,x)$ is complex valued. 
In this paper, we consider the  nonlinearity $G$ of the polynomial 
\[
        G(z)=G(z_1,\cdots,z_{2(\gamma+1)})
        :=\sum_{m\le |\alpha|\le l}C_{\alpha}z^{\alpha},
\]
for $z\in \C^{2(\gamma+1)}$, where $m,l\in\N$ with $3\le m\le l$ and $C_{\alpha}\in \C$ with $\alpha\in (\N\cup \{0\})^{2(\gamma+1)}$ is a constant.
The purpose of the present paper is to prove well-posedness of the problem (\ref{0-1}) in the lower order Sobolev space $H^s(\R)$ or with more general nonlinearities than previous results. 

\ \ Our proof of the main results is based on the contraction mapping principle on a suitable function space employed by D. Pornnopparath (2018). 
To obtain the key linear and bilinear estimates, 
we construct a suitable decomposition of the Duhamel term introduced 
by I. Bejenaru, A. D. Ionescu, C. E. Kenig, and D. Tataru (2011). 
Moreover we discuss scattering of global solutions and the optimality for the regularity of our well-posedness results, namely we prove that the flow map is not smooth in several cases.
}
\end{summary}

{\footnotesize{\it Mathematics Subject Classification}\/ (2010): %
35Q55; 
35A01; 
35B45;
37K10
}\\
{\footnotesize{\it Key words and phrases}\/: %
Schr\"odinger equations,
Fourth-order dispersion,
Well-posedness,
Low regularity,
Derivative nonlinearity,
Sobolev spaces,
Scaling critical regularity,
Modulation estimate, 
Solution map, General nonlinearity
}
\tableofcontents
\section{Introduction}
\subsection{Setting of our problem}
\ \ \ In the present paper we study well-posedness for the Cauchy problem in the Sobolev space $H^s(\R)$ of 
the Schr\"odinger equation with the fourth-order dispersion and $\gamma$-times derivative nonlinearities:
\begin{equation}\label{1-1}
\begin{cases}
i\pa_t u+\pa_x^4u=G\left(\left\{\partial_x^{k}u\right\}_{k\le {\gamma}},\left\{\partial_x^{k}\bar{u}\right\}_{k\le {\gamma}}\right), 
& (t,x)\in I\times \mathbb{R},
\\
\ \ \ \ \ \ \ \ \ u|_{t=t_0}=u_0\in H^s(\R), 
\end{cases}
\end{equation}
where $\gamma\in \{1,2,3\}$ denotes the order of the highest derivatives in the nonlinearity $G$, $i:=\sqrt{-1}$, $\partial_t:=\partial/\partial t$, $\partial_x:=\partial/\partial x$, $u=u(t,x): I\times\R\rightarrow \mathbb{C}$ is an unknown function of $(t,x)$, $t_0\in \R$ is an initial time, $(t_0\in) I$ denotes the maximal existence time interval of the function $u$, $u_0=u_0(x):\R\rightarrow\C$ is a prescribed function which belongs to a $L^2(\R)$-based $s$-th order Sobolev space $H^s(\R)$ for some $s\in \R$. Throughout this paper, we consider the nonlinear function $G:\C^{2(\gamma+1)}\rightarrow \C$ of the following polynomial:
\begin{equation}\label{nonl}        G(z)=G^{m,l}_{\gamma}(z)=G^{m,l}\left(z_1,\cdots,z_{2(\gamma+1)}\right)
        :=\sum_{m\le |\alpha|\le l}C_{\alpha}z^{\alpha},
\end{equation}
where $z\in \C^{2(\gamma+1)}$, $m\in \N$ and $l\in \N$ with $3\le m\le l$ denote the lowest degree and the highest degree of the polynomial $G$ respectively, and $C_{\alpha}\in \C$ with $\alpha\in (\N\cup \{0\})^{2(\gamma+1)}$ is a complex constant. 

The purpose of the present paper is to improve and generalize the results obtained in the previous papers \cite{S03, S04, HJ05, HJ07, HHW07, HJ11, W12, WG12, GSR14, RWZ14, HN15scri, HO16}, that is, to prove well-posedness in the lower order Sobolev space $H^s(\R)$ to the problem (\ref{1-1}) and to show well-posedness to (\ref{1-1}) with more general nonlinearities than the previous results (see Theorems \ref{cor2-4}, \ref{lwp1}, \ref{corlwp}, \ref{lwp2-2}, \ref{lwp2-3} and Remark \ref{smallne}). Here we say that well-posedness to (\ref{1-1}) holds if existence, uniqueness of the solution and continuous dependence upon the initial data are valid. We also discuss scattering of the global solutions (Theorems \ref{lwp2-2}, \ref{lwp2-3}) and the optimality of our well-posedness results, namely, prove that the flow map is not smooth in the sense of Fr\'echet derivative for some specific nonlinearity (see Remarks \ref{optima} and \ref{Optimal}).

\subsection{Background and known results}


\ \ There are many physical results and mathematical results about (\ref{1-1}) without derivative nonlinearities ($\gamma=0$) or with first order derivatives ($\gamma=1$) (see \cite{BKS00, GSR14, HN08, HN15scri, HO16, W12} and their references). We recall results closely related to the present study. Y. Wang \cite{W12} studied the Cauchy problem (\ref{1-1}) with a gauge invariant nonlinearity $\partial_x(|u|^{m-1}u)$ with odd $m\ge 5$ and proved small data global well-posedness in the scaling critical space $\dot{H}^{s_c}(\R)$, where $\dot{H}$ is the homogeneous Sobolev space and $s_c(1,m):=\frac{1}{2}-\frac{3}{m-1}$ (see Theorem 1.1 in \cite{W12}). The first author and Okamoto \cite{HO16} studied the Cauchy problem (\ref{1-1}) with $m=3$ or $m=4$ and proved large data local well-posedness in $L^2(\R)$ for a scaling invariant nonlinearity (\ref{nonl_sc_2}) below. 
In particular, they proved large data local well-posedness and small data scattering in $H^{s_c}(\R)$ for a specific nonlinearity $G=G_1^{4,4}=\partial_x\left(\bar{u}^4\right)$. (see Theorem 1.3 and Remark 3 in \cite{HO16}).  In the present paper, we improve the results obtained in \cite{W12, HO16} (see Remark \ref{contri1d} for more precise). Hayashi and Naumkin \cite{HN15scri} proved a small data scattering to the problem (\ref{1-1}) with $\gamma=1$ and $m\ge 5$ in a weighted Sobolev space (In fact, they treated real $m$ with $m>4$). They \cite{HN15cri} (resp. \cite{HN15cubi}) also study small data global existence and asymptotic behavior of solutions to the problem (\ref{1-1}) with $\gamma=1$ and the non-smooth quartic nonlinearity, i.e. $i\partial_x\left(|u|^3u\right)$ (resp. a cubic nonlinearity, $i\partial_x\left(|u|^2u\right))$ in a weighted Sobolev space.

Several models with the fourth-order dispersion, and the second-times derivative ($\gamma=2$) nonlinearities have been derived from the variational principle with Lagrange density by Karpman \cite{K96} and Karpman
and Shagalov \cite{KS00} to take into account the role of small fourth-order dispersion in the propagation of intense laser
beams in a bulk medium with Kerr nonlinearity, and the stability of the solitions for the derived equations was studied in \cite{K96, KS00}. Fukumoto and Moffatto \cite{FM00} introduced the following Schr\"odinger equation (\ref{eq2}), which contains not only the fourth-order dispersion and but also the second-order dispersion, and the second-order derivative ($\gamma=2$) nonlinearities:
\begin{equation}
\label{eq2}
    i\partial_t u+\nu\partial_x^4u+\partial_x^2u=G\left(\left\{\partial_x^ku\right\}_{k\le 2},\left\{\partial_x^k\bar{u}\right\}_{k\le 2}\right),\ \ (t,x)\in \R\times\R,
\end{equation}
where $\nu\in \R$ is a non-zero constant. Here the nonlinearity $G=G^{3,5}_{2}$ ($\gamma=2$, $m=3$, $l=5$) is given by
\begin{align}
\label{nlt}
   G\left(\left\{\partial_x^ku\right\}_{k\le 2},\left\{\partial_x^k\bar{u}\right\}_{k\le 2}\right):=&-\frac{1}{2}|u|^2u+\lambda_1|u|^4u+\lambda_2(\partial_xu)^2\bar{u}+\lambda_3|\partial_xu|^2u
   +\lambda_4u^2\partial_x^2\bar{u}+\lambda_5|u|^2\partial_x^2u,
\end{align}
where $\lambda_1:=3\mu/4$, $\lambda_2:=2\mu-\nu/2$, $\lambda_3:=4\mu+\nu$, $\lambda_4:=\mu$ and $\lambda_5:=2\mu-\nu$, with a real constant $\mu\in \R$. The equation (\ref{eq2}) describes the three dimensional motion of an isolated vortex filament embedded in an inviscid incompressible fluid filling an infinite region, and it is proposed as some detailed model taking account of the effect from the higher order corrections of the Da Rios model, that is,
\[
      i\partial_tu+\partial_x^2u=-\frac{1}{2}|u|^2u,\ \ (t,x)\in \R\times\R.
\]
This is the second order Schr\"odinger equation without derivative nonlinearities and with the cubic focusing nonlinearity, which has been extensively studied in the contexts of both physics and mathematics. It is also known that (\ref{eq2}) with (\ref{nlt}) is completely integrable, if and only if the identity $2\mu=-\nu$ holds, namely, the identities $\lambda_1:=-3\nu/8$, $\lambda_2:=-3\nu/2$, $\lambda_3:=-\nu$, $\lambda_4:=-\nu/2$ and $\lambda_6:=-2\nu$ hold (see \cite{F01}). Under the relation $2\mu=-\nu$, the equation (\ref{eq2}) has infinitely many conservation laws such as
\begin{align*}
    &\mathbf{\Phi}_0[u](t):=\frac{1}{2}\int_{\R}|u(t,x)|^2dx,\ \ \ \mathbf{\Phi}_1[u](t):=\frac{1}{2}\int_{\R}|\partial_xu(t,x)|^2dx-\frac{1}{8}\int_{\R}|u(t,x)|^4dx,\\
    &\mathbf{\Phi}_2[u](t):=\frac{1}{2}\int_{\R}|\partial_x^2u(t,x)|^2dx+\frac{3}{4}\int_{\R}|u(t,x)|^2\overline{u(t,x)}\partial_x^2u(t,x)dx+\frac{1}{8}\int_{\R}|u(t,x)|^2u(t,x)\partial_x^2\overline{u(t,x)}dx\\
    &\ \ \ \ \ \ \ \ \ \ \ \ \ +\frac{5}{8}\int_{\R}(\partial_xu(t,x))^2\overline{u(t,x)}^2dx+\frac{3}{4}\int_{\R}|\partial_xu(t,x)|^2|u(t,x)|^2dx+\frac{1}{16}\int_{\R}|u(t,x)|^6dx,\\
    &\mathbf{\Phi}_3[u](t)\ \ \cdots
\end{align*}
see \cite{LP91}. For more information about physical backgrounds of (\ref{eq2}), see \cite{F02}.

Next we recall several previous results about well-posedness of the Cauchy problem (\ref{1-1}) with second times ($\gamma=2$) derivative nonlinearities. Hao, Hisao and Wang \cite{HHW07} proved existence of local-in-time solution and uniqueness of solutions in the class $C\left(I;H^{s-1}(\R)\right)$ of the Cauchy problem (\ref{1-1}) with $3\le m\le l$ for arbitrary data in $H^s(\R)$, where $s\ge \frac{9}{2}$. We remark that in Theorem 1.1 in \cite{HHW07}, the regularity $s-1$ of the solution is weaker than $s$ of the initial data, namely, even if $u_0$ belongs to $H^s(\R)$, $u(t)$ may not be in $H^s(\R)\left(\subsetneq H^{s-1}(\R)\right)$ for some $t\in I$. However, this situation is not desirable from the viewpoint of well-posedness. 

In the present paper, we improve Theorem 1.1 in \cite{HHW07} in the following two sense. The first one is that we prove existence of a local-in-time solution to (\ref{1-1}) with $\gamma=2$ and $3\le m\le l$ for arbitrary data which belong to the wider class $\left(H^{s_1}(\R)\ \text{with} s_1\ge \frac{5}{2}\right)$ than theirs $\left(H^{s_2}(\R)\ \text{with} s_2>\frac{9}{2}\right)$. The second one is that we prove that for any $t\in I$, the solution $u(t)$ belongs to the same space as the initial data (see Theorem \ref{cor2-4})-Remark \ref{smallne}). They \cite {HHW07} also showed existence of solution locally in time and uniqueness of solutions in $H^{s-1}(\R)\cap H^6\left(\R;x^2dx\right)$ of the problem (\ref{1-1}) with $m=2$ for arbitrary data in $H^{s}(\R)\cap H^6\left(\R;x^2dx\right)$ with $s\ge \frac{25}{2}$. We note that in the case of $m=2$, some spatial decay as $|x|\rightarrow\infty$ assumption on data seems to be needed and we do not pursue the case of $m=2$ in the present paper.

Guo, Sun and Ren \cite{GSR14} proved local well-posedness of (\ref{1-1}) with the nonlinearity $G=G_2^{3,9}:=c_1u^2\partial_x^2\overline{u}+c_2|u|^8u$ for arbitrary data in $H^s(\R)$ with $s\ge \frac{1}{2}$, where $c_1,c_2\in \C$ are constants. Segata \cite{S03} showed local well-posedness of the Cauchy problem (\ref{eq2})-(\ref{nlt}) with a good sign $\nu<0$ and a special coefficient $\lambda_5=0$ for arbitrary data in $H^s(\R)$ with $s\ge \frac{1}{2}$. Huo and Jia \cite[Theorem1.1]{HJ05} proved the similar conclusion as \cite[Theorem2.1]{S03} without the sign condition $\nu<0$ but with $\nu=0$. We emphasize that one of our main results (Theorem \ref{corlwp}) reconstructs their results \cite{GSR14,S03,HJ05}. Segata \cite{S04} showed local well-posedness in $H^s(\R)$ of the problem (\ref{eq2})-(\ref{nlt}) with a good sign $\nu<0$ for arbitrary data in $H^s(\R)$ with $s>\frac{7}{12}$. We note that $\lambda_5$ in (\ref{nlt}) is not $0$ necessarily in the result \cite{S04}, whose situation is different from that in \cite{S03}. Huo and Jia \cite{HJ07} removed the sign condition $\nu<0$ in \cite[Theorem1.1]{S04} and proved local well-posedness in $H^s(\R)$ of the problem (\ref{eq2})-(\ref{nlt}) with $\nu>0$ and $s>\frac{1}{2}$. 

There are fewer physical results and fewer mathematical results about the problem (\ref{1-1}) with three times ($\gamma=3$) derivative nonlinearities than the other cases ($\gamma\in \{0,1,2\}$). It should be known that equation (\ref{1-1}) with $\gamma=3$ is completely integrable if and only if the nonlinearity $G=G_{3}^{3,7}$ is the following form:
\begin{equation}
\label{nonli}
    G_3^{3,7}\left(\left\{\partial_x^{k}u\right\}_{k\le 3},\left\{\partial_x^{k}\bar{u}\right\}_{k\le 3}\right):=
   \partial_x\left\{H_1^5+iH_2^3+\frac{5}{2}i\left(|u|^6u\right)\right\},
\end{equation}
where $H_1^5$ is a fifth-order polynomial and $H_2^3$ is a third-order polynomial, which are given by
\begin{align*}
  &H_1^5=H_1^5\left(u,\partial_xu,\overline{u},\partial_x\overline{u}\right):=\frac{3}{2}\partial_x(|u|^4u)+3\left(\bar{u}\partial_xu-u\partial_x\bar{u}\right)|u|^2u,\\
   &H_2^3=H_2^3\left(\left\{\partial_x^{k}u\right\}_{k\le 2},\left\{\partial_x^{k}\bar{u}\right\}_{k\le 2}\right):=-3(\partial_xu)^2\bar{u}+\partial_x^2(|u|^2)u,
\end{align*}
respectively (see \cite{YC06} and its references). We note that $H_1^5$ contains the first order derivative of $u$ and $\bar{u}$ and $H_2^3$ contains the second order derivative $u$ and $\bar{u}$, thus we see that the nonlinearity $G_3^{3,7}$ given by (\ref{nonli}) contains the third order derivative. It should be also notified that the equation (\ref{1-1}) with the nonlinearity $G_3^{3,7}$ given by (\ref{nonli}) belongs to a hierarchy of the derivative nonlinear Schr\"odinger equation, which can be written as
\begin{equation}
\label{1-6-a}
    i\partial_tU+\partial_x\left\{(-2i\Lambda)^{2n-1}U\right\}=0,
\end{equation}
where $n\in\N$, $U=U(t,x)=\left(u(t,x),\overline{u(t,x)}\right){}^{\mathrm{T}}:\R\times\R\rightarrow\C^2$ is a solution to (\ref{1-6-a}) and $\Lambda$ is the recursion operator (see (\ref{A-1}) for the definition).
When $n=1$, the equation (\ref{1-6-a}) is equivalent to the well-known derivative nonlinear Schr\"odinger equation:
\begin{equation}
\label{dnls}
i\partial_tu+\partial_x^2u=-i\partial_x\left(|u|^2u\right),
\end{equation}
 which describes nonlinear Alfv\'en waves in space plasma physics (see \cite{MH97}) and ultra-short pulse propagation (see \cite{A01}). The equation (\ref{dnls}) has also been extensively studied in the field of mathematics (see \cite{Ppre} and its references for example). Moreover, when $n=2$, we can see that the equation (\ref{1-6-a}) is equivalent to (\ref{1-1}) with the nonlinearity $G^{3,7}_3$ given by (\ref{nonli}) (see Appendix \ref{Aderi}, for the derivation of (\ref{1-1})-(\ref{nonli}) from the hierarchy (\ref{1-6-a}) with $n=2$). 
 
There are only two mathematical studies about well-posedness of the Cauchy problem of (\ref{1-1}) with the third order derivative nonlinearities ($\gamma=3$) as far as the authors know. Ruzhansky, Wang and Zhang \cite[Theorem1.2]{RWZ14} proved the small data global well-posedness and scattering in the modulation space $M^{3+\frac{1}{m-1}}_{2,1}(\R)$ with $m\ge 6$ (see (1.6) in \cite{RWZ14} for the definition of the modulation spaces). As in \cite[Remark1.3]{RWZ14}, if $m\ge 10$, then this result covers initial data in the Sobolev space $H^{\frac{7}{2}+\frac{1}{m-1}+\varepsilon}(\R)$ with an arbitrarily small $\varepsilon>0$.
However, it is not clear whether their solution belongs to the same space as initial data or not for $t\in I$, whose situation is not preferable from the viewpoint of well-posedness.

Huo and Jia \cite[Theorem1.1]{HJ11} proved local well-posedness of the problem (\ref{1-1}) with $\gamma=3$, $m=3$ and $l\in \N$ with $m<l$, that is $G_3^{3,l}$, for small data in $H^s(\R)$ with $s>4$. The proof of \cite[Theorem1.1]{HJ11} is based on a dyadic Fourier restriction space, which is similar to the function space (\ref{fsb}) we use in the present paper.

However, well-posedness in the Sobolev space $H^s(\R)$ for $s\le 4$ to the problem (\ref{1-1}) with the nonlinearity $G_{3}^{3,l}$ ($\gamma=3$ and $m=3$) was still a major open problem. In the present paper, we solve this problem and prove local well-posedness of the problem (\ref{1-1}) with $G_{3}^{3,l}$ for small data in $H^4(\R)$ (see Theorem \ref{cor2-4}). 

Finally we also emphasize that one of our main results (Theorem \ref{lwp1}) implies that the Cauchy problem (\ref{1-1}) with the nonlinearity $G_3^{3,7}$ given by (\ref{nonli}), where the equation (\ref{1-1}) is completely integrable and belongs to a hierarchy (\ref{1-6-a}) of the derivative nonlinear Schr\"odinger equation, is locally well-posed in $H^1(\R)$ for small data in $H^1(\R)$, which is also a completely new result. 

\ \ Equation (\ref{1-1}) is invariant under the translation with respect to time and space variables. Thus we may assume that the initial time is zero, i.e. $t_0=0$.
\subsection{Scaling critical Sobolev index}
\ \ Before stating our main results, we introduce a scaling critical Sobolev index $s_c$ for the Cauchy problem (\ref{1-1}). Such index often divides well-posedness and ill-posedness of Cauchy problems for evolution equations. If the nonlinear term $G=G_{\gamma}^{m,l}$ with $m=l$ is the following form:
\begin{equation}\label{nonl_sc}
G^{m,m}_{\gamma}\left(\left\{\partial_x^{k}u\right\}_{k\le {\gamma}},\left\{\partial_x^{k}\bar{u}\right\}_{k\le {\gamma}}\right)=\sum_{\mathbf{k}+\mathbf{l}=m}\sum_{|\alpha|+|\beta|=\gamma}C_{\alpha,\beta}^{\mathbf{k},\mathbf{l}}(\partial_x^{\alpha_1}u)\cdots(\partial_x^{\alpha_{\mathbf{k}}}u)\left(\partial_x^{\beta_1}\overline{u}\right)\cdots\left(\partial_x^{\beta_{\mathbf{l}}}\overline{u}\right),
\end{equation}
where $\alpha:=(\alpha_1,\cdots,\alpha_{\mathbf{k}})\in (\N\cup\{0\})^{\mathbf{k}}$, $\beta:=(\beta_1,\cdots,\beta_{\mathbf{l}})\in (\N\cup\{0\})^{\mathbf{l}}$ are multi-indices and $C_{\alpha,\beta}^{\mathbf{k},\mathbf{l}}\in \C$ is a constant, then equation (\ref{1-1}) is invariant under the scaling transformation $u\mapsto u_{\vartheta}$ for $\vartheta>0$, which is defined by
\[
u_{\vartheta}(t,x):=\vartheta^{\frac{4-\gamma}{m-1}} u\left(\vartheta^{4}t,\vartheta x\right),
\]
where $u:I\times\R\rightarrow \C$ is a solution to (\ref{1-1}). A simple computation gives $u_{\vartheta}(0,x)=\vartheta^{\frac{4-\gamma}{m-1}}u_0(\vartheta x)$ and
\[
      \left\|u_{\vartheta}(0,\cdot)\right\|_{\dot{H}^s}=\vartheta^{\frac{4-\gamma}{m-1}-\frac{1}{2}+s}\|u_0\|_{\dot{H}^s},
\]
where for $s\in \R$, $\dot{H}^s=\dot{H}^s(\R)$ denotes the $L^2(\R)$-based $s$-th order homogeneous Sobolev space. From this observation, we define the scaling critical (Sobolev) index $s_c$ as \[
s_c=s_c(\gamma,m):=\frac{1}{2}-\frac{4-\gamma}{m-1}.
\]
If $s=s_c$, then $\dot{H}^s$-norm of initial data is also invariant under the scaling transformation. The case $s=s_c$ is called scaling critical, the case $s>s_c$ is called scaling subcritical and the case $s<s_c$ is called scaling supercritical.

We also introduce a minimal regularity (Sobolev) exponent $s_0=s_0(\gamma,m)$ given by
\begin{equation}
\label{minir}
     s_0=s_0(\gamma, m):=
     \left\{\begin{array}{ll}
	\frac{\gamma -1}{2},&m=3,\\
	\frac{2\gamma-3}{6},&m= 4,\\
	s_c+\epsilon,&m\ge 5
	\end{array}\right.
\end{equation}
for $\gamma \in \{1,2\}$, where $\epsilon >0$ is an arbitrary positive number and
\begin{equation}
\label{minir2}
     s_0=s_0(3, m):=
     \left\{\begin{array}{ll}
	1,&m=3,\\
	\frac{1}{2},&m\ge 4. 
	\end{array}\right.
\end{equation}
We note that if $s$ satisfies $s\ge s_0(\gamma,m)$ with $\gamma\in \{1,2\}$, then $s$ belongs to the scaling subcritical case $s>s_{c}$.
\subsection{Main results}
\ \ In this subsection, we state our main results in the present paper.

\begin{theorem}[Well-posedness for general nonlinearity]
\label{cor2-4}
Let $\gamma\in \{1,2,3\}$, $m,l\in \N$ with $3\le m\le l$ and $s\ge \frac{3\gamma-1}{2}$. 
Then the Cauchy problem (\ref{1-1}) with (\ref{nonl}) is locally well-posed in $H^s(\mathbb{R})$ for small initial data $u_0\in H^s(\mathbb{R})$.
\end{theorem}

\begin{remark}
Theorem~\ref{cor2-4} with $\gamma =3$ gives extensions of \cite[Theorem1.1]{HJ11} with $n=1$. More precisely, Theorem \ref{cor2-4} with $\gamma=3$ implies local well-posedness to the problem (\ref{1-1}) with the general nonlinearity (\ref{nonl}) for small initial data in the lower order Sobolev space, that is $H^4(\R)$, than $H^{4+\epsilon}(\R)$ with a positive $\epsilon>0$, whose function space is used in \cite[Theorem1.1]{HJ11}.
\end{remark}

For the scaling invariant nonlinearity $G_{\gamma}^{m,m}$, which is defined by (\ref{nonl_sc}), we can prove local well-poseness in $H^s(\R)$ with $s\ge \max\left\{ s_0,0\right\}$, where $s_0$ is called a minimal regularity given by (\ref{minir}) and (\ref{minir2}):
\begin{theorem}[Well-posedness for scaling invariant nonlinearity]
\label{lwp1}
We assume that the nonlinearity $G=G_{\gamma}^{m,m}$ is the form of (\ref{nonl_sc}). Let $\gamma \in \{1,2,3\}$, $m\ge3$, and $s\ge \max\{s_0,0\}$.
Then the Cauchy problem (\ref{1-1}) is locally well-posed in $H^s(\mathbb{R})$ for small initial data $u_0\in H^s(\mathbb{R})$. 
\end{theorem}
The precise statement of the theorem is stated in Theorem \ref{lwpcomp}.

We can also get the following local well-posedness result to the problem (\ref{1-1}) with the following nonlinearity (\ref{nonli3}):
\begin{theorem}
\label{corlwp}
Let $\gamma\in \left\{1,2,3\right\}$ and $m,l\in \N$ with $3\le m\le l$. We assume that the nonlinear function $G_{\gamma}^{m,l}$ is the form of 
\begin{equation}
\label{nonli3}
G_{\gamma}^{m,l}\left(\left\{\partial_x^{k}u\right\}_{k\le {\gamma}},\left\{\partial_x^{k}\bar{u}\right\}_{k\le {\gamma}}\right):=\sum_{m\le \mathbf{k}+\mathbf{l}\le l}
\sum_{|\alpha|+|\beta|\le \gamma}C_{\alpha,\beta}^{\mathbf{k},\mathbf{l}}(\partial_x^{\alpha_1}u)\cdots(\partial_x^{\alpha_{\mathbf{k}}}u)\left(\partial_x^{\beta_1}\overline{u}\right)\cdots\left(\partial_x^{\beta_{\mathbf{l}}}\overline{u}\right),
\end{equation}
where $C_{\alpha,\beta}^{\mathbf{k},\mathbf{l}}\in \C$ is a constant.
Then the Cauchy problem (\ref{1-1}) is locally well-posed in $H^s(\mathbb{R})$ for small initial data $u_0\in H^s(\mathbb{R})$ with $s\ge \frac{\gamma-1}{2}$.
\end{theorem}

\begin{remark}
The nonlinearity defined in (\ref{nonli3}) 
is general form such as 
the each terms do not contain more than $\gamma$ derivatives. 
\end{remark}

\begin{remark}
\label{smallne}
In the case of $\gamma\in \{1,2\}$, namely $\gamma\ne 3$, the small assumption on the initial data $u_0$ assumed in Theorems~\ref{cor2-4}, 
 ~\ref{lwp1}, and ~\ref{corlwp} can be removed (see Theorem \ref{sc_inv_multi_e}).
\end{remark}

\begin{remark}
Theorem \ref{cor2-4} with $\gamma=2$ and Remark \ref{smallne} gives an improvement of \cite[Theorem1.1]{HHW07} in the following two sense: The first one is that Theorem \ref{cor2-4} with Remark \ref{smallne} gives existence of a local-in-time solution to (\ref{1-1}) with $\gamma=2$, $m\in [3,l]$ and the general nonlinearity (\ref{nonl}) for arbitrary data which belong to the wider class, that is, $H^{s_1}(\R)\ \text{with}\  s_1\ge \frac{5}{2}$, than $H^{s_2}(\R)\ \text{with}\ s_2>\frac{9}{2}$ used in \cite[Theorem1.1]{HHW07}. The second one is that Theorem \ref{cor2-4} with Remark \ref{smallne} verifies that for any $t\in I$, the solution $u(t)$ to the problem (\ref{1-1})-(\ref{nonl}) belongs to the same space as the initial data, whose situation improves the previous result \cite[Theorem1.1]{HHW07}.
\end{remark}

\begin{remark}
Theorem \ref{corlwp} with $\gamma=2$ and Remark \ref{smallne} gives extensions of \cite[Theorem1.1]{GSR14}, \cite[Theorem2.1]{S03}, \cite[Theorem1.1]{HJ05}, \cite[Theorem1.1]{S04} and \cite[Theorem1.1]{HJ07}. More precisely, Theorem \ref{corlwp} with $\gamma=2$ and Remark \ref{smallne} give large data local well-posedness in $H^s(\R)$ with $s\ge \frac{1}{2}$ to the problem (\ref{1-1}) with more general nonlinearities (\ref{nonli3}) with $\gamma=2$ than both $G_2^{3,9}:=c_1u^2\partial_x^2\bar{u}+c_2|u|^8u$ with $c_1,c_2\in \C$ in \cite[Thoemre1.1]{GSR14} and the physical model (\ref{nlt}) with $\lambda_5=0$ in \cite[Theorem2.1]{S03} ($\nu<0$) and \cite[Theorem1.1]{HJ05} ($\nu>0$). Moreover, Theorem \ref{corlwp} with $\gamma=2$ and Remark \ref{smallne} also imply large data local well-posedness to the problem (\ref{1-1}) with the physical model (\ref{nlt}) for the initial data in $H^s(\R)$ with the lower order regularity $s\ge \frac{1}{2}$ than $s\ge \frac{7}{12}$ and $s>\frac{1}{2}$, which are assumed in the previous results \cite[Theorem1.1]{S04} and \cite[Theorem1.1]{HJ07} respectively.
\end{remark}

\begin{remark}
Theorem \ref{corlwp} with $\gamma=3$ gives small data local well-posedness in $H^1(\R)$ to the problem (\ref{1-1}) with the nonlinearity $G_3^{3,7}$ given by (\ref{nonli}), where the equation (\ref{1-1}) is completely integrable and belongs to a hierarchy (\ref{1-6-a}) of the derivative nonlinear Schr\"odinger equation.
\end{remark}
\begin{remark}
\label{optima}
Let $\gamma\in \{1,2,3\}$. Then we can prove that the data-to-solution map $u_0\mapsto u$ to the problem (\ref{1-1}) 
with the gauge invariant cubic nonlinearity $G_{\gamma}^{3,3}\left(\left\{\partial_x^{k}u\right\}_{k\le {\gamma}},\left\{\partial_x^{k}\bar{u}\right\}_{k\le {\gamma}}\right):=\partial_x^{\gamma}(|u|^2u)$ is not $C^3$ in $H^s(\R)$ 
for $s<\frac{\gamma -1}{2}$ in the sense of Fr\'echet derivative. Indeed, if we choose $f_N\in L^2$ satisfying
\[
\widehat{f_N}(\xi):=N^{-s+\frac{1}{2}}\mathbf{1}_{[N-N^{-1},N+N^{-1}]}(\xi)
\]
for $N\gg 1$ as the initial data, then we can prove the estimate
\[
\sup_{0\le t\le 1}
\left\|\int_0^te^{i(t-t')\partial_x^4}\partial_x^{\gamma}\left(|e^{it\partial_x^4}f_N(t')|^2e^{it\partial_x^4}f_N(t')\right)dt'\right\|_{H^s}
\gtrsim N^{-2s+\gamma -1}\rightarrow \infty
\]
as $N\rightarrow\infty$ for $s<\frac{\gamma -1}{2}$ by the same argument as in the proof of \cite[Theorem1.4]{HO16}, where the implicit constant is independent of $N$. This means that Theorem~\ref{lwp1} with $m=3$ is optimal 
as long as we use the iteration argument. 
We can also obtain
\[
\sup_{0\le t\le 1}
\left\|\int_0^te^{i(t-t')\partial_x^4}\left(\left|\partial_x^{\gamma}e^{it\partial_x^4}f_N(t')\right|^2\partial_x^{\gamma}e^{it\partial_x^4}f_N(t')\right)dt'\right\|_{H^s}
\gtrsim N^{-2s+3\gamma -1}\rightarrow \infty
\]
as $N\rightarrow\infty$ for $s<\frac{3\gamma -1}{2}$. 
This means that Theorem~\ref{cor2-4} is optimal 
as long as we use the iteration argument. 
\end{remark}
%
%
%
Next we consider the following scaling invariant nonlinearity
\begin{equation}\label{nonl_sc_2}
G=G_{\gamma}^{m,m}\left(\left\{\partial_x^{k}u\right\}_{k\le {\gamma}},\left\{\partial_x^{k}\bar{u}\right\}_{k\le {\gamma}}\right):=\partial_x^{\gamma}\mathcal{P}_m(u,\overline{u}), 
\end{equation}
where $\gamma\in \{1,2,3\}$, $m\in \N$ and $\mathcal{P}_m:\C^2\rightarrow \C$ is the $m$-th order polynomial defined by
\begin{equation}
\label{poly}
\mathcal{P}(z,w)=\mathcal{P}_m(z,w):=\sum_{k=0}^m C_{k}z^{k}w^{m-k}.
\end{equation}
Here $C_{k}\in \C$ with $k\in \left\{0,\cdots,m\right\}$ is a complex constant. For the  nonlinearity (\ref{nonl_sc_2}), we can prove the following global well-posedness results in $H^s(\R)$ in the scaling critical or subcritical case $s\ge s_c$ under $\gamma=3$ and $m\ge 5$ or $\gamma\in \{1,2\}$ and $m\ge 4$:
\begin{theorem}[Well-posedness and scattering at the scaling critical regularity for $\gamma =3$]
\label{lwp2-2}
We assume that the nonlinearity $G=G_3^{m,m}$ is the form of (\ref{nonl_sc_2}). 
Let $\gamma =3$, $m\ge 5$, and $s\ge s_c$. 
Then, the Cauchy problem (\ref{1-1}) is globally well-posed in $H^s(\mathbb{R})$ for small initial data $u_0\in H^s(\mathbb{R})$. Moreover, the global solution $u$ scatters in $H^s(\R)$ as $t\rightarrow \pm \infty$.
\end{theorem}
\begin{theorem}[Well-posedness and scattering at the scaling critical regularity for $\gamma \in \{1,2\}$]
\label{lwp2-3}
We assume that the nonlinearity $G=G_{\gamma}^{m,m}$ is the form of (\ref{nonl_sc_2}). 
Let $\gamma \in \{1,2\}$, $m\ge 4$, and $s\ge s_c$. 
Then, the Cauchy problem (\ref{1-1}) is locally well-posed in $H^s(\mathbb{R})$ for arbitrary initial data $u_0\in H^s(\R)$ and globally well-posed in $H^s(\mathbb{R})$ for small initial data $u_0\in H^s(\mathbb{R})$. Moreover, the global solution $u$ scatters in $H^s(\R)$ as $t\rightarrow \pm \infty$.
\end{theorem}
%



%

%

\begin{remark}
\label{contri1d}
Theorem \ref{lwp2-3} with $\gamma =1$ gives an improvement of the results obtained in the previous papers \cite{W12} and \cite{HO16}. 
In \cite{W12}, well-posedness in the scaling critical Sobolev $H^{s_c}(\R)$ was shown only for the special gauge invariant nonlinearity $\partial_x(|u|^{m-1}u)$ with $m\ge 5$. 
In Theorem 1.2 and Remark 3 in \cite{HO16}, the scaling invariant nonlinearity $G_1^{4,4}$ given by (\ref{nonl_sc_2}) with $m=4$ was studied. Well-posedness in the scaling critical Sobolev space $H^{s_c}(\R)$ was proved only for the special nonlinearity $\partial_x(\overline{u}^4)$ (see \cite[Theorem]{HO16}). For the other quartic nonlinearities, that is, $\partial_x(u^4)$, $\partial_x(u^2|u|^2)$, $\partial_x(|u|^4)$, $\partial_x(|u|^2\overline{u}^2)$, well-posedness was proved in the space $L^2(\R)$ (see \cite[Remark 3]{HO16}), which is a smaller space than the scaling critical Sobolev space $H^{s_c}(\R)$.
\end{remark}
\begin{remark}
\label{Optimal}
Let $\gamma\in \{1,2,3\}$, $m\in \N$ with $m\ge 4$ and $s<s_c$. Then we can prove that the data-to solution map $u_0\mapsto u$ to the problem (\ref{1-1}) with a specific scaling invariant nonlinearity $G^{m,m}_{\gamma}=\partial_x^{\gamma}\left(u^m\right)$ is not $C^m$ in the same manner as the proof of \cite[Theorem 1.4 (ii)]{HO16}. This implies that Theorems \ref{lwp2-2} and \ref{lwp2-3} are optimal as long as we use the iteration argument.
\end{remark}


%


\subsection{Strategy, difficulties and idea for the proof of the main results}
\ \ The strategy of our proof of the main results is based on the contraction argument on a suitable function space employed in \cite{Ppre} with several multilinear estimates from the auxiliary space to the solution space. The multilinear estimates are basically proved by combining the linear estimates (Strichartz estimates, Kato-type smoothing estimates, Maximal function estimates, Kenig-Ruiz estimates), a suitable decomposition of the Duhamel term introduced in \cite{BIKT11}, bilinear Strichartz estimates on the solution spaces with the Littlewood-Paley decomposition and modulation estimates.

Especially, to obtain the multilinear estimates (Theorem \ref{sc_inv_multi_e}) in the case of $(\gamma,m)=(3,3)$, we employ the bilinear Strichartz estimate (Theorem \ref{bilisol}) on the solution spaces. By using Theorem \ref{sc_inv_multi_e}, we can prove the well-posedness to the problem (\ref{1-1}) with the scaling invariant nonlinearity (\ref{nonl_sc}) in the Sobolev space $H^1(\R)$. 
By using changing varables with $u=\langle \partial_x\rangle^3v$ 
and applying the multilinear estimate, we can get the well-posedness for the general nonlinearity $G_3^{3,l}$ in the Sobolev space $H^4(\R)$ (Theorem \ref{cor2-4}). This improves the previous result \cite[Theorem1.1]{HJ11}. We note that such bilinear Strichartz estimates were not used in the previous papers \cite{HJ11, RWZ14}.

In the proof of the scaling critical $s=s_{c}(\gamma,m)$ case and $(\gamma,m)=\left(3,5\right)$, $\left(2,4\right)$ or $\left(1,4\right)$ case of Theorems \ref{lwp2-2} and \ref{lwp2-3}, we need a more delicate argument than the other cases such as the scaling subcritical case $s>s_c(\gamma,m)$. Indeed, in such cases, we employ more sophisticated solution spaces and their auxiliary spaces (see (\ref{fsb})) than the spaces given by Definition \ref{def4-1}. By using these spaces, we can use so-called modulation estimates and deal with nonlinear interactions more precisely. Moreover we also prove more refined bilinear Strichartz estimates (Theorem \ref{bsere}) than Theorem \ref{bilisol} and apply them to get multilinear estimates (Theorems \ref{mest_5_cri}, \ref{mest_4_cri} and \ref{multi_est_3mg1}).

\subsection{Organization of the present paper}
\ \ The rest of the present paper is organized as follows. 
In Section \ref{pre}, we introduce several notations used throughout this paper and collect fundamental estimates in Fourier analysis and several space-time estimates for solutions to the free fourth order Schr\"odinger equation. 
We define the solution spaces and their auxiliary spaces to prove Theorems \ref{cor2-4}, \ref{lwp1}, \ref{corlwp}. In Section \ref{decom}, we derive several space-time estimates for the Duhamel term. Especially the proof of the estimate on the $L_x^1L_t^2$-norm of the Duhamel term is given by using a decomposition of it introduced in \cite{BIKT11}. In section \ref{multi}, we prove a bilinear Strichartz estimate on the solution spaces (Theorem \ref{bilisol}) via the decomposition of the Duhamel term again. Moreover we prove multilinear estimates by combining the Littlewood-Paley theory, the linear estimates and the bilinear Strichartz estimate (Theorem \ref{sc_inv_multi_e}). In section \ref{multie}, we prove multilinear estimates for several specific scaling invariant nonlinearities at the scaling critical regularity in the cases $m\ge 6$ with $\gamma=3$ and $m\ge 5$ with $\gamma=2$ via the linear estimates. We note that in section \ref{multie}, we do not need the bilinear Strichartz estimate (Theorem \ref{bilisol}). In section \ref{mesc_1}, we introduce more sophisticated solution spaces and their auxiliary spaces (\ref{fsb}) than the spaces given by Definition \ref{def4-1} to treat the similar nonlinearities as studied in Section \ref{multie} at the scaling critical regularity in the cases $m=5$ with $\gamma=3$ and $m=4$ with $\gamma=2$. We prove more refined bilinear Strichartz estimates on the solution spaces (Theorem \ref{bsere}). By applying Theorem \ref{bsere} and treating nonlinear interactions more precisely, we dirive multilinear estimates (Theorems \ref{mest_5_cri} and \ref{mest_4_cri}). In section \ref{mesc_2}, we introduce the similar solution spaces and their auxiliary spaces (\ref{fs14}) as (\ref{fsb}), to treat the similar nonlinearities as studied in Section \ref{multie} at the scaling critical regularity in the case $m=4$ with $\gamma=1$. The proof of Theorem \ref{multi_est_3mg1} is done via the almost similar manner as the proof of Theorem \ref{mest_4_cri}. In section \ref{well-po}, we give a proof of Theorem \ref{cor2-4}-Theorem \ref{lwp2-3}.


\section{Prelminaries}
\label{pre}

\subsection{Notations}
We summarize the notations used throughout this paper.
For a time interval $I$ and a Hilbert space $\mathcal{H}$, we write the function space composed of continuous functions from $I$ to $\mathcal{H}$ as $C\left(I;\mathcal{H}\right)$. 
For a Banach space $E\subset C(\R;\mathcal{H})$, 
we define the time restriction space $E(I)$ as 
\[
E(I):=\{u\in C(I;\mathcal{H})|\ {}^{\exists}v\in C(\R;\mathcal{H})\ {\rm s.t.}\ v|_{I}=u\},\ \ 
\|u\|_{E(I)}:=\inf\{\|v\|_{E}|\ v|_{I}=u\}.
\]
In particular, we write $E(T)$ instead of $E(I)$ if 
$I=[0,T]$ for $T>0$. 

For $1\leq p\leq \infty$, we denote the Lebesuge space by $L^{p}=L^{p}\left(\mathbb{R}^{\ell}\right)$, where $\ell=1$ or $2$ with the norm $\| f\| _{L^{p}}:=\left(\int_{\mathbb{R}^{\ell}}\vert f(x)\vert ^{p}dx\right)^{1/p}$ if $1\leq p<\infty $ and 
$\| f\|_{L^{\infty}}:=\text{ess.sup}_{x\in \mathbb{R}^{\ell}}\vert f(x)\vert$.
For $1\leq p,q\leq \infty$ and a time interval $I$, we use the space-time Lebesgue space $L_{t}^{p}\left(I;L_{x}^{q}\right)$ with the norm 
\[
\| u\| _{L_{t}^{p}\left(I;L_{x}^{q}\right)}:=\left\|\| u(t)\| _{L_{x}^{q}}\right\|_{L_{t}^{p}(I)}.
\] 
We also use the time-space Lebesgue space
$L_x^q\left(\R;L_t^p(I)\right)$ with the norm 
\[
\| u\| _{L_{x}^{q}\left(\R;L_{t}^{p}(I)\right)}:=\left\|\| u(x)\| _{L_{t}^{p}(I)}\right\|_{L_{x}^{q}(\R)}.
\] 
We often omit the time interval $I=[0,T]$ $(T>0)$ and the whole space $\R$, and write
$L_{T}^{p}L_{x}^{q}=L_{t}^{p}\left([0,T];L_{x}^{q}(\R)\right)$, 
$L_{x}^{q}L_{T}^{p}=L_{x}^{q}\left(\R;L_{t}^{p}([0,T])\right)$, 
$L_{t}^{p}L_{x}^{q}=L_{t}^{p}\left(\R;L_{x}^{q}(\R)\right)$, and $L_{x}^{q}L_{t}^{p}=L_{x}^{q}\left(\R;L_{t}^{p}(\R)\right)$, if they do not
cause a confusion. Let $\mathcal{S}\left(\mathbb{R}^{\ell}\right)$ be the
rapidly decaying function space. For $f\in \mathcal{S}(\mathbb{R})$, we define the Fourier transform of $f$ as 
\begin{equation*}
\mathcal{F}\left[ f\right] \left( \xi \right) =\widehat{f}\left( \xi \right) :=%
\frac{1}{\sqrt{ 2\pi}}\int_{\mathbb{R}}e^{-ix\xi
}f\left( x\right) dx,
\end{equation*}%
and the inverse Fourier transform of $f$ as%
\begin{equation*}
\mathcal{F}^{-1}\left[ f\right] \left( x\right) :=\frac{1}{\sqrt{ 2\pi
}}\int_{\mathbb{R}}e^{ix\xi }f\left( \xi \right) d\xi
,
\end{equation*}%
and extend them to $\mathcal{S}^{\prime }(\mathbb{R})$ by duality. We
also define the time-space Fourier transform of $u\in \mathcal{S}(\R\times\R)$ as 
\[
\mathcal{F}_{t,x}[u](\tau ,\xi ):=\frac{1}{\sqrt{2\pi}}\int_{\R\times\R}e^{ix\xi+it\tau}u(t,x)dtdx.
\] 
For a measurable function $\mathbf{m}:\R\rightarrow \C$, we denote the Fourier multiplier operator by $\mathbf{m}(\partial_x)$, which is given by
\begin{align}
\label{3-1-5}
	[\mathbf{m}(\partial_x) f] (x) := \mathcal{F}^{-1} \left[ \mathbf{m}(\xi) \widehat{f}(\xi) \right] (x),\ x\in \R.
\end{align}
For $s\in \R$, we denote the inhomogeneous $L^2$-based Sobolev space by $H^{s}=H^{s}(\mathbb{R})$ with the norm
\[
\Vert f\Vert _{H^{s}}:=\left\|\langle\partial_x\rangle^sf\right\|_{L^2}=\left\Vert \langle \xi \rangle ^{s}\widehat{f}\right\Vert _{L^{2}},
\]
where $\langle \cdot \rangle :=1+\vert \cdot \vert$. We also use the $L^2$-based homogeneous Sobolev space $\dot{H}^s=\dot{H}^s(\R)$ with the norm \[
\|f\|_{\dot{H}^s}:=\left\||\partial_x|^sf\right\|_{L^2}=\left\||\xi|^s\widehat{f}\right\|_{L^2}.
\]

We introduce the free propagator of the fourth-order Schr\"{o}dinger equation $\left\{e^{it\partial_x^4}\right\}_{t\in \R}$ defined by 
\begin{equation}
\label{free}
    \left(e^{it\partial_x^4}f\right)(x):=\mathcal{F}^{-1}\left[e^{it\xi^4}\widehat{f}(\xi)\right](x)=\frac{1}{\sqrt{2\pi}}\int_{\R}e^{i(x\xi+t\xi^4)}\widehat{f}(\xi)d\xi,
\end{equation}
for $(t,x)\in \R\times\R$. For a space-time function $F\in L^1_{\text{loc}}\left(0,\infty;L^2_x(\R)\right)$, we define the integral operator $\mathcal{I}$ as
\begin{equation}
\label{dua}
   \mathcal{I}[F](t):=\int_0^te^{i(t-t')\partial_x^4}F(t')dt',
\end{equation}
for $t\ge 0$ and $\mathcal{I}[F](t)=0$ otherwise. We introduce the fundamental solution $\mathcal{K}$ to the free fourth-order Schr\"odinger equation given by
\begin{equation}
\label{fs}
    \mathcal{K}=\mathcal{K}(t,x):=\mathcal{F}^{-1}_{\xi}\left[e^{it\xi^4}\right](x)=\frac{1}{\sqrt{2\pi}}\int_{\R}e^{i(x\xi+t\xi^4)}d\xi.
\end{equation}
We note that the right hand side of (\ref{fs}) is a formal expression, since $e^{it\xi^4}$ does not belong to $L^1_x(\R)$ but belong to $\mathcal{S}'(\R)$ for any $t\in \R$. Moreover, the identity 
\begin{equation}
\label{2-2-1}
   \mathcal{I}[F](t)=\frac{1}{\sqrt{2\pi}}\int_{\R}\int_0^t\mathcal{K}(t-t',x-y)F(t',y)dt'dy 
\end{equation}
holds for any $t\ge 0$. This expression is utilized to prove Proposition \ref{propdeco}.

We use the convention that
capital letters denote dyadic numbers, e.g., $N=2^{n}$ for $n\in \Z$. 
We fix a nonnegative even function 
\begin{equation}
\label{2-2-5}
\varphi \in C_{0}^{\infty }((-2,2))
\end{equation}
with $\varphi (r)=1$ for $|r|\leq 1$ and $\varphi \left( r\right) \leq 1$
for $1\leq \left\vert r\right\vert \leq 2.$ Set $\psi _{N}(r):=\varphi
(r/N)-\varphi (2r/N)$ for $N\in 2^{\mathbb{Z}}$.
For $N\in 2^{\Z}$, 
we denote the Littlewood-Palay projection by $P_{N}$, whose symbol is given by $\varphi _{N}(|\xi |),$ i.e. 
\[
(P_{N}f)(x):=\mathcal{F}^{-1}\left[\psi _{N}(|\xi |)\widehat{f}(\xi)\right](x).
\]
We also define the operators $P_{>N}:=\sum_{M>N}P_{M}$ and $P_{\leq N}:=\Id-P_{>N}.$ We often use
abbreviations $f_{N}=P_{N}f$, $f_{\leq N}=P_{\leq N}f$, etc, if they do not cause a confusion. 
For an interval $I\subset \R$, we denote the characteristic function on $I$ by $\mathbf{1}_I$, which is defined by
\[
     \mathbf{1}_I(\vartheta):=\left\{\begin{array}{ll}
	1,&\text{if}\ \vartheta\in I,\\
	0,&\text{if}\ \vartheta\in \R\backslash I.
	\end{array}\right.
\] 
We define Dirac's delta function centered at the origin as $\delta=\delta(x)\in\mathcal{S}'(\R)$.
We use the shorthand $A\lesssim B$ to denote the estimate $A\leq CB$ with
some constant $C>0$, and $A\ll B$ to denote the estimate $A\leq C^{-1}B$ for
some large constant $C>0$. The notation $A\sim B$ stands for $A\lesssim B$
and $B\lesssim A$.

Next we state the definition of the $(H^s-)$ solution to the Cauchy problem (\ref{1-1}).
\begin{definition}[$H^s$-solution]
Let $s\in \R$ and $I\subset \R$ be a time interval. A function $u:I\times\R\rightarrow\C$ is a ($H^s$-) solution to (\ref{1-1}) on $I$, if $u\in C(I;H^s(\R))$ and satisfies the Duhamel formula
\[
       u(t)=e^{it\pa_x^4}u_0-i\mathcal{I}[G(u)](t)
\]
in $H^s(\R)$-sense for any $t\in I$, where the free fourth order Schr\"odinger group $\left\{e^{it\partial_x^4}\right\}_{t\in \R}$ is given by (\ref{free}) and the integral operator $\mathcal{I}$ is given by (\ref{dua}). If the maximal existence time interval $I=\R$, then $u$ is called a global $(H^s-)$ solution to (\ref{1-1}).
\end{definition}

Next we recall the following Bernstein- and Sobolev- inequalities.
\begin{lemma}[Bernstein inequalities, Sobolev inequalities]
\label{lemBe}
Let $p,q$ satisfy $1\le p\le q\le \infty$, $s>0$ and $N\in 2^{\Z}$. Then the estimates
\begin{align*}
    \|P_{> N}f\|_{L^p(\R)}&\lesssim N^{-s}\left\||\partial_x|^sP_{> N}f\right\|_{L^p(\R)}\\
    \left\||\partial_x|^sP_{\le N}f\right\|_{L^p(\R)}&\lesssim N^{s}\left\|P_{\le N}f\right\|_{L^p(\R)}\\
    \left\||\partial_x|^{\pm s}P_{ N}f\right\|_{L^p(\R)}&\sim N^{\pm s}\left\|P_{ N}f\right\|_{L^p(\R)}\\
    \left\|P_{\le N}f\right\|_{L^q(\R)}&\lesssim N^{\frac{1}{p}-\frac{1}{q}}\left\|P_{\le N}f\right\|_{L^p(\R)}\\
    \left\|P_{ N}f\right\|_{L^p(\R)}&\lesssim N^{\frac{1}{p}-\frac{1}{q}}\left\|P_{ N}f\right\|_{L^p(\R)}
\end{align*}
hold provided that the right-hand sides are finite, where the implicit constants depend only on $p,q,s$.
\end{lemma}
For the proof of this lemma, see Appendix in \cite{T06} for example.

Next, we recall the Littlewood-Paley theorem.
\begin{lemma}[Lettlewood-Paley theorem]
\label{LP}
  Let $p\in (1,\infty)$ and $f\in L^p(\R)$. Then the equivalency
  \[
  \left\|\left(\sum_{N\in 2^{\Z}}\left|P_Nf(\cdot)\right|^2\right)^{\frac{1}{2}}\right\|_{L_x^p(\R)}\sim \|f\|_{L^p(\R)}
  \]
  holds, where the implicit constant depends only on $p$.
\end{lemma}
For the proof of this lemma, see \cite{St93} for instance.


\subsection{Several estimates for solution to the free fourth-order Schr\"odinger equation}
In this subsection, we collect several estimates of solutions to the free fourth-order Schr\"odinger equation. 
We introduce the Strichartz estimates. Before stating the estimates, we define admissible pairs as follows.
\begin{definition}[Admissible pairs]
We say that a pair $(q,r)$ is \textit{admissible} if it satisfies $2\le q,r\le \infty$ and
\[
            \frac{2}{q}+\frac{1}{r}=\frac{1}{2}.
\]
\end{definition}

\begin{lemma}[Strichartz estimates]
\label{lem3-1}
\begin{enumerate}
\item Let $(q,r)$ be admissible and $I$ be a time interval. Then the estimate
\begin{equation}
       \label{3-3-3}  \left\||\partial_x|^{\frac{2}{q}}e^{it\partial_x^4}\phi\right\|_{L_t^{q}(I;L_x^r(\R))}
         \lesssim \left\|\phi\right\|_{L_x^2(\R)}
\end{equation}
holds for any $\phi\in L^2(\R)$, where the implicit constant depends only on $q$ and $r$.
\item Let $(\tilde{q},\tilde{r})$ be admissible and $(\tilde{q}',\tilde{r}')$ be the pair of H\"older conjugate of $(\tilde{q},\tilde{r})$ and $I$ be a time interval. Then the estimate
\begin{equation}
    \label{3-4-3}    \left\||\partial_x|^{-\frac{2}{q}-\frac{2}{\tilde{q}}}\mathcal{I}[F]\right\|_{L_t^q(I;L_x^r(\R))}
        \lesssim \left\|F\right\|_{L_t^{\tilde{q}'}(I;L_x^{\tilde{r}'}(\R))}
\end{equation}
holds for any $F\in L_t^{\tilde{q}'}\left(I;L_x^{\tilde{r}'}\right)$, where the implicit constant depends only on $q,r,\tilde{q},\tilde{r}$.
\end{enumerate}
\end{lemma}
For the proof, see Proposition 3.1 in \cite{P07} or Proposition 2.3 in \cite{HO16}.

\begin{lemma}[Kato type smoothing \cite{K83}]
\label{kato}
Let $I$ be a time interval. Then the estimate
\[
\left\||\partial_x|^{\frac{3}{2}}e^{it\partial_x^4}\phi\right\|_{L^{\infty}_x\left(\R:L^2_t(I)\right)}
\le \|\phi\|_{L^2_x(\R)}
\]
holds for any $\phi\in L^2(\R)$.
\end{lemma}
The proof of this estimate can be found in \cite{K83}.


\begin{proposition}[Maximal function estimate \cite{S04}]
\label{Mfe}
Let $\epsilon>0$ and $0<T<1$. Then there exists a positive constant $C>0$
such that for any $\phi\in L^2(\R)$, the inequality
\begin{equation}
\label{3-7-1}
    \left\|\langle\partial_x\rangle^{-(1+\epsilon)}e^{it\partial_x^4}\phi\right\|_{L^2_x(\R;L^{\infty}_t([0,T])}\le C\|\phi\|_{L^2_x(\R)}
\end{equation}
holds. 
\end{proposition}

This proposition is nothing but Proposition 2.2 in \cite{S04}.

\begin{lemma}[Kenig-Ruiz estimate \cite{KR83, KPV91}]
\label{KRe}
Let $I$ be a time interval. Then there exists a positive constant $C>0$ independent of $I$ such that for any $\phi\in L^2(\R)$, the inequality
\[
     \left\||\partial_x|^{-\frac{1}{4}}e^{it\partial_x^4}\phi\right\|_{L^4_x\left(\R;L_t^{\infty}(I)\right)}\le C\|\phi\|_{L^2_x(\R)}
\]
holds.
\end{lemma}
For the proof of this lemma, see Theorem 2.5 in \cite{KPV91}.

\subsection{Introduction of function spaces and their properties}
In this subsection, we introduce solution spaces for the Cauchy problem (\ref{1-1}) and their auxiliary spaces (see also \cite{Ppre}).
\begin{definition}[$L^2(\R)$-based auxiliary space, solution space]
\label{def4-1}
For a dyadic number $N\in 2^{\N \cup \{0\}}$, the function space $Y_N$ is defined by
\[
    Y_N:=\left\{F\in L_x^1L_t^2+L_t^1L^2_x\ :\ \left\|F\right\|_{Y_N}<\infty\right\},
\]
with the norm
\begin{align*}
   \|F\|_{Y_N}:=\inf&\left\{N^{-\frac{3}{2}}\left\|F_1\right\|_{L^1_xL_t^2}+\left\|F_2\right\|_{L_t^1L_x^2}\ 
   :\ F=F_1+F_2,\ F_1\in L_x^1L_t^2,\ F_2\in L_t^1L_x^2\right\}.
\end{align*}
The function space $X_N$ is defined by
\[
    X_N:=\left\{u\in L_t^{\infty}L^2_x\ :\ \|u\|_{X_N}<\infty\right\},
\]
with the norm
\[
\begin{split}
    \|u\|_{X_N}
    &:=
    \|u\|_{L_t^{\infty}L_x^2}+N^{\frac{1}{2}}\|u\|_{L_t^{4}L_x^{\infty}}
    +N^{-(1+\epsilon)}\|u\|_{L^{2}_xL_t^{\infty}}+N^{-\frac{1}{4}}\|u\|_{L^{4}_xL_t^{\infty}}
    +N^{\frac{3}{2}}\|u\|_{L_x^{\infty}L_t^2}
+\left\|\left(i\partial_t+\partial_x^4\right)u\right\|_{Y_N}, 
\end{split}
\]
where $\epsilon >0$ is an arbitrary positive number. 
\end{definition}

\begin{remark}
\begin{enumerate}
    \item The function spaces $Y_N$ and $X_N$ given in Definition \ref{def4-1} are Banach spaces. 
    \item The power of the dyadic numbers and the function spaces in $X_N$, i.e. 
    \[N^{\frac{1}{2}}\|u\|_{L_t^{4}L_x^{\infty}},\ \ 
    N^{-(1+\epsilon)}\|u\|_{L_x^2L_t^{\infty}},\ \ 
    N^{-\frac{1}{4}}\|u\|_{L_x^4L_t^{\infty}},\ \ 
    N^{\frac{3}{2}}\|u\|_{L_x^{\infty}L_t^2},
    \]
    come from the Strichartz estimate (Lemma \ref{lem3-1}), the maximal function estimate (Proposition \ref{Mfe}), the Kenig-Ruiz estimate (Lemma \ref{KRe}) and the Kato type smoothing (Lemma \ref{kato}) respectively.
    \item In the previous work \cite{Ppre}, a similar semi-norm to $\left\|\left(i\partial_t+\partial_x^4\right)u\right\|_{Y_N}$ which appears in $X_N$-norm is used to study low regularity well-posedness for the second order Schr\"odinger equation with derivative nonlinearities:
    \[
      i\partial_tu+\partial_x^2u=G^{m,l}_{1}\left(u,\partial_xu,\overline{u},\partial_x\overline{u}\right),
    \]
    where $3\le m\le l$.
    \item In \cite{HJ11}, the authors used the function space such as Definition~\ref{def4-1} 
    for high frequency to prove the well-posedness of (\ref{1-1}) with $\gamma =3$. 
    But thay used the Besov type Fourier restriction norm instead of $L^1_TL^2_x$. 
    We will also use it for the special cases. (See, Section~\ref{mesc_1} below.)
    \item For any function $\phi=\phi(x)$ on $\R$, the solution $e^{it\partial_x^4}\phi$ to the free fourth-order Schr\"odinger equation satisfies the following identity
    \[
    \left\|\left(i\partial_t+\partial_x^4\right)e^{it\partial_x^4}\phi\right\|_{Y_N}=0,
    \]
    which implies that $\|(i\partial_t+\partial_x^4)u\|_{Y_N}$ is semi-norm.
    \item For a complex valued function $u=u(t,x)$ on $\R\times\R$, the relation \[
    \left\|\left(i\partial_t+\partial_x^4\right)\overline{u}\right\|_{Y_N}=\left\|\left(i\partial_t-\partial_x^4\right)u\right\|_{Y_N}\ne \left\|\left(i\partial_t+\partial_x^4\right)u\right\|_{Y_N}\]
    hold. This implies $\|u\|_{X_N}\ne \|\overline{u}\|_{X_N}$. 
    
\end{enumerate}
\end{remark}

The following proposition means boundedness from $L^2(\R)$ to $X_N$ for localized solutions to the free fourth-order Schr\"odinger equation.
\begin{proposition}[Estimate for localized free solutions from $L^2(\R)$ to $X_N$]
\label{prop4-1-1}
Let $0<T<1$ and $N\in 2^{\N}$. 
Then there exists a positive constant $C>0$
such that the estimate
\[
      \left\|e^{it\partial_x^4}P_N\phi\right\|_{X_N(T)}\le C\|P_N\phi\|_{L^2(\R)}
\]
holds.  
\end{proposition}

Proposition \ref{prop4-1-1} follows from the definition of $X_N(T)$-norm, 
the Strichartz estimate (\ref{3-3-3}), the Kato type smoothing (Lemma \ref{kato}), 
Kenig-Ruiz estimate (Lemma \ref{KRe}), and the maximal function estimate (\ref{3-7-1}). 
%
\begin{proposition}[Estimate for localized free solutions from $X_1$ to $L^2(\R)$]
\label{prop4-1-1-2}
Let $0<T<1$. 
Then there exists a positive constant $C>0$
such that the estimate
\[
      \left\|e^{it\partial_x^4}P_{\le 1}\phi\right\|_{X_1(T)}\le C\|P_{\le 1}\phi\|_{L^2(\R)}
\]
holds.  
\end{proposition}
\begin{proof}
Because
\[
\|e^{it\partial_x^4}P_{\le 1}\phi\|_{L_T^{\infty}L_x^2}+
\|e^{it\partial_x^4}P_{\le 1}\phi\|_{L_x^{2}L_T^{\infty}}+
\|e^{it\partial_x^4}P_{\le 1}\phi\|_{L_T^{4}L_x^{\infty}}
\lesssim \|P_{\le 1}\phi\|_{L^2_x}
\]
by the unitarity of $e^{it\partial_x^4}$ on $L^2$, 
Proposition~\ref{Mfe}, 
and Lemma~\ref{KRe}, 
it suffices to show that
\[
\left\|e^{it\partial_x^4}P_{\le 1}\phi\right\|_{L^4_TL^{\infty}_x}\lesssim \|P_{\le 1}\phi\|_{L^2(\R)} 
\]
and
\[
\left\|e^{it\partial_x^4}P_{\le 1}\phi\right\|_{L^{\infty}_xL^{2}_T}\lesssim \|P_{\le 1}\phi\|_{L^2(\R)}.
\]
By the Bernstein inequality and the unitarity of $e^{it\partial_x^4}$, we have
\[
\left\|e^{it\partial_x^4}P_{\le 1}\phi\right\|_{L^4_TL^{\infty}_x}
\lesssim \left\|e^{it\partial_x^4}P_{\le 1}\phi\right\|_{L^4_TL^{2}_x}
\lesssim T^{\frac{1}{4}}\left\|e^{it\partial_x^4}P_{\le 1}\phi\right\|_{L^{\infty}_TL^{2}_x}
\lesssim \|P_{\le 1}\phi\|_{L^2(\R)}. 
\]
On the other hand, 
we have
\[
\left\|e^{it\partial_x^4}P_{\le 1}\phi\right\|_{L^{\infty}_xL^{2}_T}
\lesssim \left\|e^{it\partial_x^4}P_{\le 1}\phi\right\|_{L^{2}_TL^{\infty}_x}
\lesssim T^{\frac{1}{2}}\left\|e^{it\partial_x^4}P_{\le 1}\phi\right\|_{L^{\infty}_TL^{2}_x}
\lesssim \|P_{\le 1}\phi\|_{L^2(\R)}
\]
by the same argument. 
\end{proof}

Next we introduce the following solution space and its auxiliary space of Besov type for $H^s$-solution to the Cauchy problem (\ref{1-1}).
\begin{definition}[$H^s(\R)$-based auxiliary space, Solution space]
\label{def3-4}
Let $s\in \R$, $T>0$. 
We define the function spaces $X^s(T)$ and $Y^s(T)$ by the norms
\[
     \|u\|_{X^s}=\|u\|_{X^s(T)}:=\|P_{\le 1}u\|_{X_1(T)}+
     \left(\sum_{N\in 2^{\N}}N^{2s}\left\|P_Nu\right\|_{X_N(T)}^2\right)^{\frac{1}{2}},
\]
\[
     \|F\|_{Y^s}=\|F\|_{Y^s(T)}:=\|P_{\le 1}F\|_{Y_1(T)}+
     \left(\sum_{N\in 2^{\N}}N^{2s}\left\|P_NF\right\|_{Y_N(T)}^2\right)^{\frac{1}{2}}
\]
respectively.
\end{definition}

\begin{remark}
The function spaces $Y^s(T)$ and $X^s(T)$ given in Definition \ref{def3-4} are Banach spaces.
\end{remark}

\begin{proposition}[Estimate for free solutions from $X^s$ to $H^s(\R)$]
\label{prop3-8}
Let $0<T<1$, 
$s\in \R$ and $\phi\in H^s(\R)$. 
Then there exists a positive constant $C_0>0$
such that the estimate
\[
      \left\|e^{it\partial_x^4}\phi\right\|_{X^s(T)}\le C_0\|\phi\|_{H^s(\R)}
\]
holds. 
\end{proposition}

Proposition \ref{prop3-8} follows from Proposition \ref{prop4-1-1}, 
Proposition \ref{prop4-1-1-2}, 
the Plancherel theorem and the properties of the Littlewood-Paley projection $P_N$.


\section{Decomposition of the Duhamel term and its application}
\label{decom}
\ \ Our aim of this section is to prove the following estimates (Theorem \ref{thm3-1_low} and  ~\ref{thm3-1}) for the Duhamel term $\mathcal{I}[F]$, where the integral operator $\mathcal{I}$ is defined by (\ref{dua}), from the auxiliary space $Y_N$ to 
the solution space $X_N$, 
whose function spaces are defined in Definition \ref{def4-1}. The proof is based on the method of the proof of Lemma 7.4 in \cite{BIKT11}. (Also see Proposition 3.3 in \cite{Ppre}.)
\begin{theorem}[Estimate for the localized Duhamel term from $Y_1$ to $X_1$]
\label{thm3-1_low}
Let $0<T<1$. 
Then there exists a positive constant $C>0$ independent of $T$ and $F$ such that the estimates
\begin{equation}
\left\|P_{\le 1}\mathcal{I}[F]\right\|_{X_1(T)}\le C\|P_{\le 1}F\|_{Y_1(T)}
\end{equation}
holds.
\end{theorem}
\begin{proof}
By Proposition \ref{prop4-1-1-2}, we have
\[
   \left\|P_{\le 1}\mathcal{I}[F]\right\|_{X_1(T)}\lesssim \int _0^T
   \left\|e^{i(t-t')\partial_x^4}P_{\le 1}F(t')\right\|_{X_1(T)}dt'
   \lesssim \int_0^T\left\|e^{-it'\partial_x^4}P_{\le 1}F(t')\right\|_{L^2_x}dt'
   \lesssim \|P_{\le 1}F\|_{L^1_TL^2_x}. 
\]
Therefore, it suffices to show that
\[
\left\|P_{\le 1}\mathcal{I}[F]\right\|_{X_1(T)}
\lesssim \|P_{\le 1}F\|_{L^1_xL^2_T}.
\]
Because we have
\[
\left\|P_{\le 1}\mathcal{I}[F]\right\|_{L^{\infty}_TL^{2}_x}
\lesssim \left\|P_{\le 1}\mathcal{I}[F]\right\|_{L^2_xL^{\infty}_T}
\]
and
\[
\left\|P_{\le 1}\mathcal{I}[F]\right\|_{L^{\infty}_xL^{2}_T}
\lesssim \left\|P_{\le 1}\mathcal{I}[F]\right\|_{L^{\infty}_xL^{4}_T}
\lesssim \left\|P_{\le 1}\mathcal{I}[F]\right\|_{L^{4}_TL^{\infty}_x}
\lesssim \left\|P_{\le 1}\mathcal{I}[F]\right\|_{L^{4}_{T,x}}
\lesssim \left\|P_{\le 1}\mathcal{I}[F]\right\|_{L^{4}_{x}L^{\infty}_T}
\]
by the H\"older inequality, and the Sobolev inequality, 
it suffices to show that
\begin{equation}\label{low_duam_keyest}
\left\|P_{\le 1}\mathcal{I}[F]\right\|_{L^{p}_xL^{\infty}_T}
\lesssim \|P_{\le 1}F\|_{L^1_xL^2_T}
\end{equation}
for $p=2$ and $4$. 
We put $\breve{\varphi}(\xi):=\varphi\left(\frac{\xi}{2}\right)$, 
$\chi (x):=\F_{\xi}^{-1}[\breve{\varphi}](x)$, where $\varphi$ is given by (\ref{2-2-5}) and 
\[
K(t,x):=\ee_{[0,1]}(t)e^{it\partial_x^4}\chi (x),\ \ 
G(t,x):=\ee_{[0,1]}(t)P_{\le 1}F(t,x).
\]
Then, we obtain $P_{\le 1}\mathcal{I}[F](t,x)=(K*G)(t,x)$ 
because $\varphi =\breve{\varphi}\varphi$, where $*$ denotes the 
time-space convolution. 
Therefore, by the Young inequality, we have
\[
\left\|P_{\le 1}\mathcal{I}[F]\right\|_{L^{p}_xL^{\infty}_T}
\lesssim \|K\|_{L^p_xL^{\infty}_T}\|G\|_{L^1_{x,T}}. 
\]
We note that 
\[
\|K\|_{L^p_xL^{\infty}_T}\lesssim \|\chi\|_{L^2_x}<\infty
\]
for $p=2$ and $4$ by Proposition~\ref{Mfe} and Lemma~\ref{KRe}. 
Therefore, by the H\"older inequality, we obtain (\ref{low_duam_keyest}). 
\end{proof}
\begin{theorem}[Estimate for the localized Duhamel term from $Y_N$ to $X_N$]
\label{thm3-1}
Let $0<T<1$ and  
$N\in 2^{\N}$. 
Then there exists a positive constant $C>0$ independent of $T$, $N$ and $F$ such that the estimate
\begin{equation}
    \left\|P_N\mathcal{I}[F]\right\|_{X_N(T)}\le C\|P_NF\|_{Y_N(T)}
\end{equation}
holds.
\end{theorem}

\begin{corollary}[Estimate for the Duhamel term from $Y^s$ to $X^s$]
\label{propdeco}
Let $s\in \R$, $0<T<1$, and $F\in Y^s(T)$. 
Then there exists a positive constant $C$ independent of $s$, $T$ and $F$ such that the estimate
\[
      \left\|\mathcal{I}[F]\right\|_{X^s(T)}\le C \|F\|_{Y^s(T)}
\]
holds.
\end{corollary}
Collorary~\ref{propdeco} follows from Theorem~\ref{thm3-1_low} and 
~\ref{thm3-1}.

Because we have
\begin{equation}\label{duamel_est_1_3}
\|P_N\mathcal{I}[F]\|_{X_N(T)}
\le C \|P_NF\|_{L^1_TL^2_x}
\end{equation}
by the same argument as in the proof of Theorem~\ref{thm3-1_low},  
to obtain Theorem~\ref{thm3-1}, it suffices to show that
\begin{equation}\label{duamel_est_2}
\|P_N\mathcal{I}[F]\|_{X_N}
\le CN^{-\frac{3}{2}}\|P_NF\|_{L^1_xL^2_T}.
\end{equation}
In the following, we focus on the proof of (\ref{duamel_est_2}). 
We assume that the function $F$ is defined in $\R\times \R$. 
Let $N\in 2^{\N}$. For $y\in \R$, we introduce the function $\mathrm{w}_y=\mathrm{w}_{y,N}:\R\times\R\rightarrow\C$ given by
\[
    \mathrm{w}_y(t,x)=\mathrm{w}_{y,N}(t,x):=
    \frac{1}{\sqrt{2\pi}}\int_0^t(\breve{P_N}\mathcal{K})(t-t',x-y)(P_NF)(t',y)dt', 
\]
where $\breve{P_N}:=P_{N/2}+P_N+P_{2N}$ 
and $\mathcal{K}\in \mathcal{S}'(\R\times\R)$ is the fundamental solution of the fourth order Schr\"odinger equation, which is defined by (\ref{fs}). 
We note that the identity
\begin{equation}\label{duamel_wy}
P_N\mathcal{I}[F](t,x)
=\int_{\R}\mathrm{w}_{y,N}(t,x)dy
\end{equation}
holds. 
Boundedness of the function $\mathrm{w}_y$ from $L^2_t(\R)$ to $X_N$ is obtained as follows:
\begin{lemma}[Estimate of the function $\mathrm{w}_y$ from $L^2(\R)$ to $X_N$]
\label{w_y_est}
Let $N\in 2^{\N}$, $y\in \R$. It holds that
\begin{equation}
    \|\mathrm{w}_y\|_{X_N}\lesssim N^{-\frac{3}{2}}\|P_NF(\cdot,y)\|_{L^2_t(\R)}.
\end{equation}
\end{lemma}

We only consider the case of $y=0$.
We put $F_0(t):=(P_NF)(t,0)$.
To obtain Lemma~\ref{w_y_est}, 
it suffices to show that
\begin{equation}\label{w_0_est}
\|w_0\|_{X_N}\lesssim N^{-\frac{3}{2}}\|F_0\|_{L^2_t(\R)}. 
\end{equation}
We introduce the operators $P_{+}$ and $P_{-}$ given by
\begin{equation}\label{proj_plus_minus}
\widehat{P_{+}f}(\xi ):=\ee_{[0,\infty)}(\xi)\widehat{f}(\xi ),\ 
\widehat{P_{-}f}(\xi ):=\ee_{(-\infty, 0]}(\xi)\widehat{f}(\xi ). 
\end{equation}
The inequality (\ref{w_0_est}) 
is implied by the following 
proposition with $L \sim N$. 
\begin{proposition}\label{duam_decom}
Let $N\in 2^{\N}$, $0<L \lesssim N$. 
Let $h\in L^{\infty}(\R^2)$ is defined by
\[
\begin{split}
w_0(t,x)=-e^{it\partial_x^4}\mathcal{L}v_0(x)
&+(P_{<L/2^{50}}\ee_{(-\infty, 0]})(x)(P_{+}e^{it\partial_x^4}v_0)(x)\\
&-(P_{<L/2^{50}}\ee_{[0,\infty)})(x)(P_{-}e^{it\partial_x^4}v_0)(x)
+h(t,x), 
\end{split}
\]
where
\[
v_0(x):=\F_{\xi}^{-1}[\psi_N(\xi)\F_t[F_0](\xi^4)](x),\ \ 
\mathcal{L}v_0(x):=\F_{\xi}^{-1}[\psi_N(\xi)\F_t[\ee_{(-\infty ,0]}F_0](\xi^4)](x).
\]
Then, $h$ satisfies
\begin{equation}\label{hest}
\|h\|_{L^q_xL^p_t}\lesssim L^{-\frac{1}{2}-\frac{1}{p}}N^{-\frac{1}{2}-\frac{1}{q}-\frac{3}{p}}\|F_0\|_{L^2_t}
\end{equation}
for any $p$, $q\ge 2$. In particular, if $N\sim L$, then we have
\[
\|h\|_{L^q_xL^p_t}\lesssim N^{-1-\frac{1}{q}-\frac{4}{p}}\|F_0\|_{L^2_t}. 
\]
\end{proposition}
\begin{proof}
We first prove that $\F_{tx}[h](\tau ,\xi )=A(\tau, \xi)\F_t[F_0](\tau )$, where
\begin{equation}\label{hA}
\begin{split}
A(\tau, \xi )&=\left\{
\psi_{N}(\xi)-(\xi^3+\xi^2\tau^{\frac{1}{4}}+\xi\tau^{\frac{1}{2}}+\tau^{\frac{3}{4}})
\frac{\ee_{\tau>0}(\tau)\psi_N(\tau^{\frac{1}{4}})}{4\tau^{\frac{3}{4}}}
\psi_{<L/2^{50}}(\xi-\tau^{\frac{1}{4}})\right.\\
&\hspace{6ex}\left.+(\xi-\tau^{\frac{1}{4}})(\xi^2+\tau^{\frac{1}{2}})
\frac{\ee_{\tau>0}(\tau)\psi_N(\tau^{\frac{1}{4}})}{4\tau^{\frac{3}{4}}}
\psi_{<L/2^{50}}(\xi+\tau^{\frac{1}{4}})
\right\}\frac{1}{i(\tau -\xi^4-i0)}. 
\end{split}
\end{equation}
Since $\ee_{(0,t]}(t')=\ee_{[0,\infty)}(t-t')-\ee_{(-\infty, 0]}(t')$, we have
\[
\begin{split}
w_0(t,x)&=\frac{1}{\sqrt{2\pi}}\left(\int_{\R}\ee_{[0,\infty)}(t-t')(P_N\mathcal{K})(t-t',x)F_0(t')dt'
-e^{it\partial_x^4}\int_{\R}\ee_{(-\infty,0]}(t')(P_N\mathcal{K})(-t',x)F_0(t')dt'\right)\\
&=:I_1-e^{it\partial_x^4}I_2.
\end{split}
\]
By the direct calculation, we have
\[
\F_x[I_2]=\frac{\psi_N(\xi)}{\sqrt{2\pi}}\int_{\R}e^{-it'\xi^4}\ee_{(-\infty,0]}(t')F_0(t')dt'
=\psi_N(\xi)\F_t[\ee_{(-\infty,0]}F_0](\xi^4)
=\widehat{\mathcal{L}v_0}(\xi).
\]
Therefore, we obtain
\[
\begin{split}
h(t,x)&=I_1
-(P_{<L/2^{50}}\ee_{(-\infty, 0]})(x)(P_{+}e^{it\partial_x^4}v_0)(x)
+(P_{<L/2^{50}}\ee_{[0,\infty)})(x)(P_{-}e^{it\partial_x^4}v_0)(x)\\
&=:I_1-J_{+}+J_{-}
\end{split}
\]
because $e^{it\partial_x^4}I_2-e^{it\partial_x^4}\mathcal{L}v_0(x)=0$. 
By the direct calculation, we have
\[
\F_{tx}[I_1](\tau, \xi)
=\F_{tx}[(\ee_{[0,\infty)}P_{N}\mathcal{K})*_tF_0](\tau, \xi)
=\frac{\psi_N(\xi)}{i(\tau -\xi^4-i0)}\F_t[F_0](\tau ). 
\]
Therefore, to obtain (\ref{hA}), it suffices to show that
\[
\F_{tx}[J_{\pm}](\tau, \xi)=Q_{\pm}(\xi, \tau^{\frac{1}{4}})\frac{\ee_{\tau>0}(\tau)\psi_N( \tau^{\frac{1}{4}})}{4\tau^{\frac{3}{4}}}\psi_{<L/2^{50}}(\xi \mp \tau^{\frac{1}{4}})
\frac{\F_t[F_0](\tau )}{i(\tau -\xi^4-i0)}, 
\]
where 
\[
Q_{\pm}(\xi, \tau^{\frac{1}{4}}):=\frac{\xi^4-\tau}{\xi \mp \tau^{\frac{1}{4}}}. 
\]
We note that
\[
\xi^4-\tau=(\xi-\tau^{\frac{1}{4}})(\xi^3+\xi^2\tau^{\frac{1}{4}}+\xi\tau^{\frac{1}{2}}+\tau^{\frac{3}{4}})
=(\xi +\tau^{\frac{1}{4}})(\xi-\tau^{\frac{1}{4}})(\xi^2+\tau^{\frac{1}{2}}). 
\]
By the direct calculation, we have
\[
\F_x[P_{<L/2^{50}}\ee_{(-\infty, 0]}](\xi )
=-\frac{\psi_{<L/2^{50}}(\xi)}{i(\xi +i0)},\ \ 
\F_x[P_{<L/2^{50}}\ee_{[0,\infty)}](\xi )
=-\frac{\psi_{<L/2^{50}}(\xi)}{i(\xi -i0)}
\]
and
\[
\F_{tx}[P_{\pm}e^{it\partial_x^4}v_0](\tau, \xi)
=\delta (\tau -\xi^4)\ee_{\xi \gtrless 0}(\xi )\psi_{N}(\xi )\F_t[F_0](\xi^4). 
\]
Therefore, 
by using the variable transform $\eta \mapsto \omega$ as 
$\eta =\pm \omega^{\frac{1}{4}}$, we have
\begin{equation}\label{pm_fourier_cal}
\begin{split}
\F_{tx}[J_{\pm}](\tau, \xi)
&=-\left(\frac{\psi_{<L/2^{50}}(\xi)}{i(\xi\pm i0)}\right)*_{\xi}
\left(\delta (\tau-\xi^4)\ee_{\xi \gtrless 0}(\xi )\psi_{N}(\xi )\F_t[F_0](\xi^4)\right)\\
&=\mp \int_{0}^{\pm\infty}
\frac{\psi_{<L/2^{50}}(\xi-\eta)}{i(\xi-\eta\pm i0)}
\delta (\tau -\eta^4)\psi_N(\eta)\F_t[F_0](\eta^4)d\eta\\
&=\mp\int_0^{\infty}\frac{\psi_{<L/2^{50}}(\xi\mp \omega^{\frac{1}{4}})}{i(\xi\mp\omega^{\frac{1}{4}}\pm i0)}\delta (\tau-\omega)\psi_N(\omega^{\frac{1}{4}})\F_t[F_0](\omega)
\frac{\pm 1}{4\omega^{\frac{3}{4}}}d\omega\\
&=-\frac{\psi_{<L/2^{50}}(\xi\mp \tau^{\frac{1}{4}})}{i(\xi\mp\tau^{\frac{1}{4}}\pm i0)}\psi_N(\tau^{\frac{1}{4}})\F_t[F_0](\tau)
\frac{\ee_{\tau>0}(\tau)}{4\tau^{\frac{3}{4}}}\\
&=Q_{\pm}(\xi, \tau^{\frac{1}{4}})\frac{\ee_{\tau>0}(\tau)\psi_N( \tau^{\frac{1}{4}})}{4\tau^{\frac{3}{4}}}\psi_{<L/2^{50}}(\xi \mp \tau^{\frac{1}{4}})
\frac{\F_t[F_0](\tau )}{i(\tau -\xi^4-i0)}. 
\end{split}
\end{equation}
As a result, we obtain (\ref{hA}). 

Next, we prove (\ref{hest}). 
We divide $A(\tau, \xi)$ into 
\[
A(\tau, \xi)=\sum_{j=1}^4A_j(\tau, \xi),\ \ 
A_j(\tau, \xi)=\ee_{\Omega_j}(\tau,\xi)A(\tau, \xi),
\]
where
\[
\begin{split}
\Omega_1&:=\left\{(\tau, \xi)\left|\ \tau >0,\ 
|\xi-\tau^{\frac{1}{4}}|<\frac{L}{2^{100}}\right.\right\},\ \ 
\Omega_2:=\left\{(\tau, \xi)\left|\ \tau >0,\ 
|\xi+\tau^{\frac{1}{4}}|<\frac{L}{2^{100}}\right.\right\},\\
\Omega_3&:=\left\{(\tau, \xi)\left|\ 
\tau \le 0\right.\right\},\ \ 
\Omega_4:=\R^2\backslash(\Omega_1\cup\Omega_2\cup\Omega_3). 
\end{split}
\]
First, we assume $(\tau, \xi)\in \Omega_1$. 
Then, we have $\psi_{<L/2^{50}}(\xi-\tau^{\frac{1}{4}})=1$, 
$\psi_{<L/2^{50}}(\xi+\tau^{\frac{1}{4}})=0$, and $\xi\sim \tau^{\frac{1}{4}}>0$ 
if $\xi \sim N$ or $\tau^{\frac{1}{4}}\sim N$. 
Furthermore, by the Talor expansion, we obtain
\[
\psi_N(\xi)=\psi_N(\tau^{\frac{1}{4}})+(\xi-\tau^{\frac{1}{4}})O(N^{-1}). 
\]
Therefore, we have
\[
\begin{split}
A_1(\tau, \xi)&=
\psi_N(\tau^{\frac{1}{4}})\left(1-\frac{\xi^3+\xi^2\tau^{\frac{1}{4}}+\xi\tau^{\frac{1}{2}}
+\tau^{\frac{3}{4}}}{4\tau^{\frac{3}{4}}}\right)\frac{1}{i(\tau-\xi^4-i0)}
+\frac{\xi-\tau^{\frac{1}{4}}}{i(\tau-\xi^4-i0)}O(N^{-1})\\
&=\frac{\psi_N(\tau^{\frac{1}{4}})}{4\tau^{\frac{3}{4}}}
\frac{3\tau^{\frac{1}{2}}+2\xi\tau^{\frac{1}{4}}+\xi^2}
{i(\tau^{\frac{3}{4}}+\xi\tau^{\frac{1}{2}}+\xi^2\tau^{\frac{1}{4}}+\xi^3)}
-\frac{1}{i(\tau^{\frac{3}{4}}+\xi\tau^{\frac{1}{2}}+\xi^2\tau^{\frac{1}{4}}+\xi^3)}O(N^{-1}). 
\end{split}
\]
It implies that 
\[
\begin{split}
\|A_1\|_{L^{r}_{\tau}L^2_{\xi}}^{r}
&\lesssim 
\int_{0<\tau \sim N^4}\frac{1}{\tau^{\frac{3}{4}r}}\left(
\int_{0<\xi \sim N}
\frac{(3\tau^{\frac{1}{2}}+2\xi\tau^{\frac{1}{4}}+\xi^2)^2}
{(\tau^{\frac{3}{4}}+\xi\tau^{\frac{1}{2}}+\xi^2\tau^{\frac{1}{4}}+\xi^3)^2}
d\xi\right)^{\frac{r}{2}}d\tau\\
&\ \ \ \ +\int_{0<\tau \sim N^4}\left(
\int_{0<\xi \sim N}
\frac{N^{-2}}
{(\tau^{\frac{3}{4}}+\xi\tau^{\frac{1}{2}}+\xi^2\tau^{\frac{1}{4}}+\xi^3)^2}
d\xi\right)^{\frac{r}{2}}d\tau\\
&\lesssim N^{-(\frac{7}{2}r-4)}
\end{split}
\]
for $r\ge 2$ and
\[
\begin{split}
\|A_1\|_{L^{\infty}_{\tau}L^2_{\xi}}
&\lesssim 
\sup_{0<\tau \sim N^4}\frac{1}{\tau^{\frac{3}{4}}}\left(
\int_{0<\xi \sim N}
\frac{(3\tau^{\frac{1}{2}}+2\xi\tau^{\frac{1}{4}}+\xi^2)^2}
{(\tau^{\frac{3}{4}}+\xi\tau^{\frac{1}{2}}+\xi^2\tau^{\frac{1}{4}}+\xi^3)^2}
d\xi\right)^{\frac{1}{2}}d\tau\\
&\ \ \ \ +\sup_{0<\tau \sim N^4}\left(
\int_{0<\xi \sim N}
\frac{N^{-2}}
{(\tau^{\frac{3}{4}}+\xi\tau^{\frac{1}{2}}+\xi^2\tau^{\frac{1}{4}}+\xi^3)^2}
d\xi\right)^{\frac{1}{2}}d\tau\\
&\lesssim N^{-\frac{7}{2}}.
\end{split}
\]
By the same argument, we obtain
\[
\|A_2\|_{L^{r}_{\tau}L^2_{\xi}}\lesssim N^{-(\frac{7}{2}-\frac{4}{r})}
\]
for $2\le r\le \infty$. 
Next, we assume $(\tau, \xi)\in \Omega_3$. 
Therefore, we have 
\[
A_3(\tau,\xi)=\frac{\psi_N(\xi)}{i(\tau -\xi^4-i0)}. 
\]
We note that
\[
|\tau -\xi^4|\ge -\tau +\frac{N^4}{16}\sim |\tau|+N^4
\]
when $\tau\le 0$ and $|\xi| \ge \frac{N}{2}$. 
It implies that 
\[
\begin{split}
\|A_3\|_{L^{r}_{\tau}L^2_{\xi}}^{r}
&\lesssim 
\int_{-\infty}^{0}\frac{1}{(|\tau|+N^4)^{r}}\left(
\int_{|\xi| \sim N}
d\xi\right)^{\frac{r}{2}}d\tau
\lesssim N^{-(\frac{7}{2}r-4)}
\end{split}
\]
for $r\ge 2$ and
\[
\begin{split}
\|A_3\|_{L^{\infty}_{\tau}L^2_{\xi}}
&\lesssim 
\sup_{\tau\le 0}\frac{1}{(|\tau|+N^4)}\left(
\int_{|\xi| \sim N}
d\xi\right)^{\frac{1}{2}}
\lesssim N^{-\frac{7}{2}}.
\end{split}
\]
Finally, we assume $(\tau, \xi)\in \Omega_4$. 
Then, we obtain
\[
|\tau-\xi^4|=|\tau^{\frac{1}{4}}-\xi||\tau^{\frac{1}{4}}+\xi||\tau^{\frac{1}{2}}+\xi^2|
\gtrsim L N^3 
\]
if $|\xi |\sim N$ or $\tau^{\frac{1}{4}}\sim N$. It implies that
\[
|A_4(\tau, \xi)|
\lesssim \left(\psi_N(\xi)+\psi_N(\tau^{\frac{1}{4}})\psi_{<L/2^{50}}(\xi-\tau^{\frac{1}{4}})
+\psi_N(-\tau^{\frac{1}{4}})\psi_{<L/2^{50}}(\xi+\tau^{\frac{1}{4}})\right)\frac{1}{|\tau -\xi^4|}. 
\]
Therefore, we have
\[
\begin{split}
\|A_4\|_{L^{r}_{\tau}L^2_{\xi}}^{2}
&\lesssim \|A_4\|_{L^2_{\xi}L^{r}_{\tau}}^2
\lesssim 
\int_{|\xi| \sim N}\left(
\int_{|\tau -\xi^4| \gtrsim L N^3}
\frac{1}
{(\tau-\xi^4)^r}
d\tau\right)^{\frac{2}{r}}d\xi
\lesssim L^{-2(1-\frac{1}{r})}N^{-(5-\frac{6}{r})}
\end{split}
\]
for $r\ge 2$ and
\[
\|A_4\|_{L^{\infty}_{\tau}L^2_{\xi}}
\lesssim \sup_{\tau}\left(\int_{|\xi|\sim N, |\tau -\xi^4| \gtrsim L N^3}
\frac{1}
{(\tau-\xi^4)^2}
d\xi\right)^{\frac{1}{2}}
\lesssim L^{-1}N^{-\frac{5}{2}}.
\]
As a result, for $2\le r\le \infty$, we obtain
\begin{equation}\label{Aest}
\|A\|_{L^{r}_{\tau}L^2_{\xi}}\lesssim L^{-(1-\frac{1}{r})}N^{-(\frac{5}{2}-\frac{3}{r})}. 
\end{equation}
For $p$, $q\ge 2$, we put 
\[
\frac{1}{p'}:=1-\frac{1}{p},\ \frac{1}{r}:=\frac{1}{p'}-\frac{1}{2}=\frac{1}{2}-\frac{1}{p}. 
\]
We note that $r\ge 2$. 
Because $1\le p'\le 2\le q$ and $\F_{tx}[h](\tau, \xi)=A(\tau, \xi)\F_t[F_0](\tau)$, we have
\[
\|h\|_{L^{q}_xL^{p}_t}\lesssim \|\F_t[h]\|_{L^q_{x}L^{p'}_{\tau}}
\lesssim \|\F_t[h]\|_{L^{p'}_{\tau}L^q_{x}}
\lesssim N^{\frac{1}{2}-\frac{1}{q}}\|\F_{tx}[h]\|_{L^{p'}_{\tau}L^2_{\xi}}
\lesssim N^{\frac{1}{2}-\frac{1}{q}}\|A\|_{L^{r}_{\tau}L^2_{\xi}}\|\F_t[F_0]\|_{L^2_{\tau}}. 
\]
Therefore, we obtain 
\[
\|h\|_{L^{q}_xL^{p}_t}\lesssim N^{\frac{1}{2}-\frac{1}{q}}N^{-(\frac{7}{2}-\frac{4}{r})}\|\F_t[F_0]\|_{L^2_{\tau}}\sim L^{-\frac{1}{2}-\frac{1}{p}}N^{-\frac{1}{2}-\frac{1}{q}-\frac{3}{p}}\|F_0\|_{L^2_t}
\]
by (\ref{Aest}). 
\end{proof}

\begin{remark}
If $(p,q)$ is admissible, namely, $2/p+1/q=1/2$, then we have
\[
N^{\frac{2}{p}}\|h\|_{L^p_tL^q_x}
\lesssim N^{\frac{2}{p}+\frac{1}{2}-\frac{1}{q}}\|h\|_{L^p_tL^2_x}
\lesssim N^{\frac{2}{p}+\frac{1}{2}-\frac{1}{q}}\|h\|_{L^2_xL^p_t}
\lesssim N^{\frac{2}{p}+\frac{1}{2}-\frac{1}{q}-1-\frac{1}{2}-\frac{4}{p}}\|F_0\|_{L^2_t}
=N^{-\frac{3}{2}}\|F_0\|_{L^2_t}.
\]
\end{remark}
\begin{proof}[Proof of Theorem \ref{thm3-1}]
We prove (\ref{duamel_est_2}). 
Let $F^{\infty}$ is an extension of $F$ on $\R$ 
such that $\|F^{\infty}\|_{L^1_xL^2_t}\le 2\|F\|_{L^1_xL^2_T}$ 
and define $w_y^{\infty}$ by
\[
    \mathrm{w}_y^{\infty}(t,x):=
    \frac{1}{\sqrt{2\pi}}\int_0^t(\breve{P_N}\mathcal{K})(t-t',x-y)(P_NF^{\infty})(t',y)dt', 
\]
Because
\[
P_N\mathcal{I}[F^{\infty}](t,x)
=\int_{\R}\mathrm{w}_{y}^{\infty}(t,x)dy, 
\]
we have
\[
\begin{split}
\|P_N\mathcal{I}[F^{\infty}]\|_{X_N}
&\lesssim \int_{\R}\|w_y^{\infty}\|_{X_N}dy\\
&\lesssim \int_{\R}N^{-\frac{3}{2}}\|(P_NF)^{\infty}(t, y)\|_{L^2_t}dy\\
&=N^{-\frac{3}{2}}\|P_NF^{\infty}\|_{L^1_xL^2_t}
\end{split}
\]
by Lemma~\ref{w_y_est}. This implies (\ref{duamel_est_2}). 
\end{proof}
\begin{remark}\label{Tcondi}
The estimates in this section
also can be obtained if we replace $N\in 2^{\N}$ 
by $N\in 2^{\Z}$. 
The condition $T<1$ in this section 
comes from the maximal function estimate in 
Proposition~\ref{Mfe} 
and the low frequency estimate in Proposition~\ref{prop4-1-1-2}. 
Therefore, if the definition of $\|\cdot \|_{X_N}$ does not contain 
$\|\cdot \|_{L^2_xL^{\infty}_T}$ 
and we consider the homogeneous norm instead of 
$\|\cdot \|_{X^s}$ and $\|\cdot\|_{Y^s}$, then we can remove the conditon $T<1$. 
\end{remark}


\section{Multilinear estimates for general nonlinearities}
\label{multi}
In this section, we first give a proof of a bilinear Strichartz estimate (Theorem \ref{bilisol}) on the solution spaces given in Definition \ref{def4-1}. Then we show a multilinear estimate 
(Theorem \ref{sc_inv_multi_e}) from the solution space to the auxiliary space given in Definition \ref{def3-4}.
\subsection{Bilinear Strichartz estimates}
\begin{lemma}[Bilinear Strichartz estimates $L^2(\R)\times L^2(\R)\rightarrow L^2_{t,x}(\R\times\R)$]
\label{BS}
Let $N_1, N_2\in 2^{\N}$ with $N_1\gg N_2$. Then for any functions $f,g$ satisfying $P_{N_1}f, P_{N_2}g\in L^2(\R)$, the estimate
\begin{equation}
\label{5-1}
   \left\|\left(P_{N_1}e^{it\partial_x^4}f\right)\left(P_{N_2}e^{it\partial_x^4}g\right)\right\|_{L_{t,x}^2(\R\times\R)}\le CN_1^{-\frac{3}{2}}\left\|P_{N_1}f\right\|_{L^2}\left\|P_{N_2}g\right\|_{L^2}
\end{equation}
holds, where $C$ is a positive constant independent of $N_1,N_2, f,g$.
\end{lemma}
This lemma can be proved in the almost similar manner as the proof of Lemma 3.4 in \cite{CKSTT08} (see also Remark~\ref{BSE_rema} below).

The following bilinear Strichartz type estimate is useful to constract low regular solutions.
\begin{theorem}[Bilinear Strichartz estimate on $X_{N_1}\times X_{N_2}$]
\label{bilisol}
Let $N_1,N_2\in 2^{\Z}$ and $u_1\in X_{N_1}, u_2\in X_{N_2}$. 
If $N_1\gg N_2$, then the estimate
\begin{equation}\label{bistri}
     \|P_{N_1}u_1P_{N_2}u_2\|_{L^2_{t,x}}\lesssim
     N_1^{-\frac{3}{2}}\|P_{N_1}u_1\|_{X_{N_1}}\|P_{N_2}u_2\|_{X_{N_2}}
\end{equation}
holds, where the implicit constant is independnet of $N_1,N_2,u_1,u_2$.
\end{theorem}

\begin{proof}
We put $u_{j,N_j}:=P_{N_j}u_j$ and $F_j:=(i\partial_t+\partial_x^4)u_{j}$ for $j=1,2$.
It suffices to show that
\[
\|u_{1,N_1}u_{2,N_2}\|_{L^2_{t,x}}\lesssim 
N_1^{-\frac{3}{2}}\left(\|u_{1,N_1}(0)\|_{L^2_x}+\|F_1\|_{Y_{N_1}}\right)
\left(\|u_{2,N_2}(0)\|_{L^2_x}+\|F_2\|_{Y_{N_2}}\right).
\]
This follows from the following estimates.
\begin{align}
\|u_{1,N_1}u_{2,N_2}\|_{L^2_{t,x}}&\lesssim 
N_1^{-\frac{3}{2}}\left(\|u_{1,N_1}(0)\|_{L^2_x}+\|F_1\|_{L^1_tL^2_x}\right) 
\left(\|u_{2,N_2}(0)\|_{L^2_x}+\|F_2\|_{L^1_tL^2_x}\right),\label{bi_est_1}\\ 
\|u_{1,N_1}u_{2,N_2}\|_{L^2_{t,x}}&\lesssim 
N_1^{-\frac{3}{2}}\left(\|u_{1,N_1}(0)\|_{L^2_x}+\|F_1\|_{L^1_tL^2_x}\right) 
\left(\|u_{2,N_2}(0)\|_{L^2_x}+N_2^{-\frac{3}{2}}\|F_2\|_{L^1_xL^2_t}\right),\label{bi_est_2}\\ 
\|u_{1,N_1}u_{2,N_2}\|_{L^2_{t,x}}&\lesssim 
N_1^{-\frac{3}{2}}\left(\|u_{1,N_1}(0)\|_{L^2_x}+N_1^{-\frac{3}{2}}\|F_1\|_{L^1_xL^2_t}\right) 
\left(\|u_{2,N_2}(0)\|_{L^2_x}+\|F_2\|_{L^1_tL^2_x}\right),\label{bi_est_3}\\ 
\|u_{1,N_1}u_{2,N_2}\|_{L^2_{t,x}}&\lesssim 
N_1^{-\frac{3}{2}}\left(\|u_{1,N_1}(0)\|_{L^2_x}+N_1^{-\frac{3}{2}}\|F_1\|_{L^1_xL^2_t}\right) 
\left(\|u_{2,N_2}(0)\|_{L^2_x}+N_2^{-\frac{3}{2}}\|F_2\|_{L^1_xL^2_t}\right). \label{bi_est_4}
\end{align}
We prove only (\ref{bi_est_4}) because the other estimates can be obtained by the similar or simpler way. 
We note that
\[
u_{j,N_j}(t)=e^{it\partial_x^4}u_{j,N_j}(0)-i\int_0^te^{i(t-t')\partial_x^4}P_{N_j}F_j(t')dt'
=:A_j+B_j.
\]
To obtain (\ref{bi_est_4}), we prove the followings.
\begin{align}
\|A_1A_2\|_{L^2_{t,x}}&\lesssim N_1^{-\frac{3}{2}}\|u_{1,N_1}(0)\|_{L^2_x}\|u_{2,N_2}(0)\|_{L^2_x},
\label{AB_est_1}\\
\|A_1B_2\|_{L^2_{t,x}}&\lesssim N_1^{-\frac{3}{2}}N_2^{-\frac{3}{2}}\|u_{1,N_1}(0)\|_{L^2_x}\|F_2\|_{L^1_xL^2_t},
\label{AB_est_2}\\
\|B_1A_2\|_{L^2_{t,x}}&\lesssim N_1^{-3}\|F_1\|_{L^1_xL^2_t}\|u_{2,N_2}(0)\|_{L^2_x},
\label{AB_est_3}\\
\|B_1B_2\|_{L^2_{t,x}}&\lesssim N_1^{-3}N_2^{-\frac{3}{2}}\|F_1\|_{L^1_xL^2_t}\|F_2\|_{L^1_xL^2_t}.
\label{AB_est_4}
\end{align}
(\ref{AB_est_1}) is obtained by (\ref{5-1}). 

Now we prove (\ref{AB_est_2}) and (\ref{AB_est_3}).
By Proposition~\ref{duam_decom}, we have
\[
\begin{split}
B_j&=-\int_{\R}e^{it\partial_x^4}\mathcal{L}v_{j,y}(x)dy
+\int_{\R}(P_{<N_j/2^{50}}\ee_{(-\infty, 0]})(x)(P_{+}e^{it\partial_x^4}v_{j,y})(x)dy\\
&\ \ \ \ -\int_{\R}(P_{<N_j/2^{50}}\ee_{[0,\infty)})(x)(P_{-}e^{it\partial_x^4}v_{j,y})(x)dy
+\int_{\R}h_{j,y}(t,x)dy,
\end{split}
\]
where $v_{j,y}=\F_{\xi}^{-1}[\psi_N(\xi)\F_t[F_j(t,y)](\xi^4)]$, 
$\mathcal{L}v_{j,y}=\F_{\xi}^{-1}[\psi_N(\xi)\F_t[\ee_{(-\infty ,0]}(t)F_j(t,y)](\xi^4)]$ 
and $h_{j,y}$ satisfies
\begin{equation}\label{hy_est}
\|h_{j,y}\|_{L^q_xL^p_t}\lesssim N_j^{-1-\frac{1}{q}-\frac{4}{p}}\|F_j(t,y)\|_{L^2_t}. 
\end{equation}
We note that
\begin{equation}\label{vy_est}
\|v_{j,y}(x)\|_{L^2_x}\lesssim N_j^{-\frac{3}{2}}\|F_j(t,y)\|_{L^2_t},\ 
\|\mathcal{L}v_{j,y}(x)\|_{L^2_x}\lesssim N_j^{-\frac{3}{2}}\|F_j(t,y)\|_{L^2_t}.
\end{equation}
Furthermore, for any $g\in L^2(\R^2)$, it holds that
\begin{equation}\label{Pg_est}
\|(P_{<N_j/2^{50}}\ee_{(-\infty, 0]})(x)g(t,x)\|_{L^2_{t,x}}
\lesssim \|g\|_{L^2_{t,x}}. 
\end{equation}
Indeed, if $\chi_{N_j}$ is defined by $P_{<N_j/2^{50}}f=\chi_{N_j}*f$, then we have
\[
\begin{split}
\|(P_{<N_j/2^{50}}\ee_{(-\infty, 0]})(x)g(t,x)\|_{L^2_{t,x}}
&=\|(\chi_{N_j}*\ee_{(-\infty,0]})(x)g(t,x)\|_{L^2_{tx}}\\
&\lesssim \int_{\R}|\chi_{N_j}(z)|\|\ee_{(-\infty,0]}(x-z)g(t,x)\|_{L^2_{tx}}dz\\
&\lesssim \|g\|_{L^2_{t,x}}
\end{split}
\]
because $\chi_{N_j}(x)=\F^{-1}_{\xi}[\varphi (2^{50}N_j^{-1}\xi)](x)$, 
where $\varphi$ is defined in (\ref{2-2-5}). 
By the same way, we obtain 
\begin{equation}\label{Pg_est_mi}
\|(P_{<N_j/2^{50}}\ee_{[0,\infty)})(x)g(t,x)\|_{L^2_{t,x}}
\lesssim \|g\|_{L^2_{t,x}}. 
\end{equation}
Therefore, we have
\[
\begin{split}
\|A_1B_2\|_{L^2_{t,x}}
&\lesssim \int_{\R}\|e^{it\partial_x^4}u_{1,N_1}(0)e^{it\partial_x^4}\mathcal{L}v_{2,y}\|_{L^2_{t,x}}dy
+\int_{\R}\|e^{it\partial_x^4}u_{1,N_1}(0)e^{it\partial_x^4}P_{+}v_{2,y}\|_{L^2_{t,x}}dy\\
&\ \ \ \ 
+\int_{\R}\|e^{it\partial_x^4}u_{1,N_1}(0)e^{it\partial_x^4}P_{-}v_{2,y}\|_{L^2_{t,x}}dy
+\int_{\R}\|e^{it\partial_x^4}u_{1,N_1}(0)h_{2,y}\|_{L^2_{t,x}}dy\\
&=:I+II+III+IV.
\end{split}
\]
By (\ref{5-1}) and (\ref{vy_est}), we obtain
\[
\begin{split}
I+II+III
&\lesssim \int_{\R}N_1^{-\frac{3}{2}}\|u_{1,N_1}(0)\|_{L^2_{x}}\left(\|\mathcal{L}v_{2,y}\|_{L^2_x}+\|v_{2,y}\|_{L^2_x}\right)dy\\
&\lesssim \int_{\R}N_1^{-\frac{3}{2}}N_2^{-\frac{3}{2}}\|u_{1,N_1}(0)\|_{L^2_{x}}\|F_2(t,y)\|_{L^2_t}dy\\
&=N_1^{-\frac{3}{2}}N_2^{-\frac{3}{2}}\|u_{1,N_1}(0)\|_{L^2_{x}}\|F_2\|_{L^1_xL^2_t}. 
\end{split}
\]
While, by the H\"older inequality, Lemma~\ref{kato}, and (\ref{hy_est}) with 
$(q,p)=(2,\infty)$, we obtain
\[
\begin{split}
IV&\lesssim \int_{\R}\|e^{it\partial_x^4}u_{1,N_1}(0)\|_{L^{\infty}_xL^2_t}\|h_{2,y}\|_{L^{2}_xL^{\infty}_t}dy\\
&\lesssim \int_{\R}N_1^{-\frac{3}{2}}\|u_{1,N_1}(0)\|_{L^2_x}N_2^{-\frac{3}{2}}\|F(t,y)\|_{L^2_t}dy\\
&=N_1^{-\frac{3}{2}}N_2^{-\frac{3}{2}}\|u_{1,N_1}(0)\|_{L^2_{x}}\|F_2\|_{L^1_xL^2_t}.
\end{split}
\]
Therefore, we get (\ref{AB_est_2}). By the same way, we also get (\ref{AB_est_3}).

Finally, we prove (\ref{AB_est_4}). 
By the same argument as above, $\|B_1B_2\|_{L^2_{tx}}$ is controlled by 
the summation of 
\[
\begin{split}
&\iint_{\R^2}\|e^{it\partial_x^4}\widetilde{v}_{1,y_1}e^{it\partial_x^4}\widetilde{v}_{2,y_2}\|_{L^2_{t,x}}dy_1dy_2,\ 
\iint_{\R^2}\|e^{it\partial_x^4}\widetilde{v}_{1,y_1}h_{2,y_2}\|_{L^2_{t,x}}dy_1dy_2,\\ 
&\iint_{\R^2}\|h_{1,y_1}e^{it\partial_x^4}\widetilde{v}_{2,y_2}\|_{L^2_{t,x}}dy_1dy_2,\ 
\iint_{\R^2}\|h_{1,y_1}h_{2,y_2}\|_{L^2_{t,x}}dy_1dy_2,
\end{split}
\]
where $\widetilde{v}_{j,y}\in \{\mathcal{L}{v_{j,y}}, P_+v_{j,y}, P_-v_{j,y}\}$. 
By (\ref{5-1}) and (\ref{vy_est}), we obtain
\[
\begin{split}
\iint_{\R^2}\|e^{it\partial_x^4}\widetilde{v}_{1,y_1}e^{it\partial_x^4}\widetilde{v}_{2,y_2}\|_{L^2_{t,x}}dy_1dy_2
&\lesssim \iint_{\R^2}N_1^{-\frac{3}{2}}\|\widetilde{v}_{1,y_1}\|_{L^2_{x}}\|\widetilde{v}_{2,y_2}\|_{L^2_{x}}dy_1dy_2\\
&\lesssim \iint_{\R^2}N_1^{-\frac{3}{2}}N_1^{-\frac{3}{2}}\|F_1(t,y_1)\|_{L^2_t}N_2^{-\frac{3}{2}}\|F_2(t,y_2)\|_{L^2_t}dy_1dy_2\\
&=N_1^{-3}N_2^{-\frac{3}{2}}\|F_1\|_{L^1_xL^2_t}\|F_2\|_{L^1_xL^2_t}. 
\end{split}
\]
By the H\"older inequality, Lemma~\ref{kato}, (\ref{vy_est}), and (\ref{hy_est}) with 
$(q,p)=(2,\infty)$, we obtain
\[
\begin{split}
\iint_{\R^2}\|e^{it\partial_x^4}\widetilde{v}_{1,y_1}h_{2,y_2}\|_{L^2_{t,x}}dy_1dy_2
&\lesssim \iint_{\R^2}\|e^{it\partial_x^4}\widetilde{v}_{1,y_1}\|_{L^{\infty}_{x}L^2_t}
\|h_{2,y_2}\|_{L^2_{x}L^{\infty}_t}dy_1dy_2\\
&\lesssim \iint_{\R^2}N_1^{-\frac{3}{2}}\|\widetilde{v}_{1,y_1}\|_{L^2_x}N_2^{-\frac{3}{2}}\|F_2(t,y_2)\|_{L^2_t}dy_1dy_2\\
&\lesssim \iint_{\R^2}N_1^{-3}\|F_1(t,y_1)\|_{L^2_t}N_2^{-\frac{3}{2}}\|F_2(t,y_2)\|_{L^2_t}dy_1dy_2\\
&=N_1^{-3}N_2^{-\frac{3}{2}}\|F_1\|_{L^1_xL^2_t}\|F_2\|_{L^1_xL^2_t}. 
\end{split}
\]
By the H\"older inequality and (\ref{hy_est}) with 
$(q,p)=(\infty, 2)$, $(2,\infty)$, we obtain
\[
\begin{split}
\iint_{\R^2}\|h_{1,y_1}h_{2,y_2}\|_{L^2_{t,x}}dy_1dy_2
&\lesssim \iint_{\R^2}\|h_{1,y_1}\|_{L^{\infty}_{x}L^2_t}
\|h_{2,y_2}\|_{L^2_{x}L^{\infty}_t}dy_1dy_2\\
&\lesssim \iint_{\R^2}N_1^{-3}\|F_1(t,y_1)\|_{L^2_t}N_2^{-\frac{3}{2}}\|F_2(t,y_2)\|_{L^2_t}dy_1dy_2\\
&=N_1^{-3}N_2^{-\frac{3}{2}}\|F_1\|_{L^1_xL^2_t}\|F_2\|_{L^1_xL^2_t}. 
\end{split}
\]
Therefore, we get (\ref{AB_est_4}). 
\end{proof}
\begin{corollary}\label{bilin_T_cor1}
Let $T>0$, $N_1,N_2\in 2^{\Z}$ and $u_1\in X_{N_1}(T), u_2\in X_{N_2}(T)$. 
If $N_1\gg N_2$, then the estimate
\begin{equation}\label{bistriT}
     \|P_{N_1}u_1P_{N_2}u_2\|_{L^2_TL^2_x}\lesssim
     T^{\frac{\theta}{4}}N_1^{-\frac{3}{2}+\theta}\|P_{N_1}u_1\|_{X_{N_1}(T)}\|P_{N_2}u_2\|_{X_{N_2}(T)}
\end{equation}
holds for any $\theta \in [0,1]$, where the implicit constant is independent of $T,N_1,N_2,u_1,u_2$.
\end{corollary}
\begin{proof}
By the H\"older inequality and the definition of $X_N$-norm, we have
\[
\|P_{N_1}u_1P_{N_2}u_2\|_{L^2_TL^2_x}
\le \|\ee_{(-T,T)}(t)\|_{L^4_t}\|P_{N_1}u_1\|_{L^4_tL^{\infty}_x}\|P_{N_2}u_2\|_{L^{\infty}_tL^2_x}
\lesssim T^{\frac{1}{4}}N_1^{-\frac{1}{2}}\|P_{N_1}u_1\|_{X_{N_1}}\|P_{N_2}u_2\|_{X_{N_2}}. 
\]
Therefore, by the interpolation between this estimate and (\ref{bistri}), 
we obtain (\ref{bistriT}). 
\end{proof}
\begin{remark}
The estimates (\ref{5-1}), (\ref{bistri}), and (\ref{bistriT}) 
also holds if we replace $P_{N_2}$ by $P_{\le 1}$. 
\end{remark}
\subsection{Multilinear estimate}
\begin{theorem}[Multilinear estimate]\label{sc_inv_multi_e}
Let $m\ge 3$, $\gamma \in \{1,2,3\}$, $0<T<1$. Set
\[
     s_0=s_0(\gamma, m):=
     \left\{\begin{array}{ll}
	\frac{\gamma -1}{2},&m=3,\\
	\frac{\gamma}{3}-\frac{1}{2},&m= 4,\\
	s_c+\epsilon,&m\ge 5
	\end{array}\right.
\]
for $\gamma \in \{1,2\}$ and
\[
     s_0=s_0(3, m):=
     \left\{\begin{array}{ll}
	1,&m=3,\\
	\frac{1}{2},&m\ge 4, 
	\end{array}\right.
\]
where $\epsilon >0$ is an arbitrary positive number. 
If $s\ge \max\{s_0,0\}$,
then for any $u_1$, $\cdots$, $u_m\in X^s(T)$ 
and the multi-index $\alpha =(\alpha_1,\cdots,\alpha_m)\in (\N\cup \{0\})^m$ with $|\alpha|=\gamma$, 
it holds that
\begin{equation}\label{multimultimulti_est}
\left\|
\prod_{i=1}^{m}\partial_x^{\alpha_i}\widetilde{u_i}\right\|_{Y^s(T)}
\lesssim T^{\delta}\prod_{i=1}^m\|u_i\|_{X^s(T)} 
\end{equation}
for some $\delta >0$ if $\gamma \in \{1,2\}$ and $\delta =0$ if $\gamma =3$, 
where $\widetilde{u_i}\in \{u_i,\overline{u_i}\}$. 
The implicit constant depends only on $m,s$, $s_0$.
\end{theorem}
This multilinear estimate will be used to prove Theorem \ref{lwp1} and 
\ref{corlwp} (see the proof of Theorem \ref{lwpcomp} and Remark~\ref{lwpremark1}).

\begin{proof}
Let $m\ge 3$, $s\ge s_0$, $0<T<1$, and $u_i\in X^s(T)$ $(i=1,\cdots,m)$. 
We write $P_{\le 1}=P_1$. 
For $i=1,\cdots, m$ and $N_i\in 2^{\N\cup \{0\}}$, we set $c_{i,N_i}:=N_i^s\|P_{N_i}u_i\|_{X_{N_i}(T)}$. 
We define
\[
\begin{split}
I_0&:=\left\{\sum_{N\in 2^{\N\cup \{0\}}} N^{2s}\left\|P_N\left(
   \prod_{i=1}^{m}P_{1}\widetilde{u_i}\right)\right\|_{Y_N(T)}^2\right\}^{\frac{1}{2}},\\
I_k&:=\left\{\sum_{N\in 2^{\N\cup \{0\}}} N^{2s}\left\|P_N\left(\sum_{\substack{(N_1,\cdots , N_m)\in \Phi_k\\ N_1\gg 1}}N_1^\gamma
   \prod_{i=1}^{m}P_{N_i}\widetilde{u_i}\right)\right\|_{Y_N(T)}^2\right\}^{\frac{1}{2}}\ \ (k=1,\cdots, m),
\end{split}
\]
where
\[
\Phi_k:=\{(N_1,\cdots,N_m)|\ N_1\ge \cdots \ge N_m,\ N_1\sim \cdots \sim N_k\gg N_{k+1}\} 
\]
and $N_{m+1}:=1$. 
We will show 
\begin{equation}
    I_k\lesssim T^{\delta}\prod_{i=1}^{m}\left(\sum_{N_i}c_{i,N_i}^2\right)^{\frac{1}{2}}
\end{equation}
for some $\delta \ge 0$. 
For $k=0$, by the H\"older inequality and the Bernstein inequality, we have
\[
I_0\lesssim \left\|\prod_{i=1}^mP_1\widetilde{u_i}\right\|_{Y_N}
\lesssim \left\|\prod_{i=1}^mP_1u_i\right\|_{L^1_TL^2_x}
\le T\|P_1u_1\|_{L^{\infty}_tL^2_x}\prod_{i=2}^m\|P_{1}u_i\|_{L^{\infty}_{t,x}}
\lesssim T\prod_{i=1}^mc_{i,1}.
\]
Therefore, it suffices to show
\begin{equation}\label{des_est_k1}
N^sN_1^{\gamma}
\left\|P_N\left(\prod_{i=1}^mP_{N_i}\widetilde{u_i}\right)\right\|_{Y_N(T)}
\lesssim T^{\delta}N_{2}^{-\epsilon}\left(\prod_{i=1}^mc_{i,N_i}\right)
\end{equation}
when $(N_1,\cdots,N_m)\in \Phi_1$ and 
\begin{equation}\label{des_est}
\sum_{N\lesssim N_1}N^sN_1^{\gamma}
\left\|P_N\left(\prod_{i=1}^mP_{N_i}\widetilde{u_i}\right)\right\|_{Y_N(T)}
\lesssim T^{\delta}N_{k+1}^{-\epsilon}\left(\prod_{i=1}^mc_{i,N_i}\right)
\end{equation}
when $(N_1,\cdots,N_m)\in \Phi_k$ with $k=2,\cdots ,m$.
for some $\epsilon >0$. 
Indeed, if (\ref{des_est_k1}) and (\ref{des_est}) holds, then by the Cauchy-Schwarz inequality 
for dyadic summation, we have
\[
\begin{split}
I_1^2&\lesssim \sum_N
\left(\sum_{\substack{(N_1,\cdots , N_m)\in \Phi_k\\ N_1\gg 1}}N^sN_1^\gamma
\left\|P_N\left(\prod_{i=1}^mP_{N_i}\widetilde{u_i}\right)\right\|_{Y_N(T)}\right)^2\\
&\lesssim \sum_N\left(\sum_{N_1\sim N}
\sum_{N_2\ge \cdots \ge N_m}T^{\delta}
N_{2}^{-\epsilon}\left(\prod_{i=1}^mc_{i,N_i}\right)\right)^2\\
&\lesssim T^{2\delta}\sum_N\left(\sum_{N_1\sim N}c_{1,N_1}\right)^2
\prod_{i=2}^m\left\{\left(\sum_{N_i\ge 1}N_i^{-\frac{2\epsilon}{m-1}}\right)
\left(\sum_{N_i}c_{i,N_i}^2\right)\right\}\\
&\lesssim T^{2\delta}\prod_{i=1}^{m}\left(\sum_{N_i}c_{i,N_i}^2\right)
\end{split}
\]
and
\[
\begin{split}
I_k&\lesssim \sum_{\substack{(N_1,\cdots , N_m)\in \Phi_k\\ N_1\gg 1}}\sum_{N\lesssim N_1}
N^sN_1^\gamma\left\|P_N\left(\prod_{i=1}^mP_{N_i}\widetilde{u_i}\right)\right\|_{Y_N(T)}\\
&\lesssim \sum_{(N_1,\cdots ,N_m)\in \Phi_k}
T^{\delta}N_{k+1}^{-\epsilon}\left(\prod_{i=1}^mc_{i,N_i}\right)\\
&\le T^{\delta}\left(\sum_{N_1\sim \cdots \sim N_k}\prod_{i=1}^kc_{i,N_i}\right)
\prod_{i=k+1}^m\left\{\left(\sum_{N_i\ge 1}N_i^{-\frac{2\epsilon}{m-k}}\right)^{\frac{1}{2}}
\left(\sum_{N_i}c_{i,N_i}^2\right)^{\frac{1}{2}}\right\}\\
&\lesssim T^{\delta}\prod_{i=1}^{m}\left(\sum_{N_i}c_{i,N_i}^2\right)^{\frac{1}{2}}
\end{split}
\]
for $k=2,\cdots,m$.

Now, we prove (\ref{des_est_k1}) and (\ref{des_est}). 
\\
\underline{{\bf Case 1}: $m=3$}. \\

\noindent \underline{(i)\ For $k=1$\ $(N\sim N_1\gg N_2)$}. 

By the definition of the norm $\|\cdot\|_{Y_N(T)}$, the H\"older inequality and (\ref{bistriT}), we have
\[
\begin{split}
N^sN_1^{\gamma}
\left\|P_N\left(\prod_{i=1}^3P_{N_i}\widetilde{u_i}\right)\right\|_{Y_N(T)}
&\le
N^{s-\frac{3}{2}}N_1^\gamma\left\|\prod_{i=1}^3P_{N_i}u_i\right\|_{L^1_xL^2_T}\\
&\lesssim N_1^{s-\frac{3}{2}+\gamma}
\|P_{N_1}u_1P_{N_2}u_2\|_{L^2_{x,T}}\|P_{N_3}u_3\|_{L^2_xL^{\infty}_T}\\
&\lesssim T^{\frac{\theta}{4}}N_1^{s-3+\gamma+\theta} 
N_3^{1+\epsilon}\prod_{i=1}^3\|P_{N_i}u_i\|_{X_{N_i}(T)}\\
&\sim T^{\frac{\theta}{4}}N_1^{-(3-\gamma-\theta)}N_2^{-s}N_3^{-s+1+\epsilon}\prod_{i=1}^3c_{i,N_i}\\
&\lesssim T^{\frac{\theta}{4}}N_2^{-s-\frac{3-\gamma-\theta}{2}+\frac{1+\epsilon}{2}}
N_3^{-s-\frac{3-\gamma-\theta}{2}+\frac{1+\epsilon}{2}}\prod_{i=1}^3c_{i,N_i}\\
&\lesssim T^{\frac{\theta}{4}}N_2^{-\epsilon}\prod_{i=1}^3c_{i,N_i}
\end{split}
\]
for $0\le \theta \le \min\{1,3-\gamma\}$, $\epsilon >0$, and 
$s\ge \frac{\gamma -2+\theta +3\epsilon}{2}$. 
We choose $\theta$ and $\epsilon$ as $3\epsilon +\theta \le 1$. 
In particular, we have to choose $\theta =0$ when $\gamma =3$.\\

\noindent \underline{(ii)\ For $k=2$\ $(N_1\sim N_2\gg N_3)$} 

By the definition of the norm $\|\cdot\|_{Y_N(T)}$, the H\"older inequality and (\ref{bistriT}), we have
\[
\begin{split}
\sum_{N\lesssim N_1}N^sN_1^{\gamma}
\left\|P_N\left(\prod_{i=1}^3P_{N_i}\widetilde{u_i}\right)\right\|_{Y_N(T)}
&\le N_1^{s+\gamma}\left\|\prod_{i=1}^3P_{N_i}u_i\right\|_{L^1_TL^2_x}\\
&\le 
N_1^{s+\gamma}T^{\frac{1}{4}}\|P_{N_1}u_1\|_{L^4_{T}L^{\infty}_x}\|P_{N_2}u_2P_{N_3}u_3\|_{L^2_{T,x}}\\
&\lesssim T^{\frac{1}{4}}
N_1^{s+\gamma -\frac{1}{2}}N_2^{-\frac{3}{2}}\prod_{i=1}^3\|P_{N_i}u_i\|_{X_{N_i}(T)}\\
&\sim T^{\frac{1}{4}}N_1^{-(s-\gamma +2)}N_3^{-s}\prod_{i=1}^3c_{i,N_i}\\
&\lesssim T^{\frac{1}{4}}
N_1^{-(s-\frac{\gamma -1}{2})}N_3^{-s-\frac{3-\gamma}{2}}\prod_{i=1}^3c_{i,N_i}\\
&\lesssim T^{\frac{1}{4}}N_3^{-\epsilon}\prod_{i=1}^3c_{i,N_i}
\end{split}
\]
for $\epsilon >0$ and $s\ge \max\{\frac{\gamma-1}{2}, -\frac{3-\gamma}{2}+\epsilon, 0\}$. 
We note that $\frac{\gamma-1}{2}\ge -\frac{3-\gamma}{2}+\epsilon$ 
if we choose $\epsilon$ as $\epsilon \le 1$. \\

\noindent \underline{(iii)\ For $k=3$\ $(N_1\sim N_2\sim N_3\gg 1)$} 

By the definition of the norm $\|\cdot\|_{Y_N(T)}$, the H\"older inequality, we have
\[
\begin{split}
\sum_{N\lesssim N_1}N^sN_1^{\gamma}
\left\|P_N\left(\prod_{i=1}^3P_{N_i}\widetilde{u_i}\right)\right\|_{Y_N(T)}
&\le N_1^{s+\gamma}\left\|\prod_{i=1}^3P_{N_i}u_i\right\|_{L^1_TL^2_x}\\
&\le N_1^{s+\gamma}T^{\frac{1}{2}}
\|P_{N_1}u_1\|_{L^4_{T}L^{\infty}_x}\|P_{N_2}u_2\|_{L^4_{T}L^{\infty}_x}\|P_{N_3}u_3\|_{L^{\infty}_TL^2_x}\\
&\lesssim T^{\frac{1}{2}}N_1^{s+\gamma-\frac{1}{2}}
N_2^{-\frac{1}{2}}\prod_{i=1}^3\|P_{N_i}u_i\|_{X_{N_i}(T)}\\
&\sim T^{\frac{1}{2}}N_1^{-2s+\gamma -1}\prod_{i=1}^3c_{i,N_i}\\
&\lesssim T^{\frac{1}{2}}\prod_{i=1}^3c_{i,N_i}
\end{split}
\]
for $s\ge \max\{\frac{\gamma -1}{2}, 0\}$.\\

\noindent \underline{{\bf Case 2}: $m\ge 4$}. 

We only consider the case $s< \frac{1}{2}$ 
because the case $s\ge \frac{1}{2}$ is simpler. 
By the Bernstein inequality, we have
\begin{equation}\label{m555_est}
\prod_{i=5}^m\|P_{N_i}u_i\|_{L^{\infty}_{x,T}}
\lesssim \prod_{i=5}^mN_i^{\frac{1}{2}}\|P_{N_i}u_i\|_{L^{\infty}_{T}L^2_x}
\lesssim N_5^{(m-4)(\frac{1}{2}-s)}\prod_{i=5}^mc_{i,N_i} 
\end{equation}
for $m\ge 5$ and $s<\frac{1}{2}$. 
We assume $\prod_{i=5}^m\|P_{N_i}u_i\|_{L^{\infty}_{x,T}}=1$ 
and $\prod_{i=5}^mc_{i,N_i}=1$ if $m=4$. 
Then (\ref{m555_est}) is true for $m\ge 4$. \\

\noindent \underline{(i)\ For $k=1$\ $(N\sim N_1\gg N_2)$} 

By the definition of the norm $\|\cdot\|_{Y_N(T)}$, the H\"older inequality, (\ref{bistriT}), 
and (\ref{m555_est}) we have
\[
\begin{split}N^sN_1^{\gamma}
\left\|P_N\left(\prod_{i=1}^mP_{N_i}\widetilde{u_i}\right)\right\|_{Y_N(T)}
&\le N^{s-\frac{3}{2}}N_1^{\gamma}\left\|\prod_{i=1}^mP_{N_i}u_i\right\|_{L^1_xL^2_T}\\
&\le N_1^{s-\frac{3}{2}+\gamma}
\|P_{N_1}u_1P_{N_2}u_2\|_{L^2_{x,T}}\|P_{N_3}u_3\|_{L^4_xL^{\infty}_T}
\|P_{N_4}u_4\|_{L^4_xL^{\infty}_T}\prod_{i=5}^m\|P_{N_i}u_i\|_{L^{\infty}_{x,T}}\\
&\lesssim T^{\frac{\theta}{4}}
N_1^{s-3+\gamma +\theta}N_3^{\frac{1}{4}}N_4^{\frac{1}{4}}
\prod_{i=1}^4\|P_{N_i}u_i\|_{X_{N_i}(T)}\prod_{i=5}^m\|P_{N_i}u_i\|_{L^{\infty}_{x,T}}\\
&\lesssim T^{\frac{\theta}{4}}N_1^{-(3-\gamma -\theta)}N_2^{-s}N_3^{-s+\frac{1}{4}}N_4^{-s+\frac{1}{4}}
N_5^{(m-4)(\frac{1}{2}-s)}\prod_{i=1}^mc_{i,N_i}\\
&\lesssim T^{\frac{\theta}{4}}
\prod_{i=2}^4N_i^{-s+\frac{1}{6}-\frac{3-\gamma-\theta}{3}+\frac{m-4}{3}(\frac{1}{2}-s)}
\prod_{i=1}^mc_{i,N_i}\\
&\lesssim T^{\frac{\theta}{4}}N_2^{-\epsilon}\prod_{i=1}^mc_{i,N_i}
\end{split}
\]
for $0\le \theta \le \min\{1,3-\gamma\}$, $\epsilon >0$, 
and $s\ge s_c+\frac{\theta +3\epsilon}{m-1}$.
We choose $\theta$ and $\epsilon$ as $\theta +3\epsilon \le (m-1)(s-s_c)$ 
for $s>s_c$.
In particular, we have to choose $\theta =0$ when $\gamma =3$.\\

\noindent \underline{(ii)\ For $k=2$\ $(N_1\sim N_2\gg N_3)$} 

By the definition of the norm $\|\cdot\|_{Y_N(T)}$, the H\"older inequality, 
the Sobolev inequality, (\ref{bistriT}), and (\ref{m555_est}), we have
\[
\begin{split}
\sum_{N\lesssim N_1}N^sN_1^{\gamma}\left\|P_N\left(\prod_{i=1}^mP_{N_i}\widetilde{u_i}\right)\right\|_{Y_N(T)}
&\le N_1^{s}N_1^{\gamma}\left\|\prod_{i=1}^mP_{N_i}u_i\right\|_{L^1_TL^2_x}\\
&\le N_1^{s+\gamma}\|P_{N_1}u_1P_{N_3}u_3\|_{L^2_{T,x}}
\|P_{N_2}u_2P_{N_4}u_4\|_{L^2_T L^{\infty}_x}
\prod_{i=5}^m\|P_{N_i}u_i\|_{L^{\infty}_{T,x}}\\
&\le N_1^{s+\gamma}N_2^{\frac{1}{2}}\|P_{N_1}u_1P_{N_3}u_3\|_{L^2_{T,x}}
\|P_{N_2}u_2P_{N_4}u_4\|_{L^{2}_{T,x}}\prod_{i=5}^m\|P_{N_i}u_i\|_{L^{\infty}_{x,T}}\\
&\lesssim T^{\frac{\theta}{4}}
N_1^{s+\gamma-\frac{3}{2}+\theta}N_2^{-1}\prod_{i=1}^4\|P_{N_i}u_i\|_{X_{N_i}(T)}
\prod_{i=5}^m\|P_{N_i}u_i\|_{L^{\infty}_{x,T}}\\
&\lesssim T^{\frac{\theta}{4}}N_1^{-(s-\gamma+\frac{5}{2}-\theta)}N_3^{-s}N_4^{-s}
N_5^{(m-4)(\frac{1}{2}-s)}\prod_{i=1}^mc_{i,N_i}\\
&\lesssim T^{\frac{\theta}{4}}
\prod_{i=3}^4N_i^{-s-\frac{1}{2}(s-\gamma+\frac{5}{2}-\theta)+\frac{m-4}{2}(\frac{1}{2}-s)}
\prod_{i=1}^mc_{i,N_i}\\
&\lesssim T^{\frac{\theta}{4}}N_3^{-\epsilon}\prod_{i=1}^mc_{i,N_i}
\end{split}
\]
for $0\le \theta \le 1$, $\epsilon >0$, and 
$s\ge \max\{\gamma -\frac{5}{2}+\theta, s_c+\frac{\theta+2\epsilon}{m-1},0\}$.
We  choose $\theta$ and $\epsilon$ as $\theta +2\epsilon \le (m-1)(s-s_c)$ for $s>s_c$.\\

\noindent \underline{(iii)\ For $k=3$\ $(N_1\sim N_2\sim N_3\gg N_4)$} 

By the definition of the norm $\|\cdot\|_{Y_N(T)}$, the H\"older inequality, (\ref{bistriT}),  
and (\ref{m555_est}), we have
\[
\begin{split}
\sum_{N\lesssim N_1}N^sN_1^{\gamma}
\left\|P_N\left(\prod_{i=1}^mP_{N_i}\widetilde{u_i}\right)\right\|_{Y_N(T)}
&\le N_1^{s+\gamma}\left\|\prod_{i=1}^mP_{N_i}u_i\right\|_{L^1_TL^2_x}\\
&\le N_1^{s+\gamma}
\|P_{N_1}u_1\|_{L^4_{T}L^{\infty}_x}
\|P_{N_2}u_2\|_{L^4_{T}L^{\infty}_x}\|P_{N_3}u_3P_{N_4}u_4\|_{L^2_{T,x}}
\prod_{i=5}^m\|P_{N_i}u_i\|_{L^{\infty}_{x,T}}\\
&\lesssim T^{\frac{\theta}{4}}N_1^{s+\gamma-\frac{1}{2}}N_2^{-\frac{1}{2}}N_3^{-\frac{3}{2}+\theta}
\prod_{i=1}^4\|P_{N_i}u_i\|_{X_{N_i}(T)}\prod_{i=5}^m\|P_{N_i}u_i\|_{L^{\infty}_{x,T}}\\
&\lesssim T^{\frac{\theta}{4}}N_1^{-(2s-\gamma+\frac{5}{2}-\theta)}N_4^{-s}
N_5^{(m-4)(\frac{1}{2}-s)}\prod_{i=1}^mc_{i,N_i}\\
&\lesssim T^{\frac{\theta}{4}}
N_4^{-s-(2s-\gamma+\frac{5}{2}-\theta)+(m-4)(\frac{1}{2}-s)}\prod_{i=1}^mc_{i,N_i}\\
&\lesssim T^{\frac{\theta}{4}}N_4^{-\epsilon}\prod_{i=1}^mc_{i,N_i}
\end{split}
\]
for $0\le \theta \le 1$, $\epsilon >0$, and 
$s\ge \max\{\frac{1}{2}(\gamma -\frac{5}{2}+\theta ), s_c+\frac{\theta+\epsilon}{m-1},0\}$.
We choose $\theta$ and $\epsilon$ as $\theta +\epsilon \le (m-1)(s-s_c)$ 
for $s>s_c$.\\

\noindent \underline{(iv)\ For $k=4$\ $(N_1\sim N_2\sim N_3\sim N_4\gg N_5)$} 

If $m=4$, by the definition of the norm $\|\cdot\|_{Y_N(T)}$, H\"older inequality, we have
\[
\begin{split}
\sum_{N\lesssim N_1}N^sN_1^{\gamma}
\left\|P_N\left(\prod_{i=1}^mP_{N_i}\widetilde{u_i}\right)\right\|_{Y_N(T)}
&\le N_1^{s+\gamma}\left\|\prod_{i=1}^mP_{N_i}u_i\right\|_{L^1_TL^2_x}\\
&\le N_1^{s+\gamma}
T^{\frac{1}{4}}\|P_{N_1}u_1\|_{L^4_{T}L^{\infty}_x}\|P_{N_2}u_2\|_{L^4_{T}L^{\infty}_x}
\|P_{N_3}u_3\|_{L^{4}_TL^{\infty}_x}\|P_{N_4}u_4\|_{L^{\infty}_TL^2_x}\\
&\lesssim T^{\frac{1}{4}}N_1^{s+\gamma-\frac{1}{2}}
N_2^{-\frac{1}{2}}N_3^{-\frac{1}{2}}\prod_{i=1}^4\|P_{N_i}u_i\|_{X_{N_i}(T)}\\
&\sim T^{\frac{1}{4}}N_1^{-3s+\gamma -\frac{3}{2}}\prod_{i=1}^mc_{i,N_i}\\
&\lesssim T^{\frac{1}{4}}\prod_{i=1}^4c_{i,N_i}
\end{split}
\]
for $s\ge \max\{\frac{\gamma}{3}-\frac{1}{2},0\}$.

If $m\ge 5$, by the definition of the norm $\|\cdot\|_{Y_N(T)}$, H\"older inequality, we have
\[
\begin{split}
&\sum_{N\lesssim N_1}N^sN_1^{\gamma}
\left\|P_N\left(\prod_{i=1}^mP_{N_i}\widetilde{u_i}\right)\right\|_{Y_N(T)}\\
&\le N_1^{s+\gamma}\left\|\prod_{i=1}^mP_{N_i}u_i\right\|_{L^1_TL^2_x}\\
&\le T^{\frac{\theta}{4}}N_1^{s+\gamma}
\|P_{N_1}u_1\|_{L^{\frac{4}{1-\theta}}_TL^{\infty}_x}
\|P_{N_2}u_2\|_{L^4_{T}L^{\infty}_x}
\|P_{N_4}u_4\|_{L^4_{T}L^{\infty}_x}\|P_{N_4}u_4\|_{L^4_{T}L^{\infty}_x}
\|P_{N_5}u_5\|_{L^{\infty}_TL^{2}_x}
\prod_{i=6}^m\|P_{N_i}u_i\|_{L^{\infty}_{x,T}},
\end{split}
\]
where we assumed $\prod_{i=6}^m\|P_{N_i}u_i\|_{L^{\infty}_{x,T}}=1$ if $m=5$. 
We note that
\[
\prod_{i=6}^m\|P_{N_i}u_i\|_{L^{\infty}_{x,T}}
\lesssim N_5^{(m-5)(\frac{1}{2}-s)}\prod_{i=6}^mc_{i,N_i} 
\]
for $m\ge 6$. 
By the interpolation between
\[
\|P_{N_1}u_1\|_{L^4_TL^{\infty}_x}\lesssim N_1^{-\frac{1}{2}}\|P_{N_1}u_1\|_{X_{N_1}(T)}
\]
and
\[
\|P_{N_1}u_1\|_{L^{\infty}_TL^{2}_x}\lesssim \|P_{N_1}u_1\|_{X_{N_1}(T)},
\]
we have
\[
\|P_{N_1}u_1\|_{L^{\frac{4}{1-\theta}}_TL^{\frac{2}{\theta}}_x}
\lesssim N_1^{-\frac{1-\theta}{2}}\|P_{N_1}u_1\|_{X_{N_1}(T)}
\]
for $0\le \theta \le 1$. 
This and the Sobolev inequality imply
\[
\|P_{N_1}u_1\|_{L^{\frac{4}{1-\theta}}_TL^{\infty}_x}
\lesssim N_1^{-\frac{1}{2}+\theta}\|P_{N_1}u_1\|_{X_{N_1}(T)}. 
\]
Therefore, we obtain
\[
\begin{split}
\sum_{N\lesssim N_1}N^sN_1^{\gamma}
\left\|P_N\left(\prod_{i=1}^mP_{N_i}\widetilde{u_i}\right)\right\|_{Y_N(T)}
&\lesssim T^{\frac{\theta}{4}}N_1^{s+\gamma-\frac{1}{2}+\theta}
N_2^{-\frac{1}{2}}N_3^{-\frac{1}{2}}N_4^{-\frac{1}{2}}
\prod_{i=1}^5\|P_{N_i}u_i\|_{X_{N_i}(T)}
\prod_{i=6}^m\|P_{N_i}u_i\|_{L^{\infty}_{x,T}}\\
&\lesssim T^{\frac{\theta}{4}}N_1^{-3s+\gamma -2+\theta}N_5^{-s+(m-5)(\frac{1}{2}-s)}\prod_{i=1}^mc_{i,N_i}\\
&\sim 
T^{\frac{\theta}{4}}N_5^{-(m-1)s+\frac{m-1}{2}-(4-\gamma )+\theta}
\prod_{i=1}^mc_{i,N_i}\\
&\lesssim T^{\frac{\theta}{4}}N_5^{-\epsilon}\prod_{i=1}^4c_{i,N_i}
\end{split}
\]
for $0\le \theta \le 1$, $\epsilon >0$, and 
$s\ge \max\{\frac{\gamma -2+\theta}{3}, s_c+\frac{\theta+\epsilon}{m-1},0\}$.
We choose $\theta$ and $\epsilon$ as $\theta +\epsilon \le (m-1)(s-s_c)$ 
for $s>s_c$.
\end{proof}
\begin{remark}
For $m=3$, 
if we do not use the bilinear Strichartz estimate, 
then the worst case is not $k=3$. 
Indeed, for $k=1$ ($N\sim N_1$), we have
\[
\begin{split}
N^{s-\frac{3}{2}}N_1^{3}
\left\|\prod_{i=1}^3P_{N_i}u_i\right\|_{L^1_xL^2_T}
&\le N^{s-\frac{3}{2}}N_1^{3}
\|P_{N_1}u_1\|_{L^{\infty}_xL^2_T}
\|P_{N_2}u_2\|_{L^{2}_xL^{\infty}_T}
\|P_{N_2}u_2\|_{L^{2}_xL^{\infty}_T}\\
&\lesssim N^{s-\frac{3}{2}}N_1^{-s+\frac{3}{2}}
N_2^{-s+1+\epsilon}N_3^{-s+1+\epsilon}
\prod_{i=1}^3c_{i,N_i}\\
&\lesssim N_2^{-s+1+\epsilon}N_3^{-s+1+\epsilon}
\prod_{i=1}^3c_{i,N_i}
\end{split}
\]
when $\gamma =3$. 
This guarantees the trilinear estimate 
only for $s>1$. 
Therefore, by using the bilinear Strichartz estimate,
we can improve the result in \cite{HJ11}.  
\end{remark}
\begin{remark}
When $\gamma =3$, we cannot obtain
(\ref{multimultimulti_est}) with $\delta >0$. 
This is the reason why 
large data cannot be treated in 
Theorem ~\ref{cor2-4} and ~\ref{lwp1}.  
\end{remark}
\begin{remark}
We can also obtain
\begin{equation}\label{multilin_noderiv}
\left\|\prod_{i=1}^m\widetilde{u_i}\right\|_{Y^s(T)}
\lesssim T^{\delta}\prod_{i=1}^m\|u_i\|_{X^s(T)}
\end{equation}
for the same $s$ and $\delta$ as in Theorem~\ref{sc_inv_multi_e} 
because
\[
\left\|\prod_{i=1}^m\widetilde{u_i}\right\|_{Y^s(T)}
\lesssim \left\|\prod_{i=1}^mP_{\le 1}\widetilde{u_i}\right\|_{Y^s(T)}
+\sum_{k=1}^m\left\|\left(\partial_x^{\gamma}P_{>1}\widetilde{u_k}\right)\prod_{\substack{1\le i\le m\\ i\ne k}}\widetilde{u_i}\right\|_{Y^s(T)}. 
\]
We can treat the first term of R.H.S 
by the same way as the estimate for $I_0$ 
and the second term of R.H.S by the same way 
as the estimates for $I_k$ $(k=1,\cdots,m)$. 
\end{remark}
\section{Multilinear estimates at the scaling critical regularity in the cases $m\ge 6$ with $\gamma=3$ and $m\ge 5$ with $\gamma=2$}
\label{multie}
\ \ In this section, we prove multilinear estimates at the scaling critical regularity $s=s_c(\gamma,m)$ in the cases $m\ge 6$ with $\gamma=3$ and $m\ge 5$ with $\gamma=2$. To treat these cases, we define the function spaces $X_N$ and $Z_N$ equipped with the norms
\begin{equation}
\label{fsea}
\begin{split}
\|u\|_{X_N}&:=
    \|u\|_{L_t^{\infty}L_x^2}+N^{-\frac{1}{4}}\|u\|_{L^{4}_xL_t^{\infty}}+N^{\frac{3}{2}}\|u\|_{L_x^{\infty}L_t^2},\\
\|F\|_{Z_N}&:=N^{-\frac{3}{2}}\|F\|_{L^1_xL_t^2},
\end{split}
\end{equation}
instead of Definition~\ref{def4-1}. Furthermore, we define $\dot{X}^s$, $X^s$, $\dot{Z}^s$, and $Z^s$ by
\[
\begin{split}
\|u\|_{\dot{X}^s}
&:=\left(\sum_{N\in 2^{\Z}}N^{2s}\|P_Nu\|_{X_N}^2\right)^{\frac{1}{2}},\ \ 
\|u\|_{X^s}:=\|u\|_{\dot{X}^0}+\|u\|_{\dot{X}^s},\\
\|F\|_{\dot{Z}^s}
&:=\left(\sum_{N\in 2^{\Z}}N^{2s}\|P_NF\|_{Z_N}^2\right)^{\frac{1}{2}},\ \ 
\|F\|_{Z^s}:=\|F\|_{\dot{Z}^0}+\|F\|_{\dot{Z}^s}. 
\end{split}
\]
We can see that
\begin{equation}\label{XN_Linest}
\left\|e^{it\partial_x^4}u_0\right\|_{X^s}\lesssim \|u_0\|_{H^s},\ \ 
\left\|\int_0^te^{i(t-t')\partial_x^4}F(t')dt'\right\|_{X^s}\lesssim \|F\|_{Z^s}
\end{equation}
by the same argument as in the previous sections. 
\begin{theorem}[Multilinear estimates]\label{multi_est_3m}
Let $m\ge 6$. Set
\[
     s_c=s_c(m):=
     \frac{1}{2}-\frac{1}{m-1}.
\]
{\rm (i)}\ For any $u_1$, $\cdots$, $u_m\in \dot{X}^{s_c}$, 
it holds that
\begin{equation}\label{multilin_mhigh_hom}
\left\|\partial_x^{3}
\prod_{i=1}^{m}\widetilde{u_i}\right\|_{\dot{Z}^{s_c}}
\lesssim \prod_{i=1}^m\|u_i\|_{\dot{X}^{s_c}},
\end{equation}
where $\widetilde{u_i}\in \{u_i,\overline{u_i}\}$. The implicit constant depends only on $m$.\\
{\rm (ii)}\ If $s\ge s_c$,
For any $u_1$, $\cdots$, $u_m\in X^s$, 
it holds that
\begin{equation}\label{multilin_mhigh_inhom}
\left\|\partial_x^{3}
\prod_{i=1}^{m}\widetilde{u_i}\right\|_{Z^s}
\lesssim \prod_{i=1}^m\|u_i\|_{X^{s}},
\end{equation}
where $\widetilde{u_i}\in \{u_i,\overline{u_i}\}$. The implicit constant depends only on $m$ and $s$.
\end{theorem}

\begin{proof}
Let $s\ge 0$. We assume $u_i\in \dot{X}^s\cap \dot{X}^{s_c}$ and 
$\|u_i\|_{\dot{X}^{s_c}}\lesssim 1$ ($i=1,\cdots, m$).
We first prove
\begin{equation}\label{multilin_mhigh_2}
\left\|\partial_x^{3}
\prod_{i=1}^{m}\widetilde{u_i}\right\|_{\dot{Z}^{s}}
\lesssim \sum_{i=1}^m\|u_i\|_{\dot{X}^s}.
\end{equation}
This implies 
\begin{equation}\label{multilin_mhigh_3}
\left\|\partial_x^{3}
\prod_{i=1}^{m}\widetilde{u_i}\right\|_{\dot{Z}^{s}}
\lesssim \sum_{i=1}^m\|u_i\|_{\dot{X}^s}\prod_{\substack{1\le k\le m\\ k\ne i}}\|u_k\|_{\dot{X}^{s_c}}
\end{equation}
for any $u_i\in \dot{X}^s\cap \dot{X}^{s_c}$ 
because $\|u_i\|_{\dot{X}^{s_c}}=\|\overline{u_i}\|_{\dot{X}^{s_c}}$. 
For $i=1,\cdots, m$ and $N_i\in 2^{\Z}$, we set
$c_{1,N_1}:=N_1^s\|P_{N_1}u_1\|_{X_{N_1}}$,\ 
$c_{i,N_i}:=N_i^{s_c}\|P_{N_i}u_i\|_{X_{N_i}}$ $(i=2,\cdots ,m)$.  
We define
\[
\begin{split}
I_1&=\left\{\sum_{N\in 2^{\Z}} N^{2s+6}\left(\sum_{\substack{N_1\ge \cdots\ge N_m\\ N_1\sim N}}\left\|
   \prod_{i=1}^{m}P_{N_i}u_i\right\|_{Z_N}\right)^2\right\}^{\frac{1}{2}},\\
I_2&=\sum_{N\in 2^{\Z}} N^{s+3}\sum_{\substack{N_1\ge \cdots\ge N_m\\ N_1\gg N}}\left\|
   \prod_{i=1}^{m}P_{N_i}u_i\right\|_{Z_N}, 
\end{split}
\]
Then, we have
\[
\left\|\partial_x^{3}
\prod_{i=1}^{m}\widetilde{u_i}\right\|_{\dot{Z}^{s}}\lesssim I_1+I_2.
\]
It suffices to show that
\begin{equation}\label{m5_est}
\left\|\prod_{i=1}^mP_{N_i}u_i\right\|_{L^1_xL^2_t}
\lesssim N_1^{-s-\frac{3}{2}}
\left(\frac{\prod_{i=6}^mN_i^{\frac{1}{2}-s_c}}{\prod_{i=2}^5N_i^{s_c-\frac{1}{4}}}\right)
\prod_{i=1}^mc_{i,N_i}.
\end{equation}
Indeed, if (\ref{m5_est}) holds, then we have
\[
\begin{split}
I_1^2&\lesssim \sum_{N}N^{2s+6}
\left(
\sum_{N_1\sim N}\sum_{N_2\ge \cdots \ge N_m}
\left\|\prod_{i=1}^mP_{N_i}u_i\right\|_{Z_N}
\right)^2\\
&\lesssim 
\sum_{N}\left(
\sum_{N_1\sim N}\sum_{N_2\ge \cdots \ge N_m}
N^{s+\frac{3}{2}}\left\|\prod_{i=1}^mP_{N_i}u_i\right\|_{L^1_xL^2_t}
\right)^2\\
&\lesssim \sum_{N}
\left(
\sum_{N_1\sim N}\sum_{N_2\ge \cdots \ge N_m}
N^{s+\frac{3}{2}}N_1^{-s-\frac{3}{2}}
\left(
\frac{\prod_{i=6}^mN_i^{\frac{1}{2}-s_c}}{\prod_{i=2}^5N_i^{s_c-\frac{1}{4}}}
\right)
\prod_{i=1}^mc_{i,N_i}
\right)^2\\
&\lesssim \sum_{N}
\left(
\sum_{N_1\sim N}c_{1,N_1}\sum_{N_2\ge \cdots \ge N_m}
\left(
\frac{\prod_{i=6}^mN_i^{\frac{1}{2}-s_c}}{\prod_{i=2}^5N_i^{s_c-\frac{1}{4}}}
\right)
\left(
c_{2,N_2}^2+\prod_{i=3}^mc_{i,N_i}^2
\right)
\right)^2\\
&\lesssim \sum_{N}
\left(
\sum_{N_1\sim N}c_{1,N_1}
\left(
\sum_{N_2}c_{2,N_2}^2+\prod_{i=3}^m\sum_{N_i}c_{i,N_i}^2
\right)
\right)^2\\
&\lesssim 
\left(\sum_{N_1}c_{1,N_1}^2\right)
\left(\sum_{N_2}c_{2,N_2}^2+\prod_{i=3}^m\sum_{N_i}c_{i,N_i}^2\right)^2
\end{split}
\]
by the Young inequality and
\[
\begin{split}
I_2&\le \sum_{N_1}\sum_{N_2\sim N_1}\sum_{N\ll N_1}\sum_{N_3\ge \cdots \ge N_m}
N^{s+3}\left\|\prod_{i=1}^mP_{N_i}u_i\right\|_{Z_N}\\
&\le \sum_{N_1}\sum_{N_2\sim N_1}\sum_{N\ll N_1}\sum_{N_3\ge \cdots \ge N_m}
N^{s+\frac{3}{2}}\left\|\prod_{i=1}^mP_{N_i}u_i\right\|_{L^1_xL^2_t}\\
&\lesssim \sum_{N_1}\sum_{N_2\sim N_1}\sum_{N\ll N_1}\left(\frac{N}{N_1}\right)^{s+\frac{3}{2}}c_{1,N_1}c_{2,N_2}
\sum_{N_3\ge \cdots \ge N_m}\left(\frac{\prod_{i=6}^mN_i^{\frac{1}{2}-s_c}}{\prod_{i=2}^5N_i^{s_c-\frac{1}{4}}}\right)
\prod_{i=3}^mc_{i,N_i}\\
&\lesssim 
\prod_{i=1}^m\left(\sum_{N_i}c_{i,N_i}^2\right)^{\frac{1}{2}}
\end{split}
\]
by the Cauchy-Schwarz inequality
because $(m-5)(\frac{1}{2}-s_c)= 4(s_c-\frac{1}{4})>0$. 
Therefore, we obtain (\ref{multilin_mhigh_2}) since 
\[
\sum_{N_i}c_{i,N_i}^2=\|u_i\|_{\dot{X}^{s_c}}^2\lesssim 1 
\]
for $i=2,\cdots m$. 
Now, we prove (\ref{m5_est}). 
By the H\"older inequality and the Bernstein inequality, 
we have
\[
\begin{split}
\left\|\prod_{i=1}^mP_{N_i}u_i\right\|_{L^1_xL^2_t}
&\le \|P_{N_1}u_1\|_{L^{\infty}_xL^2_t}
\prod_{i=2}^5\|P_{N_i}u_i\|_{L^4_xL^{\infty}_t}
\prod_{i=6}^m\|P_{N_i}u_i\|_{L^{\infty}_{x,t}}\\
&\lesssim N_1^{-s-\frac{3}{2}}
\left(\frac{\prod_{i=6}^mN_i^{\frac{1}{2}-s_c}}{\prod_{i=2}^5N_i^{s_c-\frac{1}{4}}}\right)
\prod_{i=1}^mc_{i,N_i}.
\end{split}
\]
The estimate (\ref{multilin_mhigh_hom}) follows from (\ref{multilin_mhigh_3}) with $s=s_c$. 
Next, we prove (\ref{multilin_mhigh_inhom}). 
By (\ref{multilin_mhigh_3}) with $s=0$, we have 
\[
\left\|\partial_x^{3}
\prod_{i=1}^{m}\widetilde{u_i}\right\|_{\dot{Z}^0}
\lesssim \sum_{i=0}^m\|u_i\|_{\dot{X}^{0}}\prod_{k\ne i}^m\|u_k\|_{\dot{X}^{s_c}}
\lesssim \prod_{i=1}^m\|u_i\|_{X^s} 
\]
for $s\ge s_c$. 
While by (\ref{multilin_mhigh_3}) with $s\ge s_c$, 
we have 
\[
\left\|\partial_x^{3}
\prod_{i=1}^{m}\widetilde{u_i}\right\|_{\dot{Z}^s}
\lesssim \sum_{i=0}^m\|u_i\|_{\dot{X}^{s}}\prod_{k\ne i}^m\|u_k\|_{\dot{X}^{s_c}}
\lesssim \prod_{i=1}^m\|u_i\|_{X^s}. 
\]
Therefore, we obtain (\ref{multilin_mhigh_inhom}).
\end{proof}
\begin{remark}
We cannot obtain (\ref{multilin_mhigh_hom}) and (\ref{multilin_mhigh_inhom}) 
for $m=5$ by the same argument 
because the factors $N_i^{s_c-\frac{1}{4}}$ ($i=2,\cdots,5$) are vanished. 
\end{remark}
\begin{theorem}[Multilinear estimates]\label{multi_est_3mg2}
Let $m\ge 5$. Set
\[
     s_c=s_c(m):=
     \frac{1}{2}-\frac{2}{m-1}.
\]
{\rm (i)}\ For any $u_1$, $\cdots$, $u_m\in \dot{X}^{s_c}$, 
it holds that
\begin{equation}\label{multilin_mhigh_hom_12}
\left\|\partial_x^{2}
\prod_{i=1}^{m}\widetilde{u_i}\right\|_{\dot{Z}^{s_c}}
\lesssim \prod_{i=1}^m\|u_i\|_{\dot{X}^{s_c}},
\end{equation}
where $\widetilde{u_i}\in \{u_i,\overline{u_i}\}$. The implicit constant depends only on $m$.\\
{\rm (ii)}\ If $s\ge s_c$,
then for any $u_1$, $\cdots$, $u_m\in X^s$, 
it holds that
\begin{equation}\label{multilin_mhigh_inhom_12}
\left\|\partial_x^{2}
\prod_{i=1}^{m}\widetilde{u_i}\right\|_{Z^s}
\lesssim \prod_{i=1}^m\|u_i\|_{X^{s}},
\end{equation}
where $\widetilde{u_i}\in \{u_i,\overline{u_i}\}$. The implicit constant depends only on $m$ and $s$.
\end{theorem}
\begin{proof}
Let $c_{i,N_i}$ $(i=1,\cdots ,m)$ 
are defined in the proof of Theorem~\ref{multi_est_3m}. 
Then, we have (\ref{m5_est}) by the same way, 
where we assumed
\[
\prod_{i=6}^mN_i^{\frac{1}{2}-s_c}=1
\]
if $m=5$. 
Therefore, we obtain
\[
\begin{split}
N^{s+2-\frac{3}{2}}
\left\|\prod_{i=1}^mP_{N_i}u_i\right\|_{L^1_xL^2_t}
&\lesssim \left(\frac{N}{N_1}\right)^{s+\frac{1}{2}}
\left(\frac{\prod_{i=2}^5N_i^{\frac{1}{4}-s_c}\prod_{i=6}^mN_i^{\frac{1}{2}-s_c}}{N_1}\right)
\prod_{i=1}^mc_{i,N_i}, 
\end{split}
\]
This implies
\[
\left\|\partial_x^{2}
\prod_{i=1}^{m}\widetilde{u_i}\right\|_{\dot{Z}^{s}}
\lesssim \sum_{i=1}^m\|u_i\|_{\dot{X}^s}\prod_{\substack{1\le k\le m\\ k\ne i}}\|u_k\|_{\dot{X}^{s_c}}.
\]
for $s\ge s_c$
by the same argument as in the proof of Theorem~\ref{multi_est_3m} 
because $1=4(\frac{1}{4}-s_c)+(m-5)(\frac{1}{2}-s_c)>0$. 
In particular, we have
\[
I_1\lesssim \prod_{i=1}^m\left(\sum_{N_i}c_{i.N_i}^2\right)^{\frac{1}{2}}
\]
and need not the renormalize argument for $\|u_i\|_{\dot{X}^{s_c}}$ ($i=2,\cdots,m$). 
\end{proof}
To treat the large data for the scaling critical case, 
we give the following.
\begin{theorem}[Multilinear estimates]\label{multi_est_3mg2_lowf}
Let $m\ge 5$, $0<T<1$, and $M\in 2^{\N}$. Set
\[
     s_c=s_c(m):=
     \frac{1}{2}-\frac{2}{m-1}.
\]
{\rm (i)}\ For any $u_1$, $\cdots$, $u_m\in \dot{X}^{s_c}$, 
it holds that
\begin{equation}\label{multilin_mhigh_hom_12_lowf}
\left\|\partial_x^{2}\left(
\prod_{i=1}^{m}\widetilde{u_i}-\prod_{i=1}^{m}P_{\ge M}\widetilde{u_i}\right)\right\|_{\dot{Z}^{s_c}(T)}
\lesssim T^{\delta}M^{\kappa}\prod_{i=1}^m\|u_i\|_{\dot{X}^{s_c}(T)}
\end{equation}
for some $\delta >0$ and $\kappa >0$ depending only on $m$.\\
{\rm (ii)}\ For any $u_1$, $\cdots$, $u_m\in X^{s_c}$, 
it holds that
\begin{equation}\label{multilin_mhigh_hom_12_lowf2}
\left\|\partial_x^{2}\left(
\prod_{i=1}^{m}\widetilde{u_i}-\prod_{i=1}^{m}P_{\ge M}\widetilde{u_i}\right)\right\|_{Z^{s_c}(T)}
\lesssim T^{\delta}M^{\kappa}\prod_{i=1}^m\|u_i\|_{X^{s_c}(T)}
\end{equation}
for some $\delta >0$ and $\kappa >0$ depending only on $m$.
\end{theorem}
\begin{proof}
By the symmetry, we can assume $N_1\ge \cdots \ge N_m$ and $N_m<M$. 
Let $s\ge 0$, $c_{1,N_1}:=N_1^s\|P_{N_1}u_1\|_{X_{N_1}(T)}$, 
$c_{i,N_i}:=N_i^{s_c}\|P_{N_i}u_i\|_{X_{N_i}(T)}$ $(i=2,\cdots ,m-1)$, 
and $c_{m,N_m}:=N_m^{s_c}\|P_{<M}P_{N_m}u_m\|_{X_{N_m}(T)}$. 
We first assume the case $m\ge 6$
By the H\"older inequality, 
we have
\[
\begin{split}
\left\|\prod_{i=1}^mP_{N_i}u_i\right\|_{L^1_xL^2_T}
&\le T^{\frac{1}{2}-\frac{1}{p}}
\left\|\prod_{i=1}^mP_{N_i}u_i\right\|_{L^1_xL^p_T}\\
&\le T^{\frac{1}{2}-\frac{1}{p}}
\|P_{N_1}u_1\|_{L^{\infty}_xL^p_T}
\left(\prod_{i=2}^5\|P_{N_i}u_i\|_{L^4_xL^{\infty}_T}
\prod_{i=6}^{m-1}\|P_{N_i}u_i\|_{L^{\infty}_{x}L^{\infty}_{T}}\right)
\|P_{<M}P_{N_m}u_m\|_{L^{\infty}_xL^{\infty}_T} 
\end{split}
\]
for $p>2$, where we assumed
\[
\prod_{i=6}^{m-1}\|P_{N_i}u_i\|_{L^{\infty}_{x}L^{\infty}_{T}}=1
\]
when $m=6$. 
By the interpolation between the two estimates
\[
\|P_{N_1}u_1\|_{L^{\infty}_xL^2_T}\lesssim N_1^{-\frac{3}{2}}\|P_{N_1}u_1\|_{X_{N_1}(T)},\ \ 
\|P_{N_1}u_1\|_{L^{\infty}_xL^{\infty}_T}\lesssim N_1^{\frac{1}{2}}\|P_{N_1}u_1\|_{X_{N_1}(T)}, 
\]
we obtain
\[
\|P_{N_1}u_1\|_{L^{\infty}_xL^p_T}\lesssim N_1^{\frac{1}{2}-\frac{4}{p}}\|P_{N_1}u_1\|_{X_{N_1}(T)}.
\]
Therefore, by the Bernstein inequality, it holds that
\[
\begin{split}
N^{s+2-\frac{3}{2}}\left\|\prod_{i=1}^mP_{N_i}u_i\right\|_{L^1_xL^2_T}
&\lesssim T^{\frac{1}{2}-\frac{1}{p}}\left(\frac{N}{N_1}\right)^{s+\frac{1}{2}}
\frac{\left(\prod_{i=2}^5N_i^{\frac{1}{4}-s_c}
\prod_{i=6}^{m-1}N_i^{\frac{1}{2}-s_c}\right)
N_m^{\frac{1}{2}-s_c}}{N_1^{\frac{4}{p}-1}}
\prod_{i=1}^mc_{i,N_i}. 
\end{split}
\]
We choose $p>2$ as $p=\frac{2(m-1)}{m-2}$. 
Then, $4(\frac{1}{4}-s_c)+(m-6)(\frac{1}{2}-s_c)=\frac{4}{p}-1>0$. 
Therefore, we obtain
\[
\left\|\partial_x^{2}
\left(\prod_{i=1}^{m}\widetilde{u_i}-\prod_{i=1}^mP_{\ge M}\widetilde{u_i}\right)\right\|_{\dot{Z}^{s}}
\lesssim T^{\frac{1}{2}-\frac{1}{p}}M^{\frac{1}{2}-s_c}
\sum_{i=1}^m\|u_i\|_{\dot{X}^s}\prod_{\substack{1\le k\le m\\ k\ne i}}\|u_k\|_{\dot{X}^{s_c}}
\]
for $s\ge s_c$
by the same argument as in the proof of Theorem~\ref{multi_est_3mg2} 
because $N_m<M$ and $\frac{1}{2}-s_c>0$. 

Next, we assume the case $m=5$. Then, $s_c=0$. Therefore, it holds that 
\[
\begin{split}
N^{s+2-\frac{3}{2}}\left\|\prod_{i=1}^5P_{N_i}u_i\right\|_{L^1_xL^2_T}
&\le T^{\frac{1}{2}-\frac{1}{p}}N^{s+2-\frac{3}{2}}
\left\|\prod_{i=1}^5P_{N_i}u_i\right\|_{L^1_xL^p_T}\\
&\le T^{\frac{1}{2}-\frac{1}{p}}N^{s+2-\frac{3}{2}}
\|P_{N_1}u_1\|_{L^{\infty}_xL^p_T}
\left(\prod_{i=2}^4\|P_{N_i}u_i\|_{L^4_xL^{\infty}_T}\right)
\|P_{<M}P_{N_5}u_5\|_{L^{4}_xL^{\infty}_T}\\
&\lesssim 
T^{\frac{1}{2}-\frac{1}{p}}\left(\frac{N}{N_1}\right)^{s+\frac{1}{2}}
\frac{\left(\prod_{i=2}^4N_i^{\frac{1}{4}}
\right)
N_5^{\frac{1}{4}}}{N_1^{\frac{4}{p}-1}}
\prod_{i=5}^mc_{i,N_i}
\end{split}
\]
and obtain 
\[
\left\|\partial_x^{2}
\left(\prod_{i=1}^{5}\widetilde{u_i}-\prod_{i=1}^mP_{\ge M}\widetilde{u_i}\right)\right\|_{\dot{Z}^{s}}
\lesssim T^{\frac{1}{2}-\frac{1}{p}}M^{\frac{1}{4}}
\sum_{i=1}^5\|u_i\|_{\dot{X}^s}\prod_{\substack{1\le k\le m\\ k\ne i}}\|u_k\|_{\dot{X}^{s_c}}
\]
for $s\ge s_c$ by choosing $p>2$ as $p=\frac{16}{7}$ because $\frac{3}{4}=4\cdot \frac{7}{16}-1>0$ 
and $N_5<M$. 
\end{proof}

\section{Multilinear estimates at the scaling critical regularity in the cases $m=5$ with $\gamma=3$ and $m=4$ with $\gamma=2$}\label{mesc_1}
In this section, we study the Cauchy problem (\ref{1-1}) with a specific scaling invariant nonlinearity (\ref{nonl_sc_2}) at the scaling critical regularity $s=s_c(\gamma,m)$ in the cases $(\gamma,m)=(3,5)$ and $(2,4)$. Let $N\in 2^{\Z}$. To do so, we use the following more sophisticated Banach spaces $X_N$ and $Y_N$ endowed with the norms 
\begin{equation}
\label{fsb}
\begin{split}
\|u\|_{X_N}&:=\|u(0)\|_{L^2_x}+\left\|\left(i\partial_t+\partial_x^4\right)u\right\|_{Y_N},\\
\|F\|_{Y_N}&:=\inf\left\{\left.\|F_1\|_{Z_N}+\|F_2\|_{\dot{X}^{0,-\frac{1}{2},1}}
\right|F=F_1+F_2\right\}, \\
\|F\|_{Z_N}&:=N^{-\frac{3}{2}}\|F\|_{L^1_xL^2_t}, 
\end{split}
\end{equation}
respectively, instead of Definition~\ref{def4-1}.
Here $\|\cdot \|_{\dot{X}^{0,b,q}}$ denotes the Besov type Fourier restriction norm given by
\[
\|u\|_{\dot{X}^{0,b,q}}
:=
\begin{cases}
\displaystyle \left(\sum_{A\in 2^{\Z}}A^{bq}\|Q_Au\|_{L^2_{t,x}}^q\right)^{\frac{1}{q}},&{\rm if}\ 1\le q<\infty, \\
\displaystyle \sup_{A\in 2^{\Z}}A^b\|Q_Au\|_{L^2_{t,x}},&{\rm if}\ q=\infty, 
\end{cases}
\]
where $Q_A$ with $A\in 2^{\Z}$ denotes the Littlewood-Paley projection defined by
\[
\F_{t,x}[Q_Au](\tau,\xi):=\psi_{A}\left(\tau -\xi^4\right)\F_{t,x}[u](\tau,\xi). 
\]
We note that $\|(i\partial_t+\partial_x^4)u\|_{\dot{X}^{0,-\frac{1}{2},1}}=\|u\|_{\dot{X}^{0,\frac{1}{2},1}}$.
Furthermore, we define the Banach spaces $\dot{X}^s$, $X^s$, $\dot{Y}^s$, and $Y^s$ endowed with the norms
\[
\begin{split}
\|u\|_{\dot{X}^s}
&:=\left(\sum_{N\in 2^{\Z}}N^{2s}\|P_Nu\|_{X_N}^2\right)^{\frac{1}{2}},\ \ 
\|u\|_{X^s}:=\|u\|_{\dot{X}^0}+\|u\|_{\dot{X}^s},\\
\|F\|_{\dot{Y}^s}
&:=\left(\sum_{N\in 2^{\Z}}N^{2s}\|P_Nu\|_{Y_N}^2\right)^{\frac{1}{2}},\ \ 
\|F\|_{Y^s}:=\|F\|_{\dot{Y}^0}+\|F\|_{\dot{Y}^s}, 
\end{split}
\] 
where $s\in \R$. We can see that by Lemma \ref{LP}, the estimates
\begin{equation}\label{XN_Linest_g2}
\left\|e^{it\partial_x^4}f\right\|_{\dot{X}^s}\lesssim \|f\|_{\dot{H}^s},\ \ 
\left\|\int_0^te^{i(t-t')\partial_x^4}F(t')dt'\right\|_{\dot{X}^s}\lesssim \|F\|_{\dot{Y}^s}
\end{equation}
hold for any $f\in \dot{H}^s(\R)$ and $F\in \dot{Y}^s$. 
\begin{remark}
The similar norms as $\|\cdot \|_{X_N}$ and $\|\cdot \|_{Y_N}$ are used by Tao (\cite{T06}) to prove the well-posedness of 
the quartic generalized Korteweg-de Vries at the scaling critical regularity $\dot{H}^{-\frac{1}{6}}$ (see \cite[Section 2]{T06}).
In the proof of \cite[Theorem 1.3]{Ppre}, Pornnopparath used such norms 
to prove the small data global well-posedness of the quintic derivative nonlinear Schr\"odinger equations 
at the scaling critical regularity $\dot{H}^{\frac{1}{4}}$.
\end{remark}
\begin{lemma}[Extension lemma]\label{BSE_Lemm}
Let $S$ be any space-time Banach space that satisfies
\[
\|g(t)F(t,x)\|_{S}\lesssim \|g\|_{L^{\infty}_t}\|F(t,x)\|_S
\]
for any $F\in S$ and $g\in L^{\infty}_t$. 
Let $\mathcal{T}\ :\ L^2(\R)\times \cdots \times L^2(\R)\rightarrow S$ 
be a spatial multilinear operator satisfying
\[
\left\|\mathcal{T}(e^{it\partial_x^4}u_{1,0},\cdots,e^{it\partial_x^4}u_{k,0})\right\|_S
\lesssim \prod_{j=1}^k\|u_{j,0}\|_{L^2_x}
\]
for any $u_{1,0},\cdots,u_{k,0}\in L^2(\R)$ with $k\in \N$. Then it holds that
\[
\|\mathcal{T}(u_1,\cdots,u_k)\|_S
\lesssim \prod_{j=1}^k\left(\|u_j(0)\|_{L^2_x}
+\left\|(i\partial_t+\partial_x^4)u_j\right\|_{\dot{X}^{0,-\frac{1}{2},1}}\right)
\]
for any $u_1,\cdots u_k\in \dot{X}^{0,\frac{1}{2},1}$. 
\end{lemma}
For the proof of this lemma, see Lemma\ 4.1 in \cite{T07}.

By Lemma~\ref{BSE_Lemm}, the linear estimates (Lemma~\ref{lem3-1}, ~\ref{kato}, ~\ref{KRe}), 
and the same argument as in the proof of Theorem~\ref{thm3-1} (See also Remark~\ref{Tcondi}), 
we obtain the following. 
\begin{proposition}\label{lin_XN}
Let $N\in 2^{\Z}$. It holds that
\[
\|P_Nu\|_{L_t^{\infty}L_x^2}+N^{\frac{1}{2}}\|P_Nu\|_{L^4_tL^{\infty}_x}+N^{\frac{3}{2}}\|P_Nu\|_{L^{\infty}_xL^2_t}
+N^{-\frac{1}{4}}\|P_Nu\|_{L^{4}_xL^{\infty}_x}\lesssim \|P_Nu\|_{X_N}
\]
for any $u\in X_N$. 
\end{proposition}
Furthermore, we get the following.
\begin{proposition}\label{dual_besov}
Let $N\in 2^{\Z}$. It holds that
\[
\|P_Nu\|_{\dot{X}^{0,\frac{1}{2},\infty}}\lesssim \|P_Nu\|_{X_N}
\]
for any $u\in X_N$. 
\end{proposition}
\begin{proof}
We put $F=F_1+F_2:=\left(i\partial_t+\partial_x^4\right)u$, where $P_NF_1\in Z_N$ and $P_NF_2\in \dot{X}^{0,-\frac{1}{2},1}$. 
Then, we have
\[
\begin{split}
u(t)&=e^{it\partial_x^4}u_0-i\int_0^te^{i(t-t')\partial_x^4}F_1(t')dt'
-i\int_0^te^{i(t-t')\partial_x^4}F_2(t')dt'\\
&=e^{it\partial_x^4}u_0-i\sum_{k=1}^2
\left(
\int_{-\infty}^te^{i(t-t')\partial_x^4}F_k(t')dt'
-e^{it\partial_x^4}\int_{-\infty}^0e^{-it'\partial_x^4}F_k(t')dt'
\right).
\end{split}
\]
Because the space-time Fourier transform of the 
linear solution is supported in $\{(\tau,\xi)|\ \tau-\xi^4=0\}$, 
we have
\[
\left\|P_Ne^{it\partial_x^4}u_0\right\|_{\dot{X}^{0,\frac{1}{2},\infty}}=0,\ \ 
\left\|P_Ne^{it\partial_x^4}\int_{-\infty}^0e^{-it'\partial_x^4}F_k(t')dt'\right\|_{\dot{X}^{0,\frac{1}{2},\infty}}=0.
\]
Therefore, it suffices to show that
\begin{equation}\label{besov_dual_keyest1}
\left\|P_N\int_{-\infty}^te^{i(t-t')\partial_x^4}F_1(t')dt'\right\|_{\dot{X}^{0,\frac{1}{2},\infty}}
\lesssim \left\|P_NF_1\right\|_{Z_N}
\end{equation}
and
\begin{equation}\label{besov_dual_keyest2}
\left\|P_N\int_{-\infty}^te^{i(t-t')\partial_x^4}F_2(t')dt'\right\|_{\dot{X}^{0,\frac{1}{2},\infty}}
\lesssim \left\|P_NF_2\right\|_{\dot{X}^{0,-\frac{1}{2},1}}.
\end{equation}

We first prove (\ref{besov_dual_keyest1}).
For $y\in \R$, we put
\[
    \mathrm{w}_{y}(t,x):=
    \frac{1}{\sqrt{2\pi}}\int_{-\infty}^t(P_N\mathcal{K})(t-t',x-y)F_1(t',y)dt', 
\]
where $\mathcal{K}$ is defined by (\ref{fs}). 
Then, we have
\begin{equation}\label{duamel_wy_73}
P_N\int_{-\infty}^te^{i(t-t')\partial_x^4}F_1(t')dt'
=\int_{\R}\mathrm{w}_{y}(t,x)dy
\end{equation}
and
\[
\F_{t,x}[w_{y}](\tau, \xi)
=\frac{1}{\sqrt{2\pi}i}\frac{e^{-i\xi y}\psi_N(\xi)}{\tau-\xi^4-i0}\F_t[F_1](\tau,y). 
\]
Therefore, for $A\in 2^{\Z}$, 
by using the variable transform $\xi \mapsto \omega$ as $\omega =\tau-\xi^4$, 
we have
\[
\begin{split}
\|Q_Aw_{y}\|_{L^2_{t,x}}
&\sim \|\psi_A(\tau-\xi^4)\F_{t,x}[w_{y}](\tau,\xi)\|_{L^2_{\tau,\xi}}\\
&\sim \left(\int_{\R}\int_{\R}\frac{\psi_N(\xi)\psi_A(\tau-\xi^4)}{(\tau-\xi^4)^2}
|\F_t[F_1](\tau,y)|^2d\xi d\tau\right)^{\frac{1}{2}}\\
&\lesssim \left(\int_{\R}N^{-3}\int_{\R}\frac{\psi_A(\omega )}{\omega^2}
|\F_t[F_1](\tau,y)|^2d\omega d\tau\right)^{\frac{1}{2}}\\
&\lesssim N^{-\frac{3}{2}}A^{-\frac{1}{2}}
\left(\int_{\R}
|\F_t[F_1](\tau,y)|^2d\tau\right)^{\frac{1}{2}}\\
&\lesssim N^{-\frac{3}{2}}A^{-\frac{1}{2}}\|F_1(t,y)\|_{L^2_t}. 
\end{split}
\]
This and (\ref{duamel_wy_73}) imply (\ref{besov_dual_keyest1}). 

Next, we prove (\ref{besov_dual_keyest2}). 
By the direct calculation, we have
\[
\F_{t,x}\left[P_N\int_{-\infty}^te^{i(t-t')\partial_x^4}F_2(t')dt'\right](\tau,\xi)
=\frac{1}{i}\frac{\psi_N(\xi)}{\tau-\xi^4-i0}\F_{t,x}[F_2](\tau,\xi). 
\]
Therefore, for $A\in 2^{\Z}$, we obtain 
\[
\left\|Q_AP_N\int_{-\infty}^te^{i(t-t')\partial_x^4}F_2(t')dt'\right\|_{L^2_{t,x}}
\lesssim A^{-1}\|Q_LP_NF_2\|_{L^2_{t,x}}. 
\]
This and the embedding 
$\dot{X}^{0,-\frac{1}{2},1}\hookrightarrow \dot{X}^{0,-\frac{1}{2},\infty}$ imply (\ref{besov_dual_keyest2}). 
\end{proof}
\subsection{Refined bilinear Strichartz estimates}
For $L\in 2^{\Z}$, we define the bilinear operators $R^{\pm}$ as
\begin{equation}
\label{defR}
 R_L^{\pm}(f,g):=\int_{\R}\int_{\R}e^{i\xi x}\psi_L(\xi_1\pm (\xi-\xi_1))\widehat{f}(\xi_1)\widehat{g}(\xi-\xi_1)d\xi_1d\xi. 
\end{equation}
\begin{lemma}[Refined bilinear Strichartz estimate $L^2(\R)\times L^2(\R)\rightarrow L^2_{t,x}(\R\times\R)$]
\label{BS2}
Let $L, N_1, N_2\in 2^{\Z}$ with $N_1\ge N_2$. Then for any functions $f,g$ satisfying $P_{N_1}f, P_{N_2}g\in L^2(\R)$, the estimates
\begin{align}
\label{5-2}
   \left\|R_L^+\left(P_{N_1}e^{it\partial_x^4}f, \overline{P_{N_2}e^{it\partial_x^4}g}\right)\right\|_{L_{t,x}^2(\R\times\R)}
   &\le CN_1^{-1}L^{-\frac{1}{2}}\left\|P_{N_1}f\right\|_{L^2}\left\|P_{N_2}g\right\|_{L^2},\\
   \label{5-3}
   \left\|R_L^-\left(P_{N_1}e^{it\partial_x^4}f, P_{N_2}e^{it\partial_x^4}g\right)\right\|_{L_{t,x}^2(\R\times\R)}
   &\le CN_1^{-1}L^{-\frac{1}{2}}\left\|P_{N_1}f\right\|_{L^2}\left\|P_{N_2}g\right\|_{L^2}
\end{align}
hold, where $C$ is a positive constant independent of $L, N_1,N_2, f,g$.
\end{lemma}
\begin{proof}
We only prove (\ref{5-3}), since (\ref{5-2}) can be obtained in the similar manner. 
By Plancherel's theorem, the equivalency
\[
\begin{split}
 &\left\|R_L^-\left(P_{N_1}e^{it\partial_x^4}f, P_{N_2}e^{it\partial_x^4}g\right)\right\|_{L_{t,x}^2(\R\times\R)}\\
 &\sim 
 \left\|
 \int_{\R}\psi_L(\xi_1- (\xi-\xi_1))e^{it\xi_1^4}\psi_{N_1}(\xi_1)\widehat{f}(\xi_1)
 e^{it(\xi-\xi_1)^4}\psi_{N_2}(\xi-\xi_1)\widehat{g}(\xi-\xi_1)d\xi_1
 \right\|_{L_{t,\xi}^2(\R\times\R)}.
\end{split}
\]
holds. Thus, it suffices to show that the estimate
\begin{equation}\label{BSE_keyest}
\begin{split}
I&:=\left|
\int_\R\left(\iint_{\Omega}e^{it\xi_1^4}\widehat{f}(\xi_1)
 e^{it(\xi-\xi_1)^4}\widehat{g}(\xi-\xi_1)
h(t,\xi)d\xi_1d\xi \right)dt\right|\\
&\lesssim 
N_1^{-1}L^{-\frac{1}{2}}\|\psi_{N_1}f\|_{L^2_\xi}\|\psi_{N_2}g\|_{L^2_\xi}\|h\|_{L^2_{t,\xi}}
\end{split}
\end{equation}
holds for any $h\in L^2(\R\times \R)$, where $\Omega=\Omega(L,N_1,N_2)$ is defined by
\[
\Omega :=\{(\xi_1,\xi)|\ |\xi_1|\sim N_1,\ |\xi-\xi_1|\sim N_2,\ |\xi_1-(\xi-\xi_1)|\sim L\}. 
\]
Since the identity
\[
\int_{\R}e^{it(\xi_1^4+(\xi-\xi_1)^4)}h(t,\xi)dt
=\sqrt{2\pi}\F_t[h]\left(-\xi_1^4-(\xi-\xi_1)^4,\xi\right)
\]
holds for any $\xi,\xi_1\in \R$, by the Cauchy-Schwarz inequality, the estimate
\[
\begin{split}
I&=\sqrt{2\pi}\left|
\iint_{\Omega}\widehat{f}(\xi_1)
 \widehat{g}(\xi-\xi_1)
\F_t[h](-\xi_1^4-(\xi-\xi_1)^4,\xi)d\xi_1d\xi \right|\\
&\lesssim 
\|\psi_{N_1}f\|_{L^2_{\xi}}\|\psi_{N_2}g\|_{L^2_{\xi}}
\left(\iint_{\Omega}|\F_t[h]\left(-\xi_1^4-(\xi-\xi_1)^4,\xi\right)|^2d\xi_1d\xi\right)^{\frac{1}{2}}
\end{split}
\]
holds. We use changing variables with $\xi_1\mapsto \tau$ as $\tau =-\xi_1^4-(\xi-\xi_1)^4$. 
Since the relations
\[
\left|\frac{d\tau}{d\xi_1}\right|
\sim \left|\xi_1^3-(\xi-\xi_1)^3\right|
\sim |\xi_1-(\xi-\xi_1)|\max\left\{|\xi_1|^2,|\xi-\xi_1|^2\right\}
\sim LN_1^2
\]
hold for any $(\xi_1,\xi)\in \Omega$, by Plancherel's theorem, the estimates
\[
\begin{split}
\iint_{\Omega}|\F_t[h](-\xi_1^4-(\xi-\xi_1)^4,\xi)|^2d\xi_1d\xi
&\lesssim N_1^{-2}L^{-1}\iint |\F_t[h](\tau, \xi)|^2d\tau d\xi\sim N_1^{-2}L^{-1}\|h\|_{L^2_{t,\xi}}^2
\end{split}
\]
hold, which implies (\ref{BSE_keyest}). 
\end{proof}
\begin{remark}\label{BSE_rema}
When $N_1\gg N_2$, the relations $|\xi_1|\sim N_1$ and $|\xi-\xi_1|\sim N_2$ 
imply the equivalency $|\xi_1-(\xi-\xi_1)|\sim N_1$. 
Therefore, Lemma~\ref{BS} follows from (\ref{5-3}) with $L=N_1$. 
\end{remark}
\begin{theorem}[Refined bilinear Strichartz estimate on $X_{N_1}\times X_{N_2}$]
\label{bsere}
Let $L,N_1,N_2\in 2^{\Z}$ and $u_1\in X_{N_1}, u_2\in X_{N_2}$. 
If $N_1\ge N_2\gtrsim L$, then the estimates
\begin{align}
\label{bistri_2}
     \left\|R_L^+\left(P_{N_1}u_1,\overline{P_{N_2}u_2}\right)\right\|_{L^2_{t,x}}
     &\lesssim
     N_1^{-1}L^{-\frac{1}{2}}\|P_{N_1}u_1\|_{X_{N_1}}\|P_{N_2}u_2\|_{X_{N_2}},\\
\label{bistri_3}
     \|R_L^-(P_{N_1}u_1,P_{N_2}u_2)\|_{L^2_{t,x}}
     &\lesssim
     N_1^{-1}L^{-\frac{1}{2}}\|P_{N_1}u_1\|_{X_{N_1}}\|P_{N_2}u_2\|_{X_{N_2}}
\end{align}
hold, where the implicit constants are independnet of $L,N_1,N_2,u_1,u_2$.
\end{theorem}
\begin{proof}
We only prove (\ref{bistri_3}) since (\ref{bistri_2}) can be obtained in the similar manner. 
We set $u_{j,N_j}:=P_{N_j}u_j$ and $F_j:=\left(i\partial_t+\partial_x^4\right)u_{j}$ for $j=1,2$.
It suffices to show that the estimate
\[
\|R_{L}^{-}(u_{1,N_1},u_{2,N_2})\|_{L^2_{t,x}}\lesssim 
N_1^{-1}L^{-\frac{1}{2}}\left(\|u_{1,N_1}(0)\|_{L^2_x}+\|F_1\|_{Y_{N_1}}\right)
\left(\|u_{2,N_2}(0)\|_{L^2_x}+\|F_2\|_{Y_{N_2}}\right)
\]
holds. This follows from the following estimates:
\begin{align}
\|R_{L}^{-}(u_{1,N_1},u_{2,N_2})\|_{L^2_{t,x}}&\lesssim 
N_1^{-1}L^{-\frac{1}{2}}\left(\|u_{1,N_1}(0)\|_{L^2_x}+\|F_1\|_{\dot{X}^{0,-\frac{1}{2},1}}\right) 
\left(\|u_{2,N_2}(0)\|_{L^2_x}+\|F_2\|_{\dot{X}^{0,-\frac{1}{2},1}}\right),\label{bi_est_L_1}\\ 
\|R_{L}^{-}(u_{1,N_1},u_{2,N_2})\|_{L^2_{t,x}}&\lesssim 
N_1^{-1}L^{-\frac{1}{2}}\left(\|u_{1,N_1}(0)\|_{L^2_x}+\|F_1\|_{\dot{X}^{0,-\frac{1}{2},1}}\right) 
\left(\|u_{2,N_2}(0)\|_{L^2_x}+N_2^{-\frac{3}{2}}\|F_2\|_{L^1_xL^2_t}\right),\label{bi_est_L_2}\\ 
\|R_{L}^{-}(u_{1,N_1},u_{2,N_2})\|_{L^2_{t,x}}&\lesssim 
N_1^{-1}L^{-\frac{1}{2}}\left(\|u_{1,N_1}(0)\|_{L^2_x}+N_1^{-\frac{3}{2}}\|F_1\|_{L^1_xL^2_t}\right) 
\left(\|u_{2,N_2}(0)\|_{L^2_x}+\|F_2\|_{\dot{X}^{0,-\frac{1}{2},1}}\right),\label{bi_est_L_3}\\ 
\|R_{L}^{-}(u_{1,N_1},u_{2,N_2})\|_{L^2_{t,x}}&\lesssim 
N_1^{-1}L^{-\frac{1}{2}}\left(\|u_{1,N_1}(0)\|_{L^2_x}+N_1^{-\frac{3}{2}}\|F_1\|_{L^1_xL^2_t}\right) 
\left(\|u_{2,N_2}(0)\|_{L^2_x}+N_2^{-\frac{3}{2}}\|F_2\|_{L^1_xL^2_t}\right). \label{bi_est_L_4}
\end{align}
We can obtain the estimate (\ref{bi_est_L_4}) in the same manner as the proof of (\ref{bi_est_4}). Indeed, we use (\ref{RL_P_est}) and (\ref{RL_P_est_mi}) below instead of (\ref{Pg_est}) and (\ref{Pg_est_mi}).
To obtain (\ref{bi_est_L_1}), (\ref{bi_est_L_2}), and (\ref{bi_est_L_3}), 
we use Lemma~\ref{BSE_Lemm}. Then, we have only to prove
\begin{align}
\left\|R_{L}^{-}\left(e^{it\partial_x^4}u_{1,N_1}(0),e^{it\partial_x^4}u_{2,N_2}(0)\right)\right\|_{L^2_{t,x}}&\lesssim 
N_1^{-1}L^{-\frac{1}{2}}\|u_{1,N_1}(0)\|_{L^2_x}\|u_{2,N_2}(0)\|_{L^2_x},\label{bi_est_L_5}\\ 
\left\|R_{L}^{-}\left(e^{it\partial_x^4}u_{1,N_1}(0),u_{2,N_2}\right)\right\|_{L^2_{t,x}}&\lesssim 
N_1^{-1}L^{-\frac{1}{2}}\|u_{1,N_1}(0)\|_{L^2_x} 
\left(\|u_{2,N_2}(0)\|_{L^2_x}+N_2^{-\frac{3}{2}}\|F_2\|_{L^1_xL^2_t}\right),\label{bi_est_L_6}\\ 
\left\|R_{L}^{-}\left(u_{1,N_1},e^{it\partial_x^4}u_{2,N_2}(0)\right)\right\|_{L^2_{t,x}}&\lesssim 
N_1^{-1}L^{-\frac{1}{2}}\left(\|u_{1,N_1}(0)\|_{L^2_x}+N_1^{-\frac{3}{2}}\|F_1\|_{L^1_xL^2_t}\right) 
\|u_{2,N_2}(0)\|_{L^2_x}.\label{bi_est_L_7}
\end{align}
The estimate (\ref{bi_est_L_5}) is obtained by (\ref{5-3}). 
We prove (\ref{bi_est_L_6}) only since (\ref{bi_est_L_7}) can be shown in the similar manner. 
We note that the identity holds
\[
u_{2,N_2}(t)=e^{it\partial_x^4}u_{2,N_2}(0)-i\int_0^te^{i(t-t')\partial_x^4}P_{N_2}F_2(t')dt'
=:A_2+B_2
\]
for any $t\in \R$. To obtain (\ref{bi_est_L_6}), we prove the following estimates:
\begin{align}
\left\|R_{L}^{-}\left(e^{it\partial_x^4}u_{1,N_1}(0),A_2\right)\right\|_{L^2_{t,x}}&\lesssim N_1^{-1}L^{-\frac{1}{2}}\|u_{1,N_1}(0)\|_{L^2_x}\|u_{2,N_2}(0)\|_{L^2_x},
\label{AB_est_L_1}\\
\left\|R_{L}^{-}\left(e^{it\partial_x^4}u_{1,N_1}(0),B_2\right)\right\|_{L^2_{t,x}}&\lesssim N_1^{-1}N_2^{-\frac{3}{2}}L^{-\frac{1}{2}}\|F_1\|_{L^1_xL^2_t}\|u_{2,N_2}(0)\|_{L^2_x}.
\label{AB_est_L_2}
\end{align}
The estimate (\ref{AB_est_L_1}) is obtained by (\ref{5-3}) 
because $A_2$ is a linear solution. 

Now we prove (\ref{AB_est_L_2}).
By Proposition~\ref{duam_decom}, we have
\[
\begin{split}
B_2&=-\int_{\R}e^{it\partial_x^4}\mathcal{L}v_{2,y}(x)dy
+\int_{\R}(P_{<L/2^{50}}\ee_{(-\infty, 0]})(x)(P_{+}e^{it\partial_x^4}v_{2,y})(x)dy\\
&\ \ \ \ -\int_{\R}(P_{<L/2^{50}}\ee_{[0,\infty)})(x)(P_{-}e^{it\partial_x^4}v_{2,y})(x)dy
+\int_{\R}h_{2,y}(t,x)dy,
\end{split}
\]
where $v_{2,y}=\F_{\xi}^{-1}[\psi_N(\xi)\F_t[F_2(t,y)](\xi^4)]$, 
$\mathcal{L}v_{2,y}=\F_{\xi}^{-1}[\psi_N(\xi)\F_t[\ee_{(-\infty ,0]}(t)F_2(t,y)](\xi^4)]$ 
and $h_{2,y}$ satisfies
\begin{equation}\label{hy_est_L}
\|h_{2,y}\|_{L^q_xL^p_t}\lesssim L^{-\frac{1}{2}-\frac{1}{p}}N_2^{-\frac{1}{2}-\frac{1}{q}-\frac{3}{p}}\|F_2(t,y)\|_{L^2_t}. 
\end{equation}
We note that
\begin{equation}\label{vy_est_L}
\|v_{2,y}(x)\|_{L^2_x}\lesssim N_2^{-\frac{3}{2}}\|F_2(t,y)\|_{L^2_t},\ 
\|\mathcal{L}v_{1,y}(x)\|_{L^2_x}\lesssim N_2^{-\frac{3}{2}}\|F_2(t,y)\|_{L^2_t}.
\end{equation}
Furthermore, it holds that
\begin{equation}\label{RL_P_est}
\|R_{L}^{-}(g_1(t,x), (P_{<L/2^{50}}\ee_{(-\infty, 0]})(x)g_2(t,x))\|_{L^2_{tx}}
\lesssim \|R_{L}^{-}(g_1,g_2)\|_{L^2_{tx}}. 
\end{equation}
Now, we prove (\ref{RL_P_est}). 
Because $|\xi_1-(\xi-\xi_1)|\sim L$ and $|\xi_2|\ll L$ imply
$|\xi_1-(\xi-\xi_2-\xi_1)|\sim L$, 
we have
\[
\begin{split}
&\F_x[R_{L}^{-}(g_1(t,x), (P_{<L/2^{50}}\ee_{(-\infty, 0]})(x)g_2(t,x))](\xi)\\
&=\int\psi_{L}(\xi_1-(\xi-\xi_1))\widehat{g_1}(t,\xi_1)
\left(\F_x[P_{<L/2^{50}}\ee_{(-\infty, 0]}]*\widehat{g_2}(t)\right)(\xi-\xi_1)d\xi_1\\
&\sim \int \psi_{L}(\xi_1-(\xi-\xi_2-\xi_1))\widehat{g_1}(t,\xi_1)\F_x[P_{<L/2^{50}}\ee_{(-\infty, 0]}](\xi_2)
\widehat{g_2}(\xi-\xi_2-\xi_1)d\xi_1d\xi_2\\
&=\int \F_x[P_{<L/2^{50}}\ee_{(-\infty, 0]}](\xi_2)\F_x[R_L^-(g_1(t,x),g_2(t,x))](\xi-\xi_2)d\xi_2\\
&=\F_x[P_{<L/2^{50}}\ee_{(-\infty, 0]}(x)R_L^-(g_1(t,x),g_2(t,x))](\xi ).
\end{split}
\]
Therefore, if $\chi_{L}$ is defined by $P_{<L/2^{50}}f=\chi_{L}*f$, then we have
\[
\begin{split}
\|R_{L}^{-}(g_1(t,x), (P_{<L/2^{50}}\ee_{(-\infty, 0]})(x)g_2(t,x))\|_{L^2_{tx}}
&=\|P_{<L/2^{50}}\ee_{(-\infty, 0]}(x)R_L^-(g_1(t,x),g_2(t,x))\|_{L^2_{tx}}\\
&=\|(\chi_{N_j}*\ee_{(-\infty,0]})(x)R_L^-(g_1(t,x),g_2(t,x))\|_{L^2_{tx}}\\
&\lesssim \int_{\R}|\chi_{N_j}(z)|\|\ee_{(-\infty,0]}(x-z)R_L^-(g_1(t,x),g_2(t,x))\|_{L^2_{tx}}dz\\
&\lesssim \|R_L^-(g_1,g_2)\|_{L^2_{tx}}
\end{split}
\]
because $\chi_{N_j}(x)=\F^{-1}_{\xi}[\varphi (2^{50}N_j^{-1}\xi)](x)$. 
We also obtain
\begin{equation}\label{RL_P_est_mi}
\|R_{L}^{-}(g_1(t,x), (P_{<L/2^{50}}\ee_{[0,\infty)})(x)g_2(t,x))\|_{L^2_{tx}}
\lesssim \|R_{L}^{-}(g_1,g_2)\|_{L^2_{tx}}. 
\end{equation}
by the same way. 
By using (\ref{RL_P_est}) and (\ref{RL_P_est_mi}), we have
\[
\begin{split}
\|R_{L}^{-}(e^{it\partial_x^4}u_{1,N_1}(0),B_2)\|_{L^2_{tx}}
&\lesssim \int_{\R}\|R_{L}^{-}(e^{it\partial_x^4}u_{1,N_1}(0), e^{it\partial_x^4}\mathcal{L}v_{2,y})\|_{L^2_{tx}}dy
+\int_{\R}\|R_L^-(e^{it\partial_x^4}u_{1,N_1}(0),e^{it\partial_x^4}P_{+}v_{2,y})\|_{L^2_{tx}}dy\\
&\ \ \ \ 
+\int_{\R}\|R_L^-(e^{it\partial_x^4}u_{1,N_1}(0),e^{it\partial_x^4}P_{-}v_{2,y})\|_{L^2_{tx}}dy
+\int_{\R}\|R_L^-(e^{it\partial_x^4}u_{1,N_1}(0),h_{2,y})\|_{L^2_{tx}}dy\\
&=:I+II+III+IV.
\end{split}
\]
By (\ref{5-3}) and (\ref{vy_est_L}), we obtain
\[
\begin{split}
I+II+III
&\lesssim \int_{\R}N_1^{-1}L^{-\frac{1}{2}}\|u_{1,N_1}(0)\|_{L^2_{x}}\left(\|v_{2,y}\|_{L^2_x}+\|\mathcal{L}v_{2,y}\|_{L^2_x}\right)dy\\
&\lesssim \int_{\R}N_1^{-\frac{3}{2}}N_2^{-\frac{3}{2}}\|u_{1,N_1}(0)\|_{L^2_{x}}\|F_2(t,y)\|_{L^2_t}dy\\
&=N_1^{-1}N_2^{-\frac{3}{2}}L^{-\frac{1}{2}}\|u_{1,N_1}(0)\|_{L^2_{x}}\|F_2\|_{L^1_xL^2_t}. 
\end{split}
\]
While, by the H\"older inequality, Lemma~\ref{kato}, and (\ref{hy_est_L}) with 
$(q,p)=(2,\infty)$, we obtain
\[
\begin{split}
IV&\lesssim \int_{\R}\|e^{it\partial_x^4}u_{1,N_1}(0)\|_{L^{\infty}_xL^2_t}\|h_{2,y}\|_{L^{2}_xL^{\infty}_t}dy\\
&\lesssim \int_{\R}N_1^{-\frac{3}{2}}\|u_{1,N_1}(0)\|_{L^2_x}L^{-\frac{1}{2}}N_2^{-1}\|F(t,y)\|_{L^2_t}dy\\
&\lesssim N_1^{-1}N_2^{-\frac{3}{2}}L^{-\frac{1}{2}}\|u_{1,N_1}(0)\|_{L^2_{x}}\|F_2\|_{L^1_xL^2_t}
\end{split}
\]
since $N_1\ge N_2$. 
Therefore, we get (\ref{AB_est_L_2}). 
\end{proof}
\begin{corollary}\label{bilin_T_cor2}
Let $T>0$, $L,N_1,N_2\in 2^{\Z}$, and $u_1\in X_{N_1}, u_2\in X_{N_2}$. 
If $N_1\ge N_2\gtrsim L$, then the estimates
\begin{align}
\label{bistri_TT_2}
     \left\|R_L^+\left(P_{N_1}u_1,\overline{P_{N_2}u_2}\right)\right\|_{L^2_{T}L^2_x}
     &\lesssim
     T^{\frac{\theta}{4}}N_1^{-1+\frac{\theta}{2}}L^{-\frac{1-\theta}{2}}\|P_{N_1}u_1\|_{X_{N_1}}\|P_{N_2}u_2\|_{X_{N_2}},\\
\label{bistri_TT_3}
     \|R_L^-(P_{N_1}u_1,P_{N_2}u_2)\|_{L^2_{T}L^2_x}
     &\lesssim
     T^{\frac{\theta}{4}}N_1^{-1+\frac{\theta}{2}}L^{-\frac{1-\theta}{2}}N_1^{-1}L^{-\frac{1}{2}}\|P_{N_1}u_1\|_{X_{N_1}}\|P_{N_2}u_2\|_{X_{N_2}}
\end{align}
hold, where the implicit constants are independnet of $T,L,N_1,N_2,u_1,u_2$.
\end{corollary}
Proof is the same as Corollary~\ref{bilin_T_cor1}.
\subsection{Multilinear estimates}
\begin{theorem}[Multilinear estimates]\label{mest_5_cri}
{\rm (i)}\ For any $u_1$, $\cdots$, $u_5\in \dot{X}^{\frac{1}{4}}$, 
it holds that
\begin{equation}\label{multilin_5_hom}
\left\|\partial_x^{3}
\prod_{i=1}^{5}\widetilde{u_i}\right\|_{\dot{Y}^{\frac{1}{4}}}
\lesssim \prod_{i=1}^5\|u_i\|_{\dot{X}^{\frac{1}{4}}},
\end{equation}
where $\widetilde{u_i}\in \{u_i,\overline{u_i}\}$. 
\\
{\rm (ii)}\ If $s\ge \frac{1}{4}$,
then for any $u_1$, $\cdots$, $u_5\in X^s$, 
it holds that
\begin{equation}\label{multilin_5_inhom}
\left\|\partial_x^{3}
\prod_{i=1}^{5}\widetilde{u_i}\right\|_{Y^s}
\lesssim \prod_{i=1}^5\|u_i\|_{X^s},
\end{equation}
where $\widetilde{u_i}\in \{u_i,\overline{u_i}\}$. 
The implicit constant depends only on $s$.
\end{theorem}
\begin{proof}
We define
\[
\begin{split}
I_1&:=\left\{\sum_{N\in 2^{\Z}} N^{2s+6}\left(\sum_{(N_1,\cdots , N_5)\in \Phi_1}\left\|
   \prod_{i=1}^{5}P_{N_i}u_i\right\|_{Z_N}\right)^2\right\}^{\frac{1}{2}},\\
I_k&:=\sum_{N\in 2^{\Z}} N^{s+3}\sum_{(N_1,\cdots , N_5)\in \Phi_k}\left\|
   \prod_{i=1}^{5}P_{N_i}u_i\right\|_{Z_N}\ \ (k=2,3,5),\\
I_4&:=\sum_{N\in 2^{\Z}} N^{s+3}\sum_{(N_1,\cdots , N_5)\in \Phi_4}\left\|
   P_N\left(\prod_{i=1}^{5}P_{N_i}\widetilde{u_i}\right)\right\|_{Y_N},
\end{split}
\]
where
\[
\begin{split}
\Phi_1&:=\{(N_1,\cdots,N_5)|\ N_1\ge \cdots \ge N_5,\ N\sim N_1\gg N_{2}\},\\
\Phi_k&:=\{(N_1,\cdots,N_5)|\ N_1\ge \cdots \ge N_5,\ N_1\gtrsim N,\ N_1\sim \cdots \sim N_k\gg N_{k+1}\}\ (k=2,3,4),\\
\Phi_5&:=\{(N_1,\cdots,N_5)|\ N_1\ge \cdots \ge N_5,\ N\lesssim N_1\sim \cdots \sim N_5\}.
\end{split}
\]
We note that
\[
\left\|\partial_x^{3}
\prod_{i=1}^{5}\widetilde{u_i}\right\|_{\dot{Y}^{s}}
\lesssim \sum_{k=1}^5I_k
\]
holds. 
Let $s\ge 0$. 
For $i=1,\cdots, m$ and $N_i\in 2^{\Z}$, we set
$c_{1,N_1}:=N_1^s\|P_{N_1}u_1\|_{X_{N_1}}$,\ 
$c_{i,N_i}:=N_i^{\frac{1}{4}}\|P_{N_i}u_i\|_{X_{N_i}}$ $(i=2,\cdots ,5)$. 

Now, we consider $I_k$ for $k\ne 4$. 
We assume $u_i\in \dot{X}^s\cap \dot{X}^{\frac{1}{4}}$ and 
$\|u_i\|_{\dot{X}^{\frac{1}{4}}}\lesssim 1$ ($i=1,\cdots, 5$). 
Then, it holds that
\begin{equation}\label{ci_X_norm}
\sum_{N_i}c_{i,N_i}^2=\|u_i\|_{\dot{X}^{\frac{1}{4}}}^2\lesssim 1
\end{equation}
for $i=2,\cdots 5$. 
We prove
\begin{equation}\label{multilin_5_2}
I_k
\lesssim \sum_{i=1}^5\|u_i\|_{\dot{X}^s}.
\end{equation}
This implies 
\begin{equation}\label{multilin_5_3}
I_k
\lesssim \sum_{i=1}^5\|u_i\|_{\dot{X}^s}\prod_{\substack{1\le k\le 5\\ k\ne i}}\|u_k\|_{\dot{X}^{\frac{1}{4}}}
\end{equation}
for any $u_i\in \dot{X}^s\cap \dot{X}^{\frac{1}{4}}$ 
because $\overline{u_i}$ $(i=1,\cdots m)$ do not appear in 
the definition of $I_k$ when $k\ne 4$. 
To obtain (\ref{multilin_5_2}), it suffices to show 
\begin{equation}\label{m5_multiest_desired}
    I_k\lesssim \left(\sum_{N_1}c_{1,N_1}^2\right)^{\frac{1}{2}}.
\end{equation}

We first prove (\ref{m5_multiest_desired}) for $k=5$. 
By the H\"older inequality and  Proposition~\ref{lin_XN}, we have
\[
\begin{split}
\left\|\prod_{i=1}^5P_{N_i}u_i\right\|_{Z_N}
&= 
N^{-\frac{3}{2}}\left\|\prod_{i=1}^5P_{N_i}u_i\right\|_{L^1_xL^2_t}\\
&\le N^{-\frac{3}{2}}\|P_{N_1}u_1\|_{L^{\infty}_{x}L^2_t}\|P_{N_2}u_2\|_{L^4_xL^{\infty}_t}
\|P_{N_3}u_3\|_{L^4_xL^{\infty}_t}\|P_{N_4}u_4\|_{L^4_xL^{\infty}_t}\|P_{N_5}u_5\|_{L^4_xL^{\infty}_t}\\
&\lesssim N^{-\frac{3}{2}}N_1^{-\frac{3}{2}}N_2^{\frac{1}{4}}N_3^{\frac{1}{4}}N_4^{\frac{1}{4}}N_5^{\frac{1}{4}}
\prod_{i=1}^5\|P_{N_i}u_i\|_{X_{N_i}}\\
&\sim N^{-\frac{3}{2}}N_1^{-s-\frac{3}{2}}\prod_{i=1}^5c_{i,N_i}.
\end{split}
\]
Therefore, by the Cauchy-Schwarz inequality for dyadic summation, we obtain
\[
\begin{split}
I_5&\lesssim \sum_N\sum_{(N_1,\cdots , N_5)\in \Phi_5}N^{s+3}
\left\|\prod_{i=1}^5P_{N_i}u_i\right\|_{Z_N}\\
&\lesssim \sum_{N_1}\sum_{N\lesssim N_1}\sum_{N_2\sim N_1}\sum_{N_3\sim N_1}
\sum_{N_4\sim N_1}\sum_{N_5\sim N_1}
N^{s+3}N^{-\frac{3}{2}}N_1^{-s-\frac{3}{2}}\prod_{i=1}^5c_{i,N_i}\\
&\lesssim \sum_{N_1}\sum_{N_2\sim N_1}\sum_{N_3\sim N_1}
\sum_{N_4\sim N_1}\sum_{N_5\sim N_1}\prod_{i=1}^5c_{i,N_i}\\
&\lesssim \prod_{i=1}^5\left(\sum_{N_i}c_{i,N_i}^2\right)^{\frac{1}{2}}.
\end{split}
\]
This implies (\ref{m5_multiest_desired}) for $k=5$ by (\ref{ci_X_norm}).

Next, we consider $I_k$ for $k=1,2,3$. 
By the same argument as in the proof of Theorem~\ref{multi_est_3m}, 
to obtain (\ref{m5_multiest_desired}), it suffices to show that
\begin{equation}\label{m5_multiest_desired_red}
\left\|\prod_{i=1}^5P_{N_i}u_i\right\|_{L^1_xL^2_t}
\lesssim N_1^{-s-\frac{3}{2}}
\left(\frac{N_5}{N_{k+1}}\right)^{\frac{1}{4}}\prod_{i=1}^5c_{i,N_i}
\end{equation}
for $(N_1,\cdots,N_5)\in \Phi_k$. 
By the H\"older inequality, (\ref{bistri_3}) with $L=N_1$, 
Bernstein inequality, and Proposition~\ref{lin_XN}, we have
\[
\begin{split}
\left\|\prod_{i=1}^5P_{N_i}u_i\right\|_{L^1_xL^2_t}
&\le \|P_{N_1}u_1P_{N_{k+1}}u_{k+1}\|_{L^2_{x,t}}\left(\prod_{\substack{2\le i\le 4\\ i\ne k+1}}
\|P_{N_i}u_i\|_{L^4_xL^{\infty}_t}\right)\|P_{N_5}u_5\|_{L^{\infty}_{t,x}}\\
&\lesssim N_1^{-\frac{3}{2}}\left(\prod_{\substack{2\le i\le 4\\ i\ne k+1}}
N_i^{\frac{1}{4}}\right)N_5^{\frac{1}{2}}
\prod_{i=1}^5\|P_{N_i}u_i\|_{X_{N_i}}\\
&\sim N_1^{-s-\frac{3}{2}}N_{k+1}^{-\frac{1}{4}}N_5^{\frac{1}{4}}\prod_{i=1}^5c_{i,N_i}.
\end{split}
\]
Therefore, we obtain (\ref{m5_multiest_desired_red}). 

Now, we consider $I_4$. We prove
\begin{equation}\label{m5_multiest_desired_k4}
    I_4\lesssim \prod_{i=1}^5\left(\sum_{N_i}c_{i,N_i}^2\right)^{\frac{1}{2}}.
\end{equation}
We put 
\[
\begin{split}
\Phi_{4,1}&:=\{(N_1,\cdots, N_5)\in \Phi_4|N\lesssim N_5\},\\
\Phi_{4,2}&:=\{(N_1,\cdots, N_5)\in \Phi_4|N\gg N_5\}, 
\end{split}
\]
and
\[
I_{4,l}:=\sum_{N\in 2^{\Z}} N^{s+3}\sum_{(N_1,\cdots , N_5)\in \Phi_{4,l}}\left\|
  P_N\left(\prod_{i=1}^{5}P_{N_i}\widetilde{u_i}\right)\right\|_{Y_N}\ \ (l=1,2).
\]
\\

\noindent \underline{(i)\ For $l=1$\ ($N\lesssim N_5$)}

By the definition of the norm $\|\cdot\|_{Y_N}$, the H\"older inequality, (\ref{bistri_3}) with $L=N_1$, 
Bernstein inequality, and Proposition~\ref{lin_XN}, we have
\[
\begin{split}
\left\|P_N\left(\prod_{i=1}^5P_{N_i}\widetilde{u_i}\right)\right\|_{Y_N}
&\le 
N^{-\frac{3}{2}}\left\|\prod_{i=1}^5P_{N_i}u_i\right\|_{L^1_xL^2_t}\\
&\le N^{-\frac{3}{2}}\|P_{N_1}u_1P_{N_5}u_5\|_{L^2_{x,t}}\|P_{N_2}u_2\|_{L^4_xL^{\infty}_t}
\|P_{N_3}u_3\|_{L^4_xL^{\infty}_t}\|P_{N_4}u_4\|_{L^{\infty}_{x,t}}\\
&\lesssim N^{-\frac{3}{2}}N_1^{-\frac{3}{2}}N_2^{\frac{1}{4}}N_3^{\frac{1}{4}}N_4^{\frac{1}{2}}
\prod_{i=1}^5\|P_{N_i}u_i\|_{X_{N_i}}\\
&\sim N^{-\frac{3}{2}}N_1^{-s-\frac{3}{2}}N_4^{\frac{1}{4}}N_5^{-\frac{1}{4}}\prod_{i=1}^5c_{i,N_i}.
\end{split}
\]
Therefore, we obtain
\[
\begin{split}
I_{4,1}
&\lesssim \sum_{N_1}\sum_{N\lesssim N_5}\sum_{N_2\sim N_1}\sum_{N_3\sim N_1}
\sum_{N_4 \sim N_1}\sum_{N_5\ll N_1}N^{s+\frac{3}{2}}N_1^{-s-\frac{3}{2}}N_4^{\frac{1}{4}}N_5^{-\frac{1}{4}}
\prod_{i=1}^5c_{i,N_i}\\
&\lesssim \sum_{N_1}\sum_{N_2\sim N_1}\sum_{N_3\sim N_1}\sum_{N_4\sim N_1}
c_{1,N_1}c_{2,N_2}c_{3,N_3}c_{4,N_4}
\sum_{N_5 \ll N_1}N_1^{-s-\frac{5}{4}}N_5^{s+\frac{5}{4}}c_{5,N_5}\\\
&\lesssim \prod_{i=1}^5\left(\sum_{N_i}c_{i,N_i}^2\right)^{\frac{1}{2}} 
\end{split}
\]
by the Cauchy-Schwarz inequality for dyadic summation. 
\\

\noindent \underline{(ii)\ For $l=1$\ ($N\gg N_5$)}

We put $P_{N_i}^{\pm}:=P_{\pm}P_{N_i}$, 
where $P_+$ and $P_{-}$ are 
defined in (\ref{proj_plus_minus}). 
Then we have
\[
\prod_{i=1}^5P_{N_i}\widetilde{u_i}=
\left\{\prod_{i=1}^4(P_{N_i}^+\widetilde{u_i}+P_{N_i}^-\widetilde{u_i})\right\}
P_{N_5}\widetilde{u_5}. 
\]
We define
\[
K:=\#\{i\in \{1,2,3,4\}|\ \widetilde{u_i}=\overline{u_i}\}. 
\]
We only have to consider the cases $K=0,1,2$ 
because the case $K=3,4$ can be treated by the same way 
of the case $K=1,0$, respectively.  
\\

\noindent \underline{{\bf Case\ 1.}\ $K=0$}

By the symmetry, we only have to prove the estimates for
\[
\begin{split}
  J_{+-}&:= P_{N_1}^+u_1P_{N_2}^-u_2P_{N_3}u_3P_{N_4}u_4P_{N_5}\widetilde{u_5},\\
  J_{++}&:= P_{N_1}^+u_1P_{N_2}^+u_2P_{N_3}^+u_3P_{N_4}^+u_4P_{N_5}\widetilde{u_5},\\
  J_{--}&:= P_{N_1}^-u_1P_{N_2}^-u_2P_{N_3}^-u_3P_{N_4}^-u_4P_{N_5}\widetilde{u_5}.
\end{split}
\]
We first prove the estimate for $J_{+-}$. 
Because $\xi_1$ and $\xi_2$ are opposite sign 
for $\xi_1\in {\rm supp}\F_x[P_{N_1}^+u_1]$ and $\xi_2\in {\rm supp}\F_x[P_{N_2}^-u_2]$, 
it holds that
\[
|\xi_1-\xi_2|=\xi_1+(-\xi_2)\ge \max\{|\xi_1|, |\xi_2|\}\sim N_1.
\]
This implies that
\[
\|P_{N_1}^+u_1P_{N_2}^-u_2\|_{L^2_{t,x}}
\lesssim N_1^{-\frac{3}{2}}\|P_{N_1}^+u_1\|_{X_{N_1}}\|P_{N_2}^-u_2\|_{X_{N_2}}
\]
by (\ref{bistri_3}) with $L=N_1$. 
Therefore, by the definition of the norm $\|\cdot\|_{Y_N}$, the H\"older inequality, the Bernstein inequality and Proposition~\ref{lin_XN}, we obtain
\[
\begin{split}
\left\|P_NJ_{+-}\right\|_{Y_N}
&\le 
N^{-\frac{3}{2}}\left\|J_{+-}\right\|_{L^1_xL^2_t}\\
&\le N^{-\frac{3}{2}}\|P_{N_1}^+u_1P_{N_2}^-u_2\|_{L^2_{t,x}}\|P_{N_3}u_3\|_{L^4_xL^{\infty}_t}
\|P_{N_4}u_4\|_{L^4_xL^{\infty}_t}\|P_{N_5}u_5\|_{L^{\infty}_{x,t}}\\
&\lesssim N^{-\frac{3}{2}}N_1^{-\frac{3}{2}}N_3^{\frac{1}{4}}N_4^{\frac{1}{4}}N_5^{\frac{1}{2}}
\prod_{i=1}^5\|P_{N_i}u_i\|_{X_{N_i}}\\
&\sim N^{-\frac{3}{2}}N_1^{-s-\frac{3}{2}}N_2^{-\frac{1}{4}}N_5^{\frac{1}{4}}\prod_{i=1}^5c_{i,N_i}.
\end{split}
\]
Because $N_1\sim N_2$, 
we have
\begin{equation}\label{Jpm_est}
N^{s+3}\left\|P_NJ_{+-}\right\|_{Y_N}
\lesssim \left(\frac{N}{N_1}\right)^{s+\frac{3}{2}}
\left(\frac{N_5}{N_1}\right)^{\frac{1}{4}}\prod_{i=1}^5c_{i,N_i}.
\end{equation}

Next, we prove the estimates for $J_{++}$. 
We put
\[
J_{++}^{\rm high}:=\sum_{A\gtrsim N_1^4}Q_AJ_{++},\ \ 
J_{++}^{\rm low}:=\sum_{A\ll N_1^4}Q_AJ_{++}. 
\]
We first consider $J_{++}^{\rm high}$. 
Because
\[
\|J_{++}^{\rm high}\|_{\dot{X}^{0,-\frac{1}{2},1}}
\sim \sum_{A\gtrsim N_1^4}A^{-\frac{1}{2}}\|Q_AJ_{++}^{\rm high}\|_{L^2_{tx}}
\lesssim N_1^{-2}\|J_{++}\|_{L^2_{tx}},
\]
by the definition of the norm $\|\cdot\|_{Y_N}$, the H\"older inequality, the Bernstein inequality and Proposition~\ref{lin_XN}, we obtain
\begin{equation}\label{Jpph_est}
\begin{split}
\|P_NJ_{++}^{\rm high}\|_{Y_N}
&\lesssim \|J_{++}^{\rm high}\|_{\dot{X}^{0,-\frac{1}{2},1}}\\
&\lesssim N_1^{-2}\|P_{N_1}^+u_1\|_{L^{\infty}_xL^{2}_t}
\|P_{N_2}^+u_2\|_{L^{4}_xL^{\infty}_t}\|P_{N_3}^+u_3\|_{L^{4}_xL^{\infty}_t}
\|P_{N_4}^+u_4\|_{L^{\infty}_{x,t}}\|P_{N_5}u_5\|_{L^{\infty}_{x,t}}\\
&\lesssim N_1^{-2}N_1^{-\frac{3}{2}}N_2^{\frac{1}{4}}N_3^{\frac{1}{4}}
N_4^{\frac{1}{2}}N_5^{\frac{1}{2}}\prod_{i=1}^5\|P_{N_i}u_i\|_{X_{N_i}}\\
&\sim N_1^{-\frac{7}{2}-s}N_4^{\frac{1}{4}}N_5^{\frac{1}{4}}\prod_{i=1}^5c_{i,N_i}.
\end{split}
\end{equation}

Finally, we consider $J_{++}^{\rm low}$. 
Let $(\tau_i,\xi_i)\in {\rm supp}\F_{t,x}[P_{N_i}^+u_i]$ $(i=1,2,3,4)$, 
$(\tau_5,\xi_5)\in {\rm supp}\F_{t,x}[P_{N_5}u_5]$. 
Then,
\[
|\tau_1+\tau_2+\tau_3+\tau_4\pm \tau_5
-(\xi_1+\xi_2+\xi_3+\xi_4\pm \xi_5)^4|
\ll N_1^4
\]
since ${\rm supp}\F_{t,x}[J_{++}^{\rm low}]\subset \{(\tau,\xi)|\ |\tau-\xi^4|\ll N_1^4\}$. 
Because $N_1\sim N_2\sim N_3\sim N_4\gg N_5$, 
there exist $C_i>0$ and $r_i\in \R$ satisfying $|r_i|\ll N_1$, 
such that
\[
\xi_i=C_iN_1+r_i\ \ (i=1,2,3,4),\ \ \xi_5=r_5.
\]
This implies
\begin{equation}\label{modulation_bd}
\begin{split}
&|(\xi_1+\xi_2+\xi_3+\xi_4\pm \xi_5)^4-(\xi_1^4+\xi_2^4+\xi_3^4+\xi_4^4\pm \xi_5^4)|\\
&\sim |(C_1+C_2+C_3+C_4)^4N_1^4-(C_1^4+C_2^4+C_3^4+C_4^4)N_1^4|\\
&\sim  N_1^4
\end{split}
\end{equation}
since $(C_1+C_2+C_3+C_4)^4>C_1^4+C_2^4+C_3^4+C_4^4$. 
Therefore, 
at least one of $\tau_i-\xi_i^4$ $(i=1,2,3,4)$ and $\tau_5 \mp \xi_5^4$ is larger than $N_1^4$. 
By the symmetry, we can assume 
${\rm supp}\F_{t,x}[P_{N_1}^+u_1]\subset \{(\tau,\xi)|\ |\tau-\xi^4|\gtrsim N_1^4\}$ 
or ${\rm supp}\F_{t,x}[P_{N_5}u_5]\subset \{(\tau,\xi)|\ |\tau\mp\xi^4|\gtrsim N_1^4\}$. 
For the former case, we have
\[
\|P_{N_1}^+u_1\|_{L^2_{tx}}
\lesssim \sum_{A\gtrsim N_1^4}\|Q_AP_{N_1}^+u_1\|_{L^2_{tx}}
\lesssim N_1^{-2}\sup_{A\in 2^\Z}A^{\frac{1}{2}}\|Q_AP_{N_1}^+u_1\|_{L^2_{tx}}
=N_1^{-2}\|P_{N_1}^+u_1\|_{\dot{X}^{0,\frac{1}{2},\infty}}
\lesssim N_1^{-2}\|P_{N_1}^+u_1\|_{X_{N_1}}
\]
by Proposition~\ref{dual_besov}. 
Therefore, by the definition of the norm $\|\cdot\|_{Y_N}$, the H\"older inequality, 
Bernstein inequality, and Proposition~\ref{lin_XN}, we obtain
\begin{equation}\label{Jppl_est_1}
\begin{split}
\left\|P_NJ_{++}^{\rm low}\right\|_{Y_N}
&\le 
N^{-\frac{3}{2}}\left\|J_{++}^{\rm low}\right\|_{L^1_xL^2_t}\\
&\le N^{-\frac{3}{2}}\|P_{N_1}^+u_1\|_{L^2_{x,t}}
\|P_{N_2}^+u_2\|_{L^4_xL^{\infty}_t}\|P_{N_3}^+u_3\|_{L^4_xL^{\infty}_t}
\|P_{N_4}^+u_4\|_{L^{\infty}_{x,t}}\|P_{N_5}u_5\|_{L^{\infty}_{x,t}}\\
&\lesssim N^{-\frac{3}{2}}N_1^{-2}N_2^{\frac{1}{4}}N_3^{\frac{1}{4}}N_4^{\frac{1}{2}}
N_5^{\frac{1}{2}}
\prod_{i=1}^5\|P_{N_i}u_i\|_{X_{N_i}}\\
&\sim N^{-\frac{3}{2}}N_1^{-s-2}N_4^{\frac{1}{4}}N_5^{\frac{1}{4}}\prod_{i=1}^5c_{i,N_i}.
\end{split}
\end{equation}
For the later case, we have
\[
\|P_{N_5}\widetilde{u_5}\|_{L^2_{t,x}}
\lesssim \sum_{A\gtrsim N_1^4}\|Q_AP_{N_5}u_5\|_{L^2_{t,x}}
\lesssim N_1^{-2}\sup_{A\in 2^\Z}A^{\frac{1}{2}}\|Q_AP_{N_5}u_5\|_{L^2_{t,x}}
=N_1^{-2}\|P_{N_5}u_5\|_{\dot{X}^{0,\frac{1}{2},\infty}}
\lesssim N_1^{-2}\|P_{N_5}u_5\|_{X_{N_5}}
\]
by Proposition~\ref{dual_besov}. 
Now, we used
\[
\|\psi_{A}(\tau +\xi^4)\F_{tx}[P_{N_5}\overline{u_5}]\|_{L^2_{t,x}}
\sim \|\psi_A(\tau -\xi^4)\F_{tx}[P_{N_5}u_5]\|_{L^2_{\tau,\xi}}
\]
when $\widetilde{u_5}=\overline{u_5}$ and $|\tau_5+\xi_5^4|\gtrsim N_1^4$. 
Therefore, by the H\"older inequality, 
the Sobolev inequality, and the Bernstein inequality, we have
\[
\|P_{N_3}^+u_3P_{N_4}^+u_4P_{N_5}\widetilde{u_5}\|_{L^2_{t,x}}
\lesssim \|P_{N_3}^+u_3\|_{L^{\infty}_tL^4_x}\|P_{N_4}^+u_4\|_{L^{\infty}_tL^4_x}
\|P_{N_5}\widetilde{u_5}\|_{L^{2}_tL^{\infty}_x}
\lesssim N_3^{\frac{1}{4}}N_4^{\frac{1}{4}}N_5^{\frac{1}{2}}N_1^{-2}
\prod_{i=3}^5\|P_{N_i}u_i\|_{X_{N_i}}. 
\]
Therefore, by the definition of the norm $\|\cdot\|_{Y_N}$, the H\"older inequality, 
and Proposition~\ref{lin_XN}, we obtain
\begin{equation}\label{Jppl_est_2}
\begin{split}
\left\|P_NJ_{++}^{\rm low}\right\|_{Y_N}
&\le 
N^{-\frac{3}{2}}\left\|J_{++}^{\rm low}\right\|_{L^1_xL^2_t}\\
&\le N^{-\frac{3}{2}}\|P_{N_1}^+u_1\|_{L^4_{x}L^{\infty}_t}
\|P_{N_2}^+u_2\|_{L^4_xL^{\infty}_t}\|P_{N_3}^+u_3P_{N_4}^+u_4P_{N_5}\widetilde{u_5}\|_{L^2_{x,t}}\\
&\lesssim N^{-\frac{3}{2}}N_1^{\frac{1}{4}}N_2^{\frac{1}{4}}
N_3^{\frac{1}{4}}N_4^{\frac{1}{4}}N_5^{\frac{1}{2}}N_1^{-2}
\prod_{i=1}^5\|P_{N_i}u_i\|_{X_{N_i}}\\
&\sim N^{-\frac{3}{2}}N_1^{-s-\frac{7}{4}}N_5^{\frac{1}{4}}\prod_{i=1}^5c_{i,N_i}.
\end{split}
\end{equation}
Because $N_1\sim N_4$ and $N\lesssim N_1$, 
We have
\begin{equation}\label{Jpp_est}
N^{s+3}\left\|P_NJ_{++}\right\|_{Y_N}
\lesssim \left\{\left(\frac{N}{N_1}\right)^{s+3}
+
\left(\frac{N}{N_1}\right)^{s+\frac{3}{2}}\right\}
\left(\frac{N_5}{N_1}\right)^{\frac{1}{4}}\prod_{i=1}^5c_{i,N_i}
\lesssim \left(\frac{N}{N_1}\right)^{s+\frac{3}{2}}
\left(\frac{N_5}{N_1}\right)^{\frac{1}{4}}\prod_{i=1}^5c_{i,N_i}.
\end{equation}
by (\ref{Jpph_est}), (\ref{Jppl_est_1}), and (\ref{Jppl_est_2}).
By the same argument, we obtain 
\begin{equation}\label{Jmm_est}
\left\|P_NJ_{--}\right\|_{Y_N}
\lesssim \left(\frac{N}{N_1}\right)^{s+\frac{3}{2}}
\left(\frac{N_5}{N_1}\right)^{\frac{1}{4}}\prod_{i=1}^5c_{i,N_i}.
\end{equation}
As a result, by (\ref{Jpm_est}), (\ref{Jpp_est}), and (\ref{Jmm_est}), we have
\[
\begin{split}
I_{4,2}&\lesssim \sum_N\sum_{(N_1,\cdots , N_5)\in \Phi_{4,2}}N^{s+3}
\left(\left\|P_NJ_{+-}\right\|_{Y_N}+\left\|P_NJ_{++}\right\|_{Y_N}+\left\|P_NJ_{--}\right\|_{Y_N}
\right)\\
&\lesssim \sum_{N_1}\sum_{N\lesssim N_1}\sum_{N_2\sim N_1}\sum_{N_3\sim N_1}
\sum_{N_4 \sim N_1}\sum_{N_5\ll N_1}
 \left(\frac{N}{N_1}\right)^{s+\frac{3}{2}}
\left(\frac{N_5}{N_1}\right)^{\frac{1}{4}}
\prod_{i=1}^5c_{i,N_i}\\
&\lesssim \prod_{i=1}^5\left(\sum_{N_i}c_{i,N_i}^2\right)^{\frac{1}{2}}.
\end{split}
\]
\\

\noindent \underline{{\bf Case\ 2.}\ $K=1$}

By the symmetry, we can assume 
$\widetilde{u_i}=u_i$ $(i=1,2,3)$, $\widetilde{u_4}=\overline{u_4}$, 
and we only have to prove the estimates for
\[
\begin{split}
  J_{1\pm}&:= P_{N_1}^{\pm}u_1P_{N_2}u_2P_{N_3}u_3P_{N_4}^{\pm}\overline{u_4}P_{N_5}\widetilde{u_5},\\
  J_{2\pm}&:= P_{N_1}^{\pm}u_1P_{N_2}^{\pm}u_2P_{N_3}^{\pm}u_3P_{N_4}^{\mp}\overline{u_4}P_{N_5}\widetilde{u_5}.
\end{split}
\]
We first prove the estimate for $J_{1\pm}$. 
Because $\xi_1$ and $\xi_4$ are same sign 
for $\xi_1\in {\rm supp}\F_x[P_{N_1}^{\pm}u_1]$ 
and $\xi_4\in {\rm supp}\F_x[P_{N_4}^{\pm}\overline{u_4}]$, it holds that 
\[
|\xi_1+\xi_4|=|\xi_1|+|\xi_4|\ge \max\{|\xi_1|, |\xi_4|\}\sim N_1.
\]
It implies that
\[
\|P_{N_1}^{\pm}u_1P_{N_4}^{\pm}\overline{u_4}\|_{L^2_{t,x}}
\lesssim N_1^{-\frac{3}{2}}\|P_{N_1}^{\pm}u_1\|_{X_{N_1}}\|P_{N_4}^{\pm}u_4\|_{X_{N_4}}
\]
by (\ref{bistri_2}) with $L=N_1$. 
Therefore, we can treat $J_{1\pm}$ by the same way 
for $J_{+-}$ in Case 1. 

Next, we prove the estimates for $J_{2+}$. 
We put
\[
J_{2+}^{\rm high}:=\sum_{A\gtrsim N_1^4}Q_AJ_{2+},\ \ 
J_{+2}^{\rm low}:=\sum_{A\ll N_1^4}Q_AJ_{2+}. 
\]
We have to only consider $J_{2+}^{\rm low}$
because
we can treat $J_{2+}^{\rm high}$ by the same way 
for $J_{++}^{\rm high}$ in Case 1. 
Let $(\tau_i,\xi_i)\in {\rm supp}\F_{t,x}[P_{N_i}^+u_i]$ $(i=1,2,3)$, 
$(\tau_4,\xi_4)\in {\rm supp}\F_{t,x}[P_{N_4}^-u_4]$, 
$(\tau_5,\xi_5)\in {\rm supp}\F_{t,x}[P_{N_5}u_5]$. 
Then,
\[
|\tau_1+\tau_2+\tau_3-\tau_4\pm \tau_5
-(\xi_1+\xi_2+\xi_3-\xi_4\pm \xi_5)^4|
\ll N_1^4
\]
since ${\rm supp}\F_{t,x}[J_{2+}^{\rm low}]\subset \{(\tau,\xi)|\ |\tau-\xi^4|\ll N_1^4\}$. 
Because $N_1\sim N_2\sim N_3\sim N_4\gg N_5$, 
there exist $C_i>0$ and $r_i\in \R$ satisfying $|r_i|\ll N_1$, 
such that
\[
\xi_i=C_iN_1+r_i\ \ (i=1,2,3),\ \ \xi_4=-C_4N_1+r_4,\ \ \xi_5=r_5.
\]
It implies
\[
\begin{split}
&|(\xi_1+\xi_2+\xi_3- \xi_4\pm \xi_5)^4-(\xi_1^4+\xi_2^4+\xi_3^4-\xi_4^4\pm \xi_5^4)|\\
&\sim |(C_1+C_2+C_3+C_4)^4N_1^4-(C_1^4+C_2^4+C_3^4-C_4^4)N_1^4|\\
&\sim  N_1^4
\end{split}
\]
since $(C_1+C_2+C_3+C_4)^4>C_1^4+C_2^4+C_3^4-C_4^4$. 
Therefore, we can treat $J_{2}^{\rm low}$ by the same way 
for $J_{++}^{\rm low}$ in Case 1. 
$J_4$ also can be treated by the same way. \\

\noindent \underline{{\bf Case\ 3.}\ $K=2$}

By the symmetry, we can assume 
$\widetilde{u_i}=u_i$ $(i=1,3)$ and $\widetilde{u_i}=\overline{u_i}$ $(i=2,4)$. 
Because $N\gg N_5$, it holds that
\[
|\xi_1+\xi_2+\xi_3+\xi_4|\sim N
\]
for $\xi_i\in {\rm supp}\F_x[P_{N_i}u_i]$ $(i=1,3)$  
and $\xi_i\in {\rm supp}\F_x[P_{N_i}\overline{u_i}]$ $(i=2,4)$. 
Therefore, at least one of $|\xi_1+\xi_2|$ and $|\xi_3+\xi_4|$ 
is larger than $N$. 
By the symmetry, we can assume $|\xi_1+\xi_2|\gtrsim N$. 
Then, we have
\[
\|P_{N_1}u_1P_{N_2}\overline{u_2}\|_{L^2_{tx}}
\lesssim N_1^{-1}N^{-\frac{1}{2}}\|P_{N_1}u_1\|_{X_{N_1}}\|P_{N_2}u_2\|_{X_{N_2}}
\]
by (\ref{bistri_2}) with $L=N$. 
Therefore, we can treat $I_{4,2}$ by the same way 
for $J_{+-}$ in Case 1. 
Indeed, by the definition of the norm $\|\cdot\|_{Y_N}$, the H\"older inequality, 
and Proposition~\ref{lin_XN}, we obtain
\[
\begin{split}
\left\|P_N\left(\prod_{i=1}^5P_{N_i}\widetilde{u_i}\right)\right\|_{Y_N}
&\le 
N^{-\frac{3}{2}}
\left\|P_N(P_{N_1}u_1P_{N_2}\overline{u_2}P_{N_3}u_3P_{N_4}\overline{u_4}P_{N_5}\widetilde{u_5})
\right\|_{L^1_xL^2_t}\\
&\le N^{-\frac{3}{2}}\|P_{N_1}u_1P_{N_2}\overline{u_2}\|_{L^2_{t,x}}\|P_{N_3}u_3\|_{L^4_xL^{\infty}_t}
\|P_{N_4}u_4\|_{L^4_xL^{\infty}_t}\|P_{N_5}u_5\|_{L^{\infty}_{x,t}}\\
&\lesssim N^{-\frac{3}{2}}N_1^{-1}N^{-\frac{1}{2}}N_3^{\frac{1}{4}}N_4^{\frac{1}{4}}N_5^{\frac{1}{2}}
\prod_{i=1}^5\|P_{N_i}u_i\|_{X_{N_i}}\\
&\sim N^{-2}N_1^{-s-1}N_2^{-\frac{1}{4}}N_5^{\frac{1}{4}}\prod_{i=1}^5c_{i,N_i}.
\end{split}
\]
It implies
\[
I_{4,2}\lesssim \prod_{i=1}^5\left(\sum_{N_i}c_{i,N_i}^2\right)^{\frac{1}{2}}. 
\]
\\

As a result, we obtain (\ref{multilin_5_3}) for $k=4$.
Therefore, we get
\begin{equation}\label{multilin_5_3_fin}
\left\|\partial_x^{3}
\prod_{i=1}^{5}\widetilde{u_i}\right\|_{\dot{Y}^{s}}
\lesssim \sum_{i=1}^5\|u_i\|_{\dot{X}^s}\prod_{\substack{1\le k\le 5\\ k\ne i}}\|u_k\|_{\dot{X}^{\frac{1}{4}}}
\end{equation}
for $s\ge 0$. 
The estimate (\ref{multilin_5_hom}) follows from (\ref{multilin_5_3_fin}) with $s=\frac{1}{4}$. 
The estimate (\ref{multilin_5_inhom}) follows from (\ref{multilin_5_3_fin}) with $s=0$ and $s\ge \frac{1}{4}$. 
\end{proof}
\begin{theorem}[Multilinear estimates]\label{mest_4_cri}
{\rm (i)}\ For any $u_1$, $\cdots$, $u_4\in \dot{X}^{-\frac{1}{6}}$, 
it holds that
\begin{equation}\label{multilin_4_hom}
\left\|\partial_x^{2}
\prod_{i=1}^{4}\widetilde{u_i}\right\|_{\dot{Y}^{-\frac{1}{6}}}
\lesssim \prod_{i=1}^4\|u_i\|_{\dot{X}^{-\frac{1}{6}}},
\end{equation}
where $\widetilde{u_i}\in \{u_i,\overline{u_i}\}$. 
\\
{\rm (ii)}\ If $s\ge -\frac{1}{6}$,
then for any $u_1$, $\cdots$, $u_4\in X^s$, 
it holds that
\begin{equation}\label{multilin_4_inhom}
\left\|\partial_x^{2}
\prod_{i=1}^{4}\widetilde{u_i}\right\|_{Y^s}
\lesssim \prod_{i=1}^4\|u_i\|_{X^s},
\end{equation}
where $\widetilde{u_i}\in \{u_i,\overline{u_i}\}$. 
The implicit constant depends only on $s$.
\end{theorem}
\begin{proof}
Let $s\ge -\frac{1}{6}$. 
For $i=1,\cdots, m$ and $N_i\in 2^{\Z}$, we set
$c_{1,N_1}:=N_1^s\|P_{N_1}u_1\|_{X_{N_1}}$,\ 
$c_{i,N_i}:=N_i^{-\frac{1}{6}}\|P_{N_i}u_i\|_{X_{N_i}}$ $(i=2,3,4)$. 
We define
\[
\begin{split}
I_1&:=\left\{\sum_{N\in 2^{\Z}} N^{2s+4}\left(\sum_{(N_1,\cdots , N_5)\in \Phi_1}\left\|
   \prod_{i=1}^{4}P_{N_i}u_i\right\|_{Y_N}\right)^2\right\}^{\frac{1}{2}},\\
I_k&:=\sum_{N\in 2^{\Z}} N^{s+2}\sum_{(N_1,\cdots , N_5)\in \Phi_k}\left\|
   \prod_{i=1}^{4}P_{N_i}u_i\right\|_{Y_N}\ \ (k=2,3,4),
\end{split}
\]
where
\[
\begin{split}
\Phi_1&:=\{(N_1,\cdots,N_4)|\ N_1\ge \cdots \ge N_4,\ N\sim N_1\gg N_{2}\},\\
\Phi_k&:=\{(N_1,\cdots,N_4)|\ N_1\ge \cdots \ge N_4,\ N_1\gtrsim N,\ N_1\sim \cdots N_k\gg N_{k+1}\}\ (k=2,3),\\
\Phi_4&:=\{(N_1,\cdots,N_4)|\ N_1\ge \cdots \ge N_4,\ N\lesssim N_1\sim \cdots \sim N_4\}.
\end{split}
\]
We prove
\begin{equation}\label{2m4_multiest_desired}
    I_k\lesssim \prod_{i=1}^4\left(\sum_{N_i}c_{i,N_i}^2\right)^{\frac{1}{2}}.
\end{equation}

First, we assume $(N_1,N_2,N_3,N_4)\in \Phi_1\cup \Phi_2\cup \Phi_3$. 
Then, $N_1\gg N_4$. 
By the definition of the norm $\|\cdot\|_{Y_N}$, the H\"older inequality, (\ref{bistri_3}) with $L=N_1$, 
Bernstein inequality, and Proposition~\ref{lin_XN}, we have
\[
\begin{split}
N^{s+2}\left\|\prod_{i=1}^4P_{N_i}u_i\right\|_{Y_N}
&\le 
N^{s+\frac{1}{2}}\left\|\prod_{i=1}^4P_{N_i}u_i\right\|_{L^1_xL^2_t}\\
&\le N^{s+\frac{1}{2}}\|P_{N_1}u_1P_{N_4}u_4\|_{L^2_{x,t}}\|P_{N_2}u_2\|_{L^4_xL^{\infty}_t}
\|P_{N_3}u_3\|_{L^4_xL^{\infty}_t}\\
&\lesssim N^{s+\frac{1}{2}}N_1^{-\frac{3}{2}}N_2^{\frac{1}{4}}N_3^{\frac{1}{4}}
\prod_{i=1}^4\|P_{N_i}u_i\|_{X_{N_i}}\\
&\sim N^{s+\frac{1}{2}}N_1^{-s-\frac{3}{2}}N_2^{\frac{5}{12}}N_3^{\frac{5}{12}}
N_4^{\frac{1}{6}}\prod_{i=1}^4c_{i,N_i}\\
&= \left(\frac{N}{N_1}\right)^{s+\frac{1}{2}}\frac{N_2^{\frac{5}{12}}N_3^{\frac{5}{12}}
N_4^{\frac{1}{6}}}{N_1}\prod_{i=1}^4c_{i,N_i}. 
\end{split}
\]
We note that $s+\frac{1}{2}>0$ for $s\ge -\frac{1}{6}$. 
Therefore, we obtain (\ref{2m4_multiest_desired}) for $k=1,2,3$
by the same argument as in the proof of Theorem~\ref{multi_est_3mg2}. 

Next, we consider $I_4$. 
We put $P_{N_i}^{\pm}:=P_{\pm}P_{N_i}$, 
where $P_+$ and $P_{-}$ are 
defined in (\ref{proj_plus_minus}). 
Then we have
\[
\prod_{i=1}^4P_{N_i}\widetilde{u_i}=
\prod_{i=1}^4(P_{N_i}^+\widetilde{u_i}+P_{N_i}^-\widetilde{u_i}). 
\]
We define
\[
K:=\#\{i\in \{1,2,3,4\}|\ \widetilde{u_i}=\overline{u_i}\}. 
\]
For the cases $K=0,1,3,4$, 
we can obtain (\ref{2m4_multiest_desired}) for $I_4$  
by almost same way to the proof of Theorem~\ref{mest_5_cri}.  
Therefore, we only give the proof for the case $K=2$. 

By the symmetry, we can assume 
$\widetilde{u_i}=u_i$ $(i=1,3)$ and $\widetilde{u_i}=\overline{u_i}$ $(i=2,4)$. 
We only have to prove the estimates for
\[
\begin{split}
  J_{1\pm}&:= P_{N_1}^{\pm}u_1P_{N_2}^{\pm}\overline{u_2}P_{N_3}u_3P_{N_4}\overline{u_4},\\
  J_{2\pm}&:= P_{N_1}^{\pm}u_1P_{N_2}^{\mp}\overline{u_2}P_{N_3}^{\pm}u_3P_{N_4}^{\mp}\overline{u_4}.
\end{split}
\]
We first prove the estimate for $J_{1\pm}$. 
Because $\xi_1$ and $\xi_2$ are same sign 
for $\xi_1\in {\rm supp}\F_x[P_{N_1}^{\pm}u_1]$ 
and $\xi_2\in {\rm supp}\F_x[P_{N_2}^{\pm}\overline{u_2}]$, 
it holds that
\[
|\xi_1+\xi_2|=|\xi_1|+|\xi_2|\ge \max\{|\xi_1|, |\xi_2|\}\sim N_1.
\]
It implies that
\[
\|P_{N_1}^{\pm}u_1P_{N_2}^{\pm}\overline{u_2}\|_{L^2_{t,x}}
\lesssim N_1^{-\frac{3}{2}}\|P_{N_1}^{\pm}u_1\|_{X_{N_1}}\|P_{N_2}^{\pm}u_2\|_{X_{N_2}}
\]
by (\ref{bistri_2}) with $L=N_1$. 
Therefore, by the definition of the norm $\|\cdot\|_{Y_N}$, the H\"older inequality, 
Bernstein inequality, and Proposition~\ref{lin_XN}, we obtain
\[
\begin{split}
N^{s+2}\left\|P_NJ_{1\pm}\right\|_{Y_N}
&\le 
N^{s+\frac{1}{2}}\left\|J_{1\pm}\right\|_{L^1_xL^2_t}\\
&\le N^{s+\frac{1}{2}}\|P_{N_1}^{\pm}u_1P_{N_2}^{\pm}\overline{u_2}\|_{L^2_{x,t}}\|P_{N_3}u_3\|_{L^4_xL^{\infty}_t}
\|P_{N_4}u_4\|_{L^4_xL^{\infty}_t}\\
&\lesssim N^{s+\frac{1}{2}}N_1^{-\frac{3}{2}}N_3^{\frac{1}{4}}N_4^{\frac{1}{4}}
\prod_{i=1}^4\|P_{N_i}u_i\|_{X_{N_i}}\\
&\sim N^{s+\frac{1}{2}}N_1^{-s-\frac{3}{2}}N_2^{\frac{1}{6}}N_3^{\frac{5}{12}}N_4^{\frac{5}{12}}
\prod_{i=1}^4c_{i,N_i}.
\end{split}
\]
Because $N_1\sim N_2\sim N_3\sim N_4$, 
we have
\begin{equation}\label{Jpm_est_24}
N^{s+2}\left\|P_NJ_{1\pm}\right\|_{Y_N}
\lesssim \left(\frac{N}{N_1}\right)^{s+\frac{1}{2}}\prod_{i=1}^4c_{i,N_i}.
\end{equation}

Next, we prove the estimates for $J_{2+}$. 
We put
\[
J_{2+}^{\rm high}:=\sum_{A\gtrsim N_1^4}Q_AJ_{2+},\ \ 
J_{2+}^{\rm low}:=\sum_{A\ll N_1^4}Q_AJ_{2+}. 
\]
We first consider $J_{2+}^{\rm high}$. 
Because
\[
\|J_{2+}^{\rm high}\|_{\dot{X}^{0,-\frac{1}{2},1}}
\sim \sum_{A\gtrsim N_1^4}A^{-\frac{1}{2}}\|Q_AJ_{2+}^{\rm high}\|_{L^2_{t,x}}
\lesssim N_1^{-2}\|J_{2+}\|_{L^2_{t,x}},
\]
by the definition of the norm $\|\cdot\|_{Y_N}$, the H\"older inequality, 
Bernstein inequality, and Proposition~\ref{lin_XN}, we obtain
\[
\begin{split}
N^{s+2}\|P_NJ_{2+}^{\rm high}\|_{Y_N}
&\lesssim \|J_{2+}^{\rm high}\|_{\dot{X}^{0,-\frac{1}{2},1}}\\
&\lesssim N^{s+2}N_1^{-2}\|P_{N_1}^+u_1\|_{L^{\infty}_xL^{2}_t}
\|P_{N_2}^-u_2\|_{L^{4}_xL^{\infty}_t}\|P_{N_3}^+u_3\|_{L^{4}_xL^{\infty}_t}
\|P_{N_4}^-u_4\|_{L^{\infty}_{x,t}}\\
&\lesssim N^{s+2}N_1^{-2}N_1^{-\frac{3}{2}}N_2^{\frac{1}{4}}N_3^{\frac{1}{4}}
N_4^{\frac{1}{2}}\prod_{i=1}^4\|P_{N_i}u_i\|_{X_{N_i}}\\
&\lesssim N^{s+2}N_1^{-\frac{7}{2}-s}N_2^{\frac{5}{12}}N_3^{\frac{5}{12}}N_4^{\frac{2}{3}}\prod_{i=1}^4c_{i,N_i}.
\end{split}
\]
Because $N_1\sim N_2\sim N_3\sim N_4$, we have
\begin{equation}\label{Jpph_est_24}
N^{s+2}\|P_NJ_{2+}^{\rm high}\|_{Y_N}\lesssim 
\left(\frac{N}{N_1}\right)^{s+2}\prod_{i=1}^4c_{i,N_i}.
\end{equation}

Finally, we consider $J_{2+}^{\rm low}$. 
Let $(\tau_i,\xi_i)\in {\rm supp}\F_{t,x}[P_{N_i}^{+}u_i]$ $(i=1,3)$, 
$(\tau_i,\xi_i)\in {\rm supp}\F_{t,x}[P_{N_i}^{-}u_i]$ $(i=2,4)$. 
Then,
\[
|\tau_1-\tau_2+\tau_3-\tau_4
-(\xi_1-\xi_2+\xi_3-\xi_4)^4|
\ll N_1^4
\]
since ${\rm supp}\F_{t,x}[J_{2+}^{\rm low}]\subset \{(\tau,\xi)|\ |\tau-\xi^4|\ll N_1^4\}$. 
Because $N_1\sim N_2\sim N_3\sim N_4$, 
there exist $C_i>0$ and $r_i\in \R$ satisfying $|r_i|\ll N_1$, 
such that
\[
\xi_i=C_iN_1+r_i\ \ (i=1,3),\ \ 
\xi_i=-C_iN_1+r_i\ \ (i=2,4).
\]
This implies
\[
\begin{split}
&|(\xi_1-\xi_2+\xi_3+\-\xi_4)^4-(\xi_1^4-\xi_2^4+\xi_3^4-\xi_4^4)|\\
&\sim |(C_1+C_2+C_3+C_4)^4N_1^4-(C_1^4-C_2^4+C_3^4-C_4^4)N_1^4|\\
&\sim  N_1^4
\end{split}
\]
since $(C_1+C_2+C_3+C_4)^4>C_1^4-C_2^4+C_3^4-C_4^4$. 
Therefore, 
at least one of $\tau_i-\xi_i^4$ $(i=1,3)$ and  
$\tau_i+\xi_i^4$ $(i=2,4)$ is larger than $N_1^4$. 
If 
${\rm supp}\F_{t,x}[P_{N_i}^+u_i]\subset \{(\tau,\xi)|\ |\tau-\xi^4|\gtrsim N_1^4\}$ $(i=1,3)$, 
then, we have
\[
\|P_{N_i}^+u_i\|_{L^2_{t,x}}
\lesssim \sum_{A\gtrsim N_1^4}\|Q_AP_{N_i}u_i\|_{L^2_{t,x}}
\lesssim N_1^{-2}\sup_{A\in 2^\Z}A^{\frac{1}{2}}\|Q_AP_{N_i}u_i\|_{L^2_{t,x}}
=N_1^{-2}\|P_{N_i}u_i\|_{\dot{X}^{0,\frac{1}{2},\infty}}
\lesssim N_1^{-2}\|P_{N_i}u_i\|_{X_{N_1}}
\]
by Proposition~\ref{dual_besov}. 
If 
${\rm supp}\F_{t,x}[P_{N_i}^-u_i]\subset \{(\tau,\xi)|\ |\tau+\xi^4|\gtrsim N_1^4\}$ $(i=2,4)$, 
then, we have
\[
\|P_{N_i}^-\overline{u_i}\|_{L^2_{tx}}
\lesssim \sum_{A\gtrsim N_1^4}\|Q_AP_{N_i}u_i\|_{L^2_{tx}}
\lesssim N_1^{-2}\sup_{A\in 2^\Z}A^{\frac{1}{2}}\|Q_AP_{N_i}u_i\|_{L^2_{tx}}
=N_1^{-2}\|P_{N_i}u_i\|_{\dot{X}^{0,\frac{1}{2},\infty}}
\lesssim N_1^{-2}\|P_{N_i}u_i\|_{X_{N_1}}
\]
by Proposition~\ref{dual_besov}. Now, we used
\[
\|\psi_{A}(\tau +\xi^4)\F_{tx}[P_{N_i}^-\overline{u_i}]\|_{L^2_{t,x}}
\sim \|\psi_A(\tau -\xi^4)\F_{tx}[P_{N_i}^-u_i]\|_{L^2_{\tau,\xi}}.
\]
We only assume the case 
${\rm supp}\F_{t,x}[P_{N_1}^+u_1]\subset \{(\tau,\xi)|\ |\tau-\xi^4|\gtrsim N_1^4\}$ 
because the other cases are same. 
By the definition of the norm $\|\cdot\|_{Y_N}$, the H\"older inequality, the Bernstein inequality and Proposition~\ref{lin_XN}, we obtain
\[
\begin{split}
N^{s+2}\left\|P_NJ_{2+}^{\rm low}\right\|_{Y_N}
&\le 
N^{s+2}N^{-\frac{3}{2}}\left\|J_{2+}^{\rm low}\right\|_{L^1_xL^2_t}\\
&\le N^{s+\frac{1}{2}}\|P_{N_1}^+u_1\|_{L^2_{x,t}}
\|P_{N_2}^-u_2\|_{L^4_xL^{\infty}_t}\|P_{N_3}^+u_3\|_{L^4_xL^{\infty}_t}
\|P_{N_4}^-u_4\|_{L^{\infty}_{x,t}}\\
&\lesssim N^{s+\frac{1}{2}}N_1^{-2}N_2^{\frac{1}{4}}N_3^{\frac{1}{4}}N_4^{\frac{1}{2}}
\prod_{i=1}^4\|P_{N_i}u_i\|_{X_{N_i}}\\
&\sim N^{s+\frac{1}{2}}N_1^{-s-2}N_2^{\frac{5}{12}}N_3^{\frac{5}{12}}
N_4^{\frac{2}{3}}\prod_{i=1}^4c_{i,N_i}.
\end{split}
\]
Because $N_1\sim N_2\sim N_3\sim N_4$, 
we have
\begin{equation}\label{Jpp_est_24}
\left\|P_NJ_{2+}\right\|_{Y_N}
\lesssim \left(\frac{N}{N_1}\right)^{s+2}\prod_{i=1}^4c_{i,N_i}.
\end{equation}
By the same argument, we obtain 
\begin{equation}\label{Jmm_est_24}
\left\|P_NJ_{2-}\right\|_{Y_N}
\lesssim \left(\frac{N}{N_1}\right)^{s+2}\prod_{i=1}^4c_{i,N_i}.
\end{equation}
As a result, by (\ref{Jpm_est_24}), (\ref{Jpp_est_24}), and (\ref{Jmm_est_24}), we have
\[
\begin{split}
I_{4}&\lesssim \sum_N\sum_{(N_1,\cdots , N_4)\in \Phi_{4}}N^{s+2}
\left(\left\|P_NJ_{1+}\right\|_{Y_N}+\left\|P_NJ_{1-}\right\|_{Y_N}
+\left\|P_NJ_{2+}\right\|_{Y_N}+\left\|P_NJ_{2-}\right\|_{Y_N}\right)\\
&\lesssim \sum_{N_1}\sum_{N\lesssim N_1}\sum_{N_2\sim N_1}\sum_{N_3\sim N_1}
\sum_{N_4 \sim N_1}
\left(\frac{N}{N_1}\right)^{s+2}
\prod_{i=1}^4c_{i,N_i}\\
&\lesssim \prod_{i=1}^4\left(\sum_{N_i}c_{i,N_i}^2\right)^{\frac{1}{2}}
\end{split}
\]
because $s+2>0$ for $s\ge -\frac{1}{6}$. 
\end{proof}
\begin{remark}
We cannot obtain the same estimate for $I_4$ 
when $K=2$ by using (\ref{bistri_2}) with $L=N$
like as the proof of Theorem~\ref{mest_5_cri}. 
Indeed, if we use (\ref{bistri_2}), 
then the power of $N$ becomes negative for $s<0$. 
Therefore, we cannot sum up with respect to $N$. 
\end{remark}
\begin{remark}\label{multi_est_time}
We can also obtain 
\[
\left\|\partial_x^{2}
\prod_{i=1}^{4}\widetilde{u_i}\right\|_{\dot{Y}^{s}(T)}
\lesssim T^{\delta}
\sum_{k=1}^4\|u_k\|_{\dot{X}^{s}(T)}
\prod_{\substack{1\le i\le 4\\ i\ne k}}\|u_i\|_{\dot{X}^{-\frac{1}{6}+\rho}(T)}
\]
for any $0<T<1$, $s>-\frac{1}{6}$, $0<\rho <\frac{1}{6}$, and some 
$\delta =\delta (\rho )>0$ 
by using Corollary~\ref{bilin_T_cor2} 
with $L=N_1$, $\theta =3\rho$ for $I_1$, $I_2$, $I_3$, 
and $J_{1\pm}$, 
\begin{equation}\label{interp_T_1}
\|P_{N_1}u_1\|_{L^{\infty}_xL^2_T}
\lesssim T^{\theta}
\|P_{N_1}u_1\|_{L^{\infty}_xL^\frac{2}{1-2\theta}_T}
\lesssim T^{\theta} N_1^{-\frac{3}{2}+4\theta}\|P_{N_1}u_1\|_{X_{N_1}(T)}
\end{equation}
with $\theta =\frac{3\rho}{4}$ for $J_{2+}^{\rm high}$, 
and
\begin{equation}\label{interp_T_2}
\|P_{N_1}^+u_1\|_{L^{2}_{T}L^2_{x}}
\lesssim T^{\frac{\theta}{2}} N_1^{-2+\theta}\|P_{N_1}u_1\|_{X_{N_1}(T)}
\end{equation}
with $\theta =3\rho$ for $J_{2+}^{\rm low}$ 
in the proof of Theorem~\ref{mest_4_cri}. 
The second estimate in (\ref{interp_T_1}) 
can be obtained by the interpolation between 
the following estimates
\[
\|P_{N_1}u_1\|_{L^{\infty}_xL^{\infty}_T}
\lesssim N_1^{\frac{1}{2}}\|P_{N_1}u_1\|_{L^{\infty}_TL^{2}_x}
\lesssim N_1^{\frac{1}{2}}\|P_{N_1}u_1\|_{X_{N_1}(T)}
,\ \ 
\|P_{N_1}u_1\|_{L^{\infty}_xL^{2}_T}
\lesssim N_1^{-\frac{3}{2}}\|P_{N_1}u_1\|_{X_{N_1}(T)}.
\]
The estimate (\ref{interp_T_2}) can be obtained by 
the interpolation between the following estimates. 
\[
\|P_{N_1}^+u_1\|_{L^{2}_{T}L^2_{x}}
\lesssim N_1^{-2}\|P_{N_1}u_1\|_{X_{N_1}(T)}
,\ \ 
\|P_{N_1}^+u_1\|_{L^{2}_{T}L^2_{x}}
\lesssim T^{\frac{1}{2}}\|P_{N_1}u_1\|_{X_{N_1}(T)}.
\]
\end{remark}
\begin{theorem}[Multilinear estimates]\label{mest_4_cri_TM}
{\rm (i)}\ For any $u_1$, $\cdots$, $u_4\in \dot{X}^{-\frac{1}{6}}$, 
it holds that
\begin{equation}\label{multilin_4_hom_TM}
\left\|\partial_x^{2}
\left(\prod_{i=1}^{4}\widetilde{u_i}
-\prod_{i=1}^{4}P_{\ge M}\widetilde{u_i}\right)\right\|_{\dot{Y}^{-\frac{1}{6}}}
\lesssim T^{\delta}M^{\kappa}\prod_{i=1}^4\|u_i\|_{\dot{X}^{-\frac{1}{6}}}
\end{equation}
for some $\delta >0$ and $\kappa >0$. 
\\
{\rm (ii)}\ 
For any $u_1$, $\cdots$, $u_4\in X^{-\frac{1}{6}}$, 
it holds that
\begin{equation}\label{multilin_4_inhom_TM}
\left\|\partial_x^{2}
\left(\prod_{i=1}^{4}\widetilde{u_i}
-\prod_{i=1}^{4}P_{\ge M}\widetilde{u_i}\right)\right\|_{Y^{-\frac{1}{6}}}
\lesssim T^{\delta}M^{\kappa}\prod_{i=1}^4\|u_i\|_{X^{-\frac{1}{6}}}
\end{equation}
for some $\delta >0$ and $\kappa >0$.
\end{theorem}
\begin{proof}
By the symmetry, we can assume $N_1\ge \cdots \ge N_4$ and $N_4<M$. 
Let $s\ge -\frac{1}{6}$, $c_{1,N_1}:=N_1^s\|P_{N_1}u_1\|_{X_{N_1}(T)}$, 
$c_{i,N_i}:=N_i^{s_c}\|P_{N_i}u_i\|_{X_{N_i}(T)}$ $(i=2,3)$, 
and $c_{4,N_4}:=N_4^{s_c}\|P_{<M}P_{N_4}u_4\|_{X_{N_4}(T)}$. 
By the H\"older inequality and 
(\ref{bistri_TT_3}) with $L=N_1$, $\theta =\frac{1}{6}$, 
we have
\[
\begin{split}
N^{s+2-\frac{3}{2}}\left\|\prod_{i=1}^4P_{N_i}u_i\right\|_{L^1_xL^2_T}
&\le N^{s+\frac{1}{2}}\|P_{N_1}u_1P_{<M}P_{N_4}u_{4}\|_{L^{2}_{x}L^{2}_{T}}
\|P_{N_2}u_2\|_{L^4_xL^{\infty}_T}\|P_{N_3}u_3\|_{L^4_xL^{\infty}_T}\\
&\lesssim T^{\frac{1}{24}}\left(\frac{N}{N_1}\right)^{s+1}
\frac{N_2^{\frac{5}{12}}N_3^{\frac{5}{12}}}{N_1^{\frac{5}{6}}}N_4^{\frac{1}{6}}
\prod_{i=1}^4c_{i,N_i}
\end{split}
\]
if $(N_1,\cdots N_m)\in \Phi_{k}$ with $k=1,2,3$. 
On the other hand, if $(N_1,\cdots N_4)\in \Phi_{4}$, 
then $N_1\sim \cdots \sim N_4<M$. 
Therefore, by the H\"older inequality, we have
\[
\begin{split}
N^{s+2-\frac{3}{2}}\left\|\prod_{i=1}^4P_{N_i}u_i\right\|_{L^1_xL^2_T}
&\le T^{\frac{1}{2}}N^{s+\frac{1}{2}}
\prod_{i=1}^4\|P_{<M}P_{N_i}u_i\|_{L^{4}_xL^{\infty}_T}\\
&\lesssim T^{\frac{1}{2}}\left(\frac{N}{N_1}\right)^{s+\frac{1}{2}}
N_1^{\frac{3}{4}}N_2^{\frac{5}{12}}N_3^{\frac{5}{12}}N_4^{\frac{5}{12}}
\prod_{i=1}^4c_{i,N_i}.
\end{split}
\]

As a result, we obtain
\[
\left\|\partial_x^2\left(
\prod_{i=1}^{4}\widetilde{u_i}
-\prod_{i=1}^{4}P_{\ge M}\widetilde{u_i}\right)\right\|_{\dot{Y}^{s}(T)}
\lesssim 
T^{\frac{1}{24}}M^{2}\sum_{i=1}^m\|u_i\|_{\dot{X}^s(T)}\prod_{\substack{1\le k\le m\\ k\ne i}}\|u_k\|_{\dot{X}^{s_c}(T)}
\]
by the same argument as in the proof of Theorem~\ref{multi_est_3mg2_lowf}. 
\end{proof}
\section{Multilinear estimates at the scaling critical regularity in the case of $m\ge 4$ with $\gamma=1$}\label{mesc_2}
In this section, we use the solution space $X_N$ and its auxiliary space $Y_N$ with the norms 
\begin{equation}
\label{fs14}
\begin{split}
\|u\|_{X_N}&:=\|u(0)\|_{L^2_x}+\left\|\left(i\partial_t+\partial_x^4\right)u\right\|_{Y_N},\\
\|u\|_{Y_N}&:=\inf\left\{\left.\|u_1\|_{L^1_tL^2_x}+\|u_2\|_{\dot{X}^{0,-\frac{1}{2},1}}
\right|u=u_1+u_2\right\}, 
\end{split}
\end{equation}
instead of Definition~\ref{def4-1}, where $N\in 2^{\Z}$. Furthermore, we introduce the function spaces $\dot{X}^s$, $X^s$, $\dot{Y}^s$, and $Y^s$ with the norms
\[
\begin{split}
\|u\|_{\dot{X}^s}
&:=\left(\sum_{N\in 2^{\Z}}N^{2s}\|P_Nu\|_{X_N}^2\right)^{\frac{1}{2}},\ \ 
\|u\|_{X^s}:=\|u\|_{\dot{X}^0}+\|u\|_{\dot{X}^s},\\
\|F\|_{\dot{Y}^s}
&:=\left(\sum_{N\in 2^{\Z}}N^{2s}\|P_NF\|_{Y_N}^2\right)^{\frac{1}{2}},\ \ 
\|F\|_{Y^s}:=\|F\|_{\dot{Y}^0}+\|F\|_{\dot{Y}^s},
\end{split}
\]
where $s\in \R$. We can easily see that the estimates
\begin{equation}\label{XN_Linest_g1}
\left\|e^{it\partial_x^4}f\right\|_{\dot{X}^s}\lesssim \|f\|_{\dot{H}^s},\ \ 
\left\|\int_0^te^{i(t-t')\partial_x^4}F(t')dt'\right\|_{\dot{X}^s}\lesssim \|F\|_{\dot{Y}^s}
\end{equation}
hold for any $f\in \dot{H}^s(\R)$ and $F\in Y^s$, where the implicit constants are independent of $f$. %
Furthermore, by the same argument as the proof of Proposition~\ref{lin_XN} and ~\ref{dual_besov}, the estimates
\begin{equation}\label{lin_XN_g1}
N^{\frac{1}{2}}\|P_Nu\|_{L^4_tL^{\infty}_x}\lesssim \|P_Nu\|_{X_N}
\end{equation}
and
\[
\left\|P_Nu\right\|_{\dot{X}^{0,\frac{1}{2},\infty}}\lesssim \|P_Nu\|_{X_N}
\]
hold for any function $u$ satisfying $P_Nu\in X_N$.
\begin{theorem}[Refined bilinear Strichartz estimates on $X_{N_1}\times X_{N_2}$]\label{bist_g1_thm}
Let $L,N_1,N_2\in 2^{\Z}$ and $P_{N_1}u_1\in X_{N_1}, P_{N_2}u_2\in X_{N_2}$. 
If $N_1\ge N_2\gtrsim L$, then the estimates
\begin{align}
\label{bistri_2_g1}
     \left\|R_L^+\left(P_{N_1}u_1,\overline{P_{N_2}u_2}\right)\right\|_{L^2_{t,x}(\R\times\R)}
     &\lesssim
     N_1^{-1}L^{-\frac{1}{2}}\|P_{N_1}u_1\|_{X_{N_1}}\|P_{N_2}u_2\|_{X_{N_2}},\\
\label{bistri_3_g1}
     \left\|R_L^-\left(P_{N_1}u_1,P_{N_2}u_2\right)\right\|_{L^2_{t,x}(\R\times\R)}
     &\lesssim
     N_1^{-1}L^{-\frac{1}{2}}\|P_{N_1}u_1\|_{X_{N_1}}\|P_{N_2}u_2\|_{X_{N_2}}
\end{align}
hold, where the implicit constants are independent of $L,N_1,N_2,u_1,u_2$. Here the bilinear operators $R_L^{\pm}$ are defined by (\ref{defR}).
\end{theorem}
\begin{proof}
We prove only (\ref{bistri_3_g1}) since (\ref{bistri_2_g1}) can be proved in the similar way. 
We set $u_{j,N_j}:=P_{N_j}u_j$ and $F_j:=\left(i\partial_t+\partial_x^4\right)u_{j}$ for $j=1,2$.
It suffices to show that the estimate
\[
\left\|R_{L}^{-}(u_{1,N_1},u_{2,N_2})\right\|_{L^2_{t,x}}\lesssim 
N_1^{-1}L^{-\frac{1}{2}}\left(\|u_{1,N_1}(0)\|_{L^2_x}+\|F_1\|_{Y_{N_1}}\right)
\left(\|u_{2,N_2}(0)\|_{L^2_x}+\|F_2\|_{Y_{N_2}}\right)
\]
holds. This follows from the following estimates:
\begin{align}
\left\|R_{L}^{-}(u_{1,N_1},u_{2,N_2})\right\|_{L^2_{t,x}}&\lesssim 
N_1^{-1}L^{-\frac{1}{2}}\left(\|u_{1,N_1}(0)\|_{L^2_x}+\|F_1\|_{\dot{X}^{0,-\frac{1}{2},1}}\right) 
\left(\|u_{2,N_2}(0)\|_{L^2_x}+\|F_2\|_{\dot{X}^{0,-\frac{1}{2},1}}\right),\label{bi_est_L_1g1}\\ 
\left\|R_{L}^{-}(u_{1,N_1},u_{2,N_2})\right\|_{L^2_{t,x}}&\lesssim 
N_1^{-1}L^{-\frac{1}{2}}\left(\|u_{1,N_1}(0)\|_{L^2_x}+\|F_1\|_{\dot{X}^{0,-\frac{1}{2},1}}\right) 
\left(\|u_{2,N_2}(0)\|_{L^2_x}+\|F_2\|_{L^1_tL^2_x}\right),\label{bi_est_L_2g1}\\ 
\left\|R_{L}^{-}(u_{1,N_1},u_{2,N_2})\right\|_{L^2_{t,x}}&\lesssim 
N_1^{-1}L^{-\frac{1}{2}}\left(\|u_{1,N_1}(0)\|_{L^2_x}+\|F_1\|_{L^1_tL^2_x}\right) 
\left(\|u_{2,N_2}(0)\|_{L^2_x}+\|F_2\|_{\dot{X}^{0,-\frac{1}{2},1}}\right),\label{bi_est_L_3g1}\\ 
\left\|R_{L}^{-}(u_{1,N_1},u_{2,N_2})\right\|_{L^2_{t,x}}&\lesssim 
N_1^{-1}L^{-\frac{1}{2}}\left(\|u_{1,N_1}(0)\|_{L^2_x}+\|F_1\|_{L^1_tL^2_x}\right) 
\left(\|u_{2,N_2}(0)\|_{L^2_x}+\|F_2\|_{L^1_tL^2_x}\right). \label{bi_est_L_4g1}
\end{align}
To obtain (\ref{bi_est_L_1g1}), (\ref{bi_est_L_2g1}), and (\ref{bi_est_L_3g1}), 
we use Lemma~\ref{BSE_Lemm}. Then, we have to prove only
\begin{align}
\left\|R_{L}^{-}(e^{it\partial_x^4}u_{1,N_1}(0),e^{it\partial_x^4}u_{2,N_2}(0))\right\|_{L^2_{t,x}}&\lesssim 
N_1^{-1}L^{-\frac{1}{2}}\|u_{1,N_1}(0)\|_{L^2_x}\|u_{2,N_2}(0)\|_{L^2_x},\label{bi_est_L_5g1}\\ 
\left\|R_{L}^{-}(e^{it\partial_x^4}u_{1,N_1}(0),u_{2,N_2})\right\|_{L^2_{t,x}}&\lesssim 
N_1^{-1}L^{-\frac{1}{2}}\|u_{1,N_1}(0)\|_{L^2_x} 
\left(\|u_{2,N_2}(0)\|_{L^2_x}+\|F_2\|_{L^1_tL^2_x}\right),\label{bi_est_L_6g1}\\ 
\left\|R_{L}^{-}(u_{1,N_1},e^{it\partial_x^4}u_{2,N_2}(0))\right\|_{L^2_{t,x}}&\lesssim 
N_1^{-1}L^{-\frac{1}{2}}\left(\|u_{1,N_1}(0)\|_{L^2_x}+\|F_1\|_{L^1_tL^2_x}\right) 
\|u_{2,N_2}(0)\|_{L^2_x}. \label{bi_est_L_7g1}
\end{align}
Since the identity
\[
u_{j,N_j}(t)=u_{j,N_j}(0)-i\int_0^te^{it\partial_x^4}(e^{-it'\partial_x^4}F_j(t'))dt', 
\]
holds, we can obtain (\ref{bi_est_L_5g1}), (\ref{bi_est_L_6g1}), (\ref{bi_est_L_7g1}), and 
(\ref{bi_est_L_4g1})  by using (\ref{5-3}) and bilinearity of the operator $R_L^-$. 
\end{proof}
\begin{theorem}[Multilinear estimates]\label{multi_est_3mg1}
Let $m\ge 4$. Set
\[
     s_c=s_c(m):=
     \frac{1}{2}-\frac{3}{m-1}.
\]
{\rm (i)}\ For any $u_1$, $\cdots$, $u_m\in \dot{X}^{s_c}$, 
it holds that
\begin{equation}\label{multilin_mhigh_hom_g1}
\left\|\partial_x
\prod_{i=1}^{m}\widetilde{u_i}\right\|_{\dot{Y}^{s_c}}
\lesssim \prod_{i=1}^m\|u_i\|_{\dot{X}^{s_c}},
\end{equation}
where $\widetilde{u_i}\in \{u_i,\overline{u_i}\}$ and the implicit constant depends only on $m$.\\
{\rm (ii)}\ If $s\ge s_c$,
then for any $u_1$, $\cdots$, $u_m\in X^s$, 
it holds that
\begin{equation}\label{multilin_mhigh_inhom_g1}
\left\|\partial_x
\prod_{i=1}^{m}\widetilde{u_i}\right\|_{Y^s}
\lesssim \prod_{i=1}^m\|u_i\|_{X^s},
\end{equation}
where $\widetilde{u_i}\in \{u_i,\overline{u_i}\}$ and the implicit constant depends only on $m$ and $s$.
\end{theorem}
\begin{proof}
Let $s\ge \min \{s_c, 0\}$. We assume $u_i\in \dot{X}^s\cap \dot{X}^{s_c}$ ($i=1,\cdots, m$). 
For $i=1,\cdots, m$ and $N_i\in 2^{\Z}$, we set
$c_{1,N_1}:=N_1^s\|P_{N_1}u_1\|_{X_{N_1}}$,\ 
$c_{i,N_i}:=N_i^{s_c}\|P_{N_i}u_i\|_{X_{N_i}}$ $(i=2,\cdots ,m)$. 
We put
\[
\begin{split}
\Phi_1&:=\{(N_1,\cdots,N_m)|\ N_1\ge \cdots \ge N_m,\ N\sim N_1\gg N_{2}\},\\
\Phi_k&:=\{(N_1,\cdots,N_m)|\ N_1\ge \cdots \ge N_m,\ N_1\gtrsim N,\ N_1\sim \cdots \sim N_k\gg N_{k+1}\}\ (k=2,3,\cdots, m-1),\\
\Phi_m&:=\{(N_1,\cdots,N_m)|\ N_1\ge \cdots \ge N_m,\ N\lesssim N_1\sim \cdots \sim N_m\}.
\end{split}
\]

We first consider the case $m\ge 5$. 
By the H\"older inequality, Theorem~\ref{bist_g1_thm} with $L=N_1$, the Bernstein inequality, and (\ref{lin_XN_g1}), 
we have
\begin{equation}\label{gamma1_mhigh_holder_1}
\begin{split}
N^{s+1}\left\|\prod_{i=1}^mP_{N_i}u_i\right\|_{L^1_tL^2_x}
&\le N^{s+1}\|P_{N_1}u_1P_{N_4}u_{4}\|_{L^{2}_{t,x}}
\|P_{N_2}u_2\|_{L^4_tL^{\infty}_x}\|P_{N_3}u_3\|_{L^4_tL^{\infty}_x}
\prod_{i=5}^{m}\|P_{N_i}u_i\|_{L^{\infty}_{t,x}}\\
&\lesssim \left(\frac{N}{N_1}\right)^{s+1}
\left(\frac{N_{2}^{-\frac{1}{2}-s_c}N_3^{-\frac{1}{2}-s_c}N_{4}^{-s_c}
\prod_{i=5}^{m}N_i^{\frac{1}{2}-s_c}}
{N_1^{\frac{1}{2}}}\right)
\prod_{i=1}^mc_{i,N_i}
\end{split}
\end{equation}
if $(N_1,\cdots N_m)\in \Phi_{k}$ with $k=1,2,3$. 
We note that $\frac{1}{2}=-1-3s_c+(m-4)(\frac{1}{2}-s_c)>0$. 
On the other hand, by the H\"older inequality, the Bernstein inequality, and (\ref{lin_XN_g1}), 
we have
\begin{equation}\label{gamma1_mhigh_holder_2}
\begin{split}
N^{s+1}\left\|\prod_{i=1}^mP_{N_i}u_i\right\|_{L^1_tL^2_x}
&\le N^{s+1}
\left(\prod_{i=1}^4\|P_{N_i}u_i\|_{L^4_tL^{\infty}_x}\right)\|P_{N_5}u_5\|_{L^{\infty}_tL^2_x}
\prod_{i=6}^m\|P_{N_i}u_i\|_{L^{\infty}_{t,x}}\\
&\lesssim \left(\frac{N^{s+1}}{N_1^{s+\frac{1}{2}}N_2^{\frac{1}{2}}}\right)
\left(\frac{N_5^{-s_c}\prod_{i=6}^mN_i^{\frac{1}{2}-s_c}}{N_2^{s_c}N_3^{\frac{1}{2}+s_c}N_4^{\frac{1}{2}+s_c}}\right)
\prod_{i=1}^mc_{i,N_i}\\
&\sim \left(\frac{N}{N_1}\right)^{s+1}
\left(\frac{N_5^{-s_c}\prod_{i=6}^mN_i^{\frac{1}{2}-s_c}}{N_1^{1+3s_c}}\right)
\prod_{i=1}^mc_{i,N_i}
\end{split}
\end{equation}
if $(N_1,\cdots N_m)\in \Phi_k$ with $k=4,\cdots,m$. 
We assume $\prod_{i=6}^mN_i^{\frac{1}{2}-s_c}=1$ if $m=5$. 
We note that $1+3s_c=-s_c+(m-5)(\frac{1}{2}-s_c)>0$. 
Therefore, we obtain 
\begin{equation}\label{multilin_mhigh_3g1}
\left\|\partial_x
\prod_{i=1}^{m}\widetilde{u_i}\right\|_{\dot{Y}^{s}}
\lesssim \sum_{i=1}^m\|u_i\|_{\dot{X}^s}\prod_{\substack{1\le k\le m\\ k\ne i}}\|u_k\|_{\dot{X}^{s_c}}
\end{equation}
for any $u_i\in \dot{X}^s\cap \dot{X}^{s_c}$
by the same argument as in the proof of Theorem~\ref{multi_est_3mg2}.

Next, we consider $m=4$. 
Then, $s_c=-\frac{1}{2}$. 
By the same argument as above, we obtain
\[
\begin{split}
N^{s+1}\left\|\prod_{i=1}^mP_{N_i}u_i\right\|_{L^1_tL^2_x}
&\lesssim \left(\frac{N}{N_1}\right)^{s+1}
\left(\frac{N_4}{N_1}\right)^{\frac{1}{2}}
\prod_{i=1}^4c_{i,N_i}
\end{split}
\]
if $(N_1,N_2,N_3, N_4)\in \Phi_{k}$ with $k=1,2,3$. 
Therefore, we only have to consider the case $(N_1,N_2,N_3,N_4)\in \Phi_4$. 
Namely, $N_1\sim N_2\sim N_3\sim N_4$. 
We put $P_{N_i}^{\pm}:=P_{\pm}P_{N_i}$, 
where $P_+$ and $P_{-}$ are 
defined in (\ref{proj_plus_minus}). 
Then we have
\[
\prod_{i=1}^4P_{N_i}\widetilde{u_i}=
\prod_{i=1}^4(P_{N_i}^+\widetilde{u_i}+P_{N_i}^-\widetilde{u_i}). 
\]
We define
\[
K:=\#\{i\in \{1,2,3,4\}|\ \widetilde{u_i}=\overline{u_i}\}. 
\]
By the same argument as in the proof of Theorem~\ref{mest_5_cri} 
for $K=0,1$ and Theorem~\ref{mest_4_cri} for $K=2$, 
we can use the bilinear estimates (\ref{bistri_2}), (\ref{bistri_3}) with $L=N_1$ or 
the modulation bound such as (\ref{modulation_bd}). 
Namely, 
\[
\|P_{N_i}\widetilde{u_i}P_{N_j}\widetilde{u_j}\|_{L^2_{t,x}}\lesssim 
N_1^{-\frac{3}{2}}\|P_{N_i}u_i\|_{X_{N_i}}\|P_{N_j}u_j\|_{X_{N_j}}
\]
for some $(i,j)$ or
\[
\|P_{N_i}\widetilde{u_i}\|_{L^2_{t,x}}\lesssim N_1^{-2}\|P_{N_i}u_i\|_{X_{N_i}}
\]
for some $i$ holds. 
For the former case with $(i,j)=(1,2)$, 
by the H\"older inequality, Theorem~\ref{bist_g1_thm} with $L=N_1$, and (\ref{lin_XN_g1}), 
we have
\[
\begin{split}
N^{s+1}\left\|\prod_{i=1}^4P_{N_i}\widetilde{u_i}\right\|_{L^1_tL^2_x}
&\le N^{s+1}
\|P_{N_1}\widetilde{u_1}P_{N_2}\widetilde{u_2}\|_{L^2_{t,x}}\|P_{N_3}u_3\|_{L^{4}_tL^{\infty}_x}\|P_{N_4}u_4\|_{L^{4}_tL^{\infty}_x}
\\
&\lesssim N^{s+1}N_1^{-\frac{3}{2}}N_3^{-\frac{1}{2}}N_4^{-\frac{1}{2}}
\prod_{i=1}^4\|P_{N_i}u_i\|_{X_{N_i}}\\
&\lesssim \left(\frac{N}{N_1}\right)^{s+1}\prod_{i=1}^4c_{i,N_i}.
\end{split}
\]
For the later case with $i=1$, 
by the H\"older inequality, the Bernstein inequality, 
we have
\[
\begin{split}
N^{s+1}\left\|\prod_{i=1}^4P_{N_i}\widetilde{u_i}\right\|_{L^1_tL^2_x}
&\le N^{s+1}
\|P_{N_1}\widetilde{u_1}\|_{L^2_{tx}}\|P_{N_2}u_2\|_{L^{\infty}_{tx}}\|P_{N_3}u_3\|_{L^{4}_tL^{\infty}_x}\|P_{N_4}u_4\|_{L^{4}_tL^{\infty}_x}
\\
&\lesssim N^{s+1}N_1^{-2}N_2^{\frac{1}{2}}N_3^{-\frac{1}{2}}N_4^{-\frac{1}{2}}
\prod_{i=1}^4\|P_{N_i}u_i\|_{X_{N_i}}\\
&\lesssim \left(\frac{N}{N_1}\right)^{s+1}\prod_{i=1}^4c_{i,N_i}.
\end{split}
\]
Therefore, we obtain (\ref{multilin_mhigh_3g1}) for $m=4$. 
\end{proof}
\begin{remark}
If $m\in \{4,5,6\}$ (then, $s_c<0$), we can also obtain 
\[
\left\|\partial_x
\prod_{i=1}^{m}\widetilde{u_i}\right\|_{\dot{Y}^{s}(T)}
\lesssim T^{\delta}
\sum_{k=1}^4\|u_k\|_{\dot{X}^{s}(T)}
\prod_{\substack{1\le i\le 4\\ i\ne k}}\|u_i\|_{\dot{X}^{s_c+\rho}(T)}
\]
for any $0<T<1$, $s>s_c$, $0<\rho <-s_c$, and some 
$\delta =\delta (\rho )>0$ by the same reason as in Remark~\ref{multi_est_time}.
\end{remark}
\begin{theorem}[Multilinear estimates]\label{multi_est_gamma1_lowf}
Let $m\ge 4$, $0<T<1$, and $M\in 2^{\N}$. Set
\[
     s_c=s_c(m):=
     \frac{1}{2}-\frac{3}{m-1}.
\]
{\rm (i)}\ For any $u_1$, $\cdots$, $u_m\in \dot{X}^{s_c}$, 
it holds that
\begin{equation}\label{multilin_mhigh_hom_12_lowf3}
\left\|\partial_x\left(
\prod_{i=1}^{m}\widetilde{u_i}-\prod_{i=1}^{m}P_{\ge M}\widetilde{u_i}\right)\right\|_{\dot{Y}^{s_c}(T)}
\lesssim T^{\delta}M^{\kappa}\prod_{i=1}^m\|u_i\|_{\dot{X}^{s_c}(T)}
\end{equation}
for some $\delta >0$ and $\kappa >0$ depending only on $m$.\\
{\rm (ii)}\ For any $u_1$, $\cdots$, $u_m\in X^{s_c}$, 
it holds that
\begin{equation}\label{multilin_mhigh_hom_12_lowf4}
\left\|\partial_x\left(
\prod_{i=1}^{m}\widetilde{u_i}-\prod_{i=1}^{m}P_{\ge M}\widetilde{u_i}\right)\right\|_{Y^{s_c}(T)}
\lesssim T^{\delta}M^{\kappa}\prod_{i=1}^m\|u_i\|_{X^{s_c}(T)}
\end{equation}
for some $\delta >0$ and $\kappa >0$ depending only on $m$.
\end{theorem}
\begin{proof}
By the symmetry, we can assume $N_1\ge \cdots \ge N_m$ and $N_m<M$. 
Let $s\ge \min\{0,s_c\}$, $c_{1,N_1}:=N_1^s\|P_{N_1}u_1\|_{X_{N_1}(T)}$, 
$c_{i,N_i}:=N_i^{s_c}\|P_{N_i}u_i\|_{X_{N_i}(T)}$ $(i=1,\cdots ,m-1)$, 
and $c_{m,N_m}:=N_m^{s_c}\|P_{<M}P_{N_m}u_m\|_{X_{N_m}(T)}$. 
We first assume the case $m\ge 5$. 
By the interpolation between the two estimates
\[
\|P_{N_3}u_3\|_{L^{4}_TL^{\infty}_x}\lesssim N_3^{-\frac{1}{2}}\|P_{N_3}u_3\|_{X_{N_3}(T)},\ \ 
\|P_{N_3}u_3\|_{L^{4}_TL^{\infty}_x}\lesssim T^{\frac{1}{2}}N_3^{\frac{1}{2}}\|P_{N_3}u_3\|_{X_{N_3}(T)},
\]
we obtain
\begin{equation}\label{l4stri_interp}
\|P_{N_3}u_3\|_{L^{4}_TL^{\infty}_x}\lesssim T^{\frac{\theta}{4}}N_3^{-\frac{1}{2}+\theta}\|P_{N_3}u_3\|_{X_{N_3}(T)}
\end{equation}
for $0<\theta <1$. 
We use this estimate with $\theta =\frac{3}{m-1}$ instead of
\[
\|P_{N_3}u_3\|_{L^{4}_TL^{\infty}_x}\lesssim N_3^{-\frac{1}{2}}\|P_{N_3}u_3\|_{X_{N_3}(T)}
\]
in (\ref{gamma1_mhigh_holder_1}) and (\ref{gamma1_mhigh_holder_2}). 
Then, we obtain 
\[
\left\|\partial_x\left(
\prod_{i=1}^{m}\widetilde{u_i}
-\prod_{i=1}^{m}P_{\ge M}\widetilde{u_i}\right)\right\|_{\dot{Y}^{s}(T)}
\lesssim 
T^{\frac{\theta}{4}}M^{\frac{1}{2}-s_c}\sum_{i=1}^m\|u_i\|_{\dot{X}^s(T)}\prod_{\substack{1\le k\le m\\ k\ne i}}\|u_k\|_{\dot{X}^{s_c}(T)}
\]
by the same argument as in the proof of Theorem~\ref{multi_est_3mg2_lowf}. 

Next, we assume the case $m=4$. Then, $s_c=-\frac{1}{2}$. 
By the H\"older inequality, Theorem~\ref{bist_g1_thm} with $L=N_1$, and 
(\ref{l4stri_interp}) with $\theta =\frac{1}{2}$, 
we have
\[
\begin{split}
N^{s+1}\left\|\prod_{i=1}^4P_{N_i}u_i\right\|_{L^1_TL^2_x}
&\le N^{s+1}\|P_{N_1}u_1P_{<M}P_{N_4}u_{4}\|_{L^{2}_{T}L^{2}_{x}}
\|P_{N_2}u_2\|_{L^4_TL^{\infty}_x}\|P_{N_3}u_3\|_{L^4_TL^{\infty}_x}\\
&\lesssim T^{\frac{1}{8}}\left(\frac{N}{N_1}\right)^{s+1}
\left(\frac{N_3}{N_1}\right)^{\frac{1}{2}}N_4^{\frac{1}{2}}
\prod_{i=1}^4c_{i,N_i}. 
\end{split}
\]
if $(N_1,\cdots N_m)\in \Phi_{k}$ with $k=1,2,3$. 
On the other hand, if $(N_1,\cdots N_4)\in \Phi_{4}$, 
then $N_1\sim \cdots \sim N_4<M$. 
Therefore, by the H\"older inequality and the Bernstein inequality, we have
\[
\begin{split}
N^{s+1}\left\|\prod_{i=1}^4P_{N_i}u_i\right\|_{L^1_TL^2_x}
&\le TN^{s+1}\|P_{<M}P_{N_1}u_1\|_{L^{\infty}_{T}L^2_x}
\prod_{i=2}^4\|P_{<M}P_{N_i}u_i\|_{L^{\infty}_TL^{\infty}_x}\\
&\lesssim T\left(\frac{N}{N_1}\right)^{s+1}N_1N_2N_3N_4
\prod_{i=1}^4c_{i,N_i}. 
\end{split}
\]

As a result, we obtain
\[
\left\|\partial_x\left(
\prod_{i=1}^{4}\widetilde{u_i}
-\prod_{i=1}^{4}P_{\ge M}\widetilde{u_i}\right)\right\|_{\dot{Y}^{s}(T)}
\lesssim 
T^{\frac{1}{8}}M^{4}\sum_{i=1}^m\|u_i\|_{\dot{X}^s(T)}\prod_{\substack{1\le k\le m\\ k\ne i}}\|u_k\|_{\dot{X}^{s_c}(T)}
\]
by the same argument as in the proof of Theorem~\ref{multi_est_3mg2_lowf}. 
\end{proof}
%
%



\section{Proof of well-posedness}
\label{well-po}

\ \ In this section, we give proofs of Theorem \ref{lwp1} and Theorem \ref{lwp2-3}. For $s\in \R$ and $r>0$, we define the closed ball $B_r\left(H^s(\R)\right)$ in $H^s(\R)$ centered at the origin with the radius $r$ as
\[
     B_r(H^s(\R)):=\left\{\phi\in H^s(\R):\ \|\phi\|_{H^s(\R)}\le r\right\}.
\]
For 
$T>0$, $m\in \N$ with $m\ge 3$ and $\rho>0$, we introduce a closed ball $X^s(T;\rho)$ in $X^s(T)$ centered at the origin with a radious $\rho$ as
\[
    X^s(T;\rho):=\left\{u\in X^s(T):\|u\|_{X^s(T)}\le \rho\right\}.
\]
Here the function space $X^s(T)$ will be chosen as Definition \ref{def3-4} for Theorem \ref{lwp1} or (\ref{fsb}) for Theorem \ref{lwp2-3} below. 
For $u_0\in B_r(H^s(\R))$, we introduce a nonlinear mapping $\Psi$ given by
\begin{equation}
\label{8-1-7}
    \Psi[u](t):=e^{it\partial_x^4}u_0-i\mathcal{I}\left[G\left(\left\{\partial_x^{k}u\right\}_{k\le \gamma},\left\{\partial_x^{k}\bar{u}\right\}_{k\le \gamma}\right)\right]
\end{equation}
for $t\in (-T,T)$, where $u\in X^s(T;\rho)$. In the following subsections, we consider whether the nonlinear mapping $\Psi$ is a contraction mapping.
\subsection{Proof of Theorem \ref{lwp1}}
\ \ In this subsection, we give a proof of Theorem \ref{lwp1}. In the proof, we choose the solution space $X^s(T)$ as given in Definition \ref{def3-4}. The main tool for the proof is the multilinear estimate (Theorem \ref{sc_inv_multi_e}). If the nonlinearity $G=G_{\gamma}^{m,m}$ is the form of (\ref{nonl_sc}), the Duhamel term is written as
\begin{align}
\label{8-2-4}
    &\mathcal{I}\left[G_{\gamma}^{m,m}\left(\left\{\partial_x^{k}u\right\}_{k\le \gamma},\left\{\partial_x^{k}\bar{u}\right\}_{k\le \gamma}\right)\right]=\sum_{\mathbf{k}+\mathbf{l}=m}\sum_{|\alpha|+|\beta|=\gamma}C_{\alpha,\beta}^{\mathbf{k},\mathbf{l}}\mathcal{I}\left[(\partial_x^{\alpha_1}u)\cdots(\partial_x^{\alpha_{\mathbf{k}}}u)\left(\partial_x^{\beta_1}\overline{u}\right)\cdots\left(\partial_x^{\beta_{\mathbf{l}}}\overline{u}\right)\right]\notag\\
    &=\sum_{|\alpha|=\gamma}C_{\alpha}^{m}\mathcal{I}\left[(\partial_x^{\alpha_1}\widetilde{u})\cdots(\partial_x^{\alpha_{\mathbf{k}}}\widetilde{u})\left(\partial_x^{\alpha_{\mathbf{k}+1}}\widetilde{u}\right)\cdots\left(\partial_x^{\alpha_{m}}\widetilde{u}\right)\right]
    =\sum_{|\alpha|=\gamma}C_{\alpha}^{m}\mathcal{I}\left[\prod_{i=1}^m\partial_x^{\alpha_{i}}\widetilde{u}\right].
\end{align}
Here to obtain the second equality of (\ref{8-2-4}), we write $u$ and $\overline{u}$ as $\widetilde{u}$ and $\beta_j$ as $\alpha_{\mathbf{k}+j}$ for $j\in \{1,\dots,\mathbf{l}\}$. 
The precise statement of Theorem \ref{lwp1} with $\gamma=3$ is as follows:

\begin{theorem}
\label{lwpcomp}
Let $\gamma=3$, $m\ge3$, $s\ge s_0$, where $s_0$ is given by (\ref{minir2}), and $T\in (0,1)$. We assume that the nonlinearity $G=G_{\gamma}^{m,m}$ is the form of (\ref{nonl_sc}).
Then the following statements hold: 
\begin{itemize}
\item (Existence): There exists a positive constant $\varepsilon=\varepsilon(m,s)>0$ such that for any $u_0\in B_{\varepsilon}(H^s(\R))$, there exists a unique solution $u\in X^s\left(T;2C_0\varepsilon\right)$ to the problem (\ref{1-1})-(\ref{nonl_sc}) on the time interval $I_T=(-T,T)$, where $C_0$ is a positive constant given in Proposition \ref{prop3-8}.
\item (Uniqueness): Let $u\in X^s(T;2C_0\varepsilon)$ be the solution obtained in the Existence part. Let $T_1\in (0,T]$ and $w\in X^s\left(T_1;2C_0\varepsilon\right)$ be another solution to (\ref{1-1})-(\ref{nonl_sc}). If the identity $u_0=w(0)$ holds, then the identity $u=w$ holds on $[-T_1,T_1]$.
\item (Continuity of the flow map): The flow map $\Xi:B_{\varepsilon}(H^s(\R))\mapsto X^s\left(T;2C_0\varepsilon\right)$,\ $u_0\mapsto u$ is Lipschitz continuous.
\end{itemize}
\end{theorem}

\begin{proof}[Proof of Theorem \ref{lwpcomp}]
(Existence): Let $\varepsilon>0$, which will be chosen later. Let $u\in X^s(T;2C_0\varepsilon)$.
By the identity (\ref{8-2-4}), Proposition \ref{prop3-8}, Corollary \ref{propdeco} and Theorem \ref{sc_inv_multi_e} with $\gamma=3$, we see that there exists a positive constant $C_1=C_1(m,s)>0$ the estimates
\begin{align}
    \left\|\Psi[u]\right\|_{X^s(T)}&\le \left\|e^{it\partial_x^4}u_0\right\|_{X^s(T)}+\sum_{|\alpha|=3}\left|C_{\alpha}^{m}\right|\left\|\mathcal{I}\left[\prod_{i=1}^m\partial_x^{\alpha_{i}}\widetilde{u}\right]\right\|_{X^s(T)}\notag\\
    &\le C_0\|u_0\|_{H^s(\R)}+C_s\sum_{|\alpha|=3}\left|C_{\alpha}^{m}\right|\left\|\prod_{i=1}^m\partial_x^{\alpha_{i}}\widetilde{u}\right\|_{Y^s(T)}\notag\\
    &\le C_0\varepsilon +C_1\left\|u\right\|_{X^s(T)}^m\le C_0\varepsilon+C_1(2C_0\varepsilon)^m\le 2C_0\varepsilon\label{8-3-5}
\end{align}
hold. Here we take $\varepsilon>0$ such as $\varepsilon\le  \left(\frac{1}{C_1C_0^{m-1}2^{m}}\right)^{\frac{1}{m-1}}$ to obtain the last inequality of (\ref{8-3-5}). This implies that the nonlinear mapping $\Psi$ is well defined from $X^s(T,2C_0\varepsilon)$ to itself. Let $u,w\in X^s(T;2C_0\varepsilon)$. By a simple computation, the identity
\begin{align}
\label{8-4-4}
    \prod_{i=1}^m u_i-    \prod_{i=1}^m w_i=\sum_{i=1}^m\prod_{\mathbf{l}=1}^{i-1} w_{\mathbf{l}}(u_i-w_i)\prod_{\mathbf{k}=i+1}^mu_{\mathbf{k}}
\end{align}
holds for any $u_1,\cdots, u_m\in \C$ and $w_1,\cdots, w_m\in \C$, 
where we assumed $\prod_{\mathbf{l}=1}^{i-1} w_{\mathbf{l}}=1$ when $i=1$ 
and $\prod_{\mathbf{k}=i+1}^mu_{\mathbf{k}}=1$ when $i=m$. 
By the identity (\ref{8-4-4}), in the similar manner as the proof of the estimate (\ref{8-3-5}), the estimates
\begin{align}
    \left\|\Psi[u]-\Psi[w]\right\|_{X^s(T)}&\le \sum_{|\alpha|=3}\left|C_{\alpha}^{m}\right|\left\|\mathcal{I}\left[\prod_{i=1}^m\partial_x^{\alpha_{i}}\widetilde{u}-\prod_{i=1}^m\partial_x^{\alpha_{i}}\widetilde{w}\right]\right\|_{X^s(T)}\notag\\
    &\le C_s\sum_{|\alpha|=3}\left|C_{\alpha}^{m}\right|\left\|\left[\prod_{i=1}^m\partial_x^{\alpha_{i}}\widetilde{u}-\prod_{i=1}^m\partial_x^{\alpha_{i}}\widetilde{w}\right]\right\|_{Y^s(T)}\notag\\
    &\le C_s\sum_{|\alpha|=3}\left|C_{\alpha}^{m}\right|\sum_{i=1}^m\left\|\prod_{\mathbf{l}=1}^{i-1} \left(\partial_x^{\alpha_{\mathbf{l}}}w_{\mathbf{l}}\right)\left\{\partial_x^{\alpha_i}(u_i-w_i)\right\}\prod_{\mathbf{k}=i+1}^m\partial_x^{\alpha_{\mathbf{k}}}u_{\mathbf{k}}\right\|_{Y^s(T)}\notag\\
    &\le C_s\sum_{|\alpha|=3}\left|C_{\alpha}^{m}\right|\sum_{i=1}^m\|w\|_{X^s(T)}^{i-1}\|u\|_{X^s(T)}^{m-i}\|u-w\|_{X^s(T)}\notag\\
    &\le C_1(2C_0\varepsilon)^{m-1}\|u-w\|_{X^s(T)}\le \frac{1}{2}\|u-w\|_{X^s(T)}
    \label{8-5-6}
\end{align}
hold. Here we take $\varepsilon$ such as $\varepsilon \le \frac{1}{2C_0}\left(\frac{1}{2C_1}\right)^{\frac{1}{m-1}}$ to obtain the last inequality of (\ref{8-5-6}). This implies that the nonlinear mapping $\Psi$ is a contraction mapping. Thus by the contraction mapping principle, we see that there exists a unique solution $u\in X^s(T;2C_0\varepsilon)$ to (\ref{1-1})-(\ref{nonl_sc}). \\
(Uniqueness) and (Continuity of the flow map) can be proved in the similar manner as the proof of the estimate (\ref{8-5-6}), which completes the proof of the theorem.
\end{proof}

The precise statement of Theorem \ref{lwp1} with $\gamma\in \{1,2\}$ is as follows.
\begin{theorem}
\label{lwpcomp2}
Let $\gamma\in \{1,2\}$, $m\ge3$, and $s\ge \max\{s_0,0\}$, where $s_0$ is given by (\ref{minir}). We assume that the nonlinearity $G=G_{\gamma}^{m,m}$ is the form of (\ref{nonl_sc}).
Then the following statements hold: 
\begin{itemize}
\item (Existence): For any $r>0$, there exists a positive $T=T(r,m,s)>0$ such that for any $u_0\in B_{r}(H^s(\R))$, there exists a solution $u\in X^s\left(T;2r\right)$ to the problem (\ref{1-1})-(\ref{nonl_sc}) on the time interval $I_T=(-T,T)$.
\item (Uniqueness): Let $u\in X^s(T;2r)$ be the solution obtained in the Existence part. Let $T_1\in (0,T]$ and $w\in X^s\left(T_1\right)$ be another solution to (\ref{1-1})-(\ref{nonl_sc}). If the identity $u_0=w(0)$ holds, then the identity $u=w$ holds on $[-T_1,T_1]$.
\item (Continuity of the flow map): The flow map $\Xi:B_{r}(H^s(\R))\mapsto X^s(T;2r)$,\ $u_0\mapsto u$ is Lipschitz continuous.
\end{itemize}
Moreover, let $\left(-T_{\text{min}}, T_{\text{max}}\right)$ be the maximal existence time interval of the solution $u$. Then the blow-up alternative holds:
\[
  T_{\text{max}}<\infty\ \Longrightarrow \lim_{t\rightarrow T_{\text{max}}-0}\|u(t)\|_{H^s(\R)}=\infty.
\]
The similar statement also holds in the negative time direction.
\end{theorem}

\begin{proof}[Proof of Theorem \ref{lwpcomp2}]
(Existence): Let $T\in (0,1)$, which will be chosen later. Let $u\in X^s\left(T;2C_0r\right)$.
By the identity (\ref{8-2-4}), Proposition \ref{prop3-8}, Corollary \ref{propdeco} and Theorem \ref{sc_inv_multi_e} with $\gamma\in \{1,2\}$, we see that there exist positive constants $\delta=\delta(m,s)>0$ and $C_1=C_1(m,s)>0$ such that the estimates
\begin{align}
    \left\|\Psi[u]\right\|_{X^s(T)}&\le \left\|e^{it\partial_x^4}u_0\right\|_{X^s(T)}+\sum_{|\alpha|=\gamma}\left|C_{\alpha}^{m}\right|\left\|\mathcal{I}\left[\prod_{i=1}^m\partial_x^{\alpha_{i}}\widetilde{u}\right]\right\|_{X^s(T)}\notag\\
    &\le C_0\|u_0\|_{H^s(\R)}+C_s\sum_{|\alpha|=\gamma}\left|C_{\alpha}^{m}\right|\left\|\prod_{i=1}^m\partial_x^{\alpha_{i}}\widetilde{u}\right\|_{Y^s(T)}\notag\\
    &\le C_0r +C_1T^{\delta}\left\|u\right\|_{X^s(T)}^m\le C_0r+C_1T^{\delta}(2C_0r)^m\le 2C_0r\label{8-6-5}
\end{align}
hold. Here we take $T>0$ such as $T\le \left(\frac{1}{2C_1(2C_0r)^{m-1}}\right)^{\frac{1}{\delta}}$ to obtain the last inequality of (\ref{8-6-5}). This implies that the nonlinear mapping $\Psi$ is well defined from $X^s\left(T,2C_0r\right)$ to itself. Let $u,w\in X^s\left(T;2C_0r\right)$. By the identity (\ref{8-4-4}), in the similar manner as the proof of the estimate (\ref{8-6-5}), the estimates
\begin{align}
    \left\|\Psi[u]-\Psi[w]\right\|_{X^s(T)}&\le \sum_{|\alpha|=\gamma}\left|C_{\alpha}^{m}\right|\left\|\mathcal{I}\left[\prod_{i=1}^m\partial_x^{\alpha_{i}}\widetilde{u}-\prod_{i=1}^m\partial_x^{\alpha_{i}}\widetilde{w}\right]\right\|_{X^s(T)}\notag\\
    &\le C_s\sum_{|\alpha|=\gamma}\left|C_{\alpha}^{m}\right|\left\|\left[\prod_{i=1}^m\partial_x^{\alpha_{i}}\widetilde{u}-\prod_{i=1}^m\partial_x^{\alpha_{i}}\widetilde{w}\right]\right\|_{Y^s(T)}\notag\\
    &\le C_s\sum_{|\alpha|=\gamma}\left|C_{\alpha}^{m}\right|\sum_{i=1}^m\left\|\prod_{\mathbf{l}=1}^{i-1} \left(\partial_x^{\alpha_{\mathbf{l}}}w_{\mathbf{l}}\right)\left\{\partial_x^{\alpha_i}(u_i-w_i)\right\}\prod_{\mathbf{k}=i+1}^m\partial_x^{\alpha_{\mathbf{k}}}u_{\mathbf{k}}\right\|_{Y^s(T)}\notag\\
    &\le C_sT^{\delta}\sum_{|\alpha|=\gamma}\left|C_{\alpha}^{m}\right|\sum_{i=1}^m\|w\|_{X^s(T)}^{i-1}\|u\|_{X^s(T)}^{m-i}\|u-w\|_{X^s(T)}\notag\\
    &\le C_1T^{\delta}(2C_0r)^{m-1}\|u-w\|_{X^s(T)}\le \frac{1}{2}\|u-w\|_{X^s(T)}
    \label{8-7-6}
\end{align}
hold. This implies that the nonlinear mapping $\Psi$ is a contraction mapping. Thus by the contraction mapping principle, we see that there exists a unique solution $u\in X^s\left(T;2C_0r\right)$ to (\ref{1-1})-(\ref{nonl_sc}).\\
(Uniqueness): We only consider the positive time direction, since the negative time diraction can be treated in the similar manner. On the contrary, we assume that there exists $t\in (0,T_1]$ such that the relation $u(t)\ne w(t)$ holds. Then we can define $t_0:=\inf\{t\in [0,T_1):u(t)\ne w(t)\}$. Since $u,w\in X^s(T_1)\hookrightarrow C([0,T_1);H^s(\R))$, the identity $u(t_0)=w(t_0)$ holds. By the time translation, we may assume that $t_0=0$. In the similar manner as the proof of the estimate (\ref{8-7-6}), there exists $\tau\in (0,T_1)$ such that the estimates
\begin{align*}
    \|u-w\|_{X^s(\tau)}\le C\tau^{\delta}\sum_{|\alpha|=\gamma}\left|C_{\alpha}^{m}\right|\sum_{i=1}^m\|w\|_{X^s(T_1)}^{i-1}\|u\|_{X^s(T_1)}^{m-i}\|u-w\|_{X^s(\tau)}\le \frac{1}{2}\|u-w\|_{X^s(\tau)}
\end{align*}
hold, which implies that the identity $u(t)=w(t)$ holds on $[0,\tau)$. This contradicts the definition of $t_0$.
\\
(Continuity of the flow map): Let $u,w\in X^s(T;\rho)$ be the solutions to the problem (\ref{1-1})-(\ref{nonl_sc}) on the time interval $I_T$ with the initial data $u_0,w_0\in B_r(H^s(\R))$, repectively. In the similar manner as the proof of the estimate (\ref{8-7-6}), the estimates
\begin{align*}
    \|u-w\|_{X^s(T)}&\le C_0\|u_0-w_0\|_{H^s(\R)}+ C_1T^{\delta}(2C_0r)^{m-1}\|u-w\|_{X^s(T)}\\
    &\le C_0\|u_0-w_0\|_{H^s(\R)}+\frac{1}{2}\|u-w\|_{X^s(T)}
\end{align*}
hold, which implies that the flow map $\Xi$ is Lipschitz continuous.\\
The blow-up alternative can be proved in the standard manner.
\end{proof}

\begin{remark}\label{lwpremark1}
\begin{enumerate}
\item Theorem \ref{corlwp} and the local well-posedness part of Theorem~\ref{lwp2-3} for $s>s_c$ 
can be proved in the similar manner as above.
\item To show Theorem \ref{cor2-4}, we introduce a new unknown function $v$ given by $v:=\langle \partial_x\rangle^{\gamma}u$. By applying the contraction mapping principle, we can construct the new function $v$. Especially, the nonlinear terms can be handled in the similar argument to treat the form of (\ref{nonli3}).
\end{enumerate}
\end{remark}

\subsection{Proof of Theorem \ref{lwp2-3}}


\ \ In this subsection, we give a proof of Theorem \ref{lwp2-3}. In the proof, when (i) $\gamma=1$, $m\ge 4$ and $s\ge s_c$, (ii) $\gamma=2$, $m=4$ and $s\ge s_c$, or (iii) $\gamma=2$, $m\ge 5$ and $s\ge s_c$, we choose the solution space $X_N$ as (i) (\ref{fs14}), (ii) (\ref{fsb}) or (iii) (\ref{fsea}), and we use the multilinear estimates of (i) Theorem \ref{multi_est_3mg1}, (ii) Theorem \ref{mest_4_cri}, or (iii) Theorem \ref{multi_est_3mg2}.

The precise statement of the global well-posedness part of 
Theorem \ref{lwp2-3} is as follows.
\begin{theorem}
\label{thm8-3}
Let $\gamma \in \{1,2\}$, $m\ge 4$, and $s\ge s_c$.
We assume that the nonlinearity $G=G_{\gamma}^{m,m}$ is the form of (\ref{nonl_sc_2}). 
Then there exists a positive constant $\varepsilon=\varepsilon(m,s)>0$ such that for any $u_0\in B_{\varepsilon}\left(H^s(\R)\right)$, there exists a unique small global solution $u\in X^s(\R)$ to the problem (\ref{1-1})-(\ref{nonl_sc_2}) on $\R$. Moreover, there exist scattering states $u_\pm\in H^s(\R)$ such that the identity
\begin{equation}
\label{scatte}
     \lim_{t\rightarrow\pm\infty}\left\|u(t)-e^{it\partial_x^4}u_{\pm}\right\|_{H^s}=0
\end{equation}
holds, where the double-sign corresponds.
\end{theorem}

\begin{proof}[Proof of Theorem \ref{thm8-3}]
We only consider the case of $\gamma=1$, since the case of $\gamma=2$ can be proved in the similar manner. Let $\varepsilon>0$, which will be chosen later. Let $u\in X^s(\R;2C_0\varepsilon)$.
By the linear estimates ((\ref{XN_Linest_g1}) for $\gamma=1$ or (\ref{XN_Linest_g2}) for $\gamma=2$) and the multilinear estimates (Theorem \ref{multi_est_3mg1} for $\gamma=1$ or Theorem \ref{mest_4_cri} for $\gamma=2$), we see that there exists a positive constant $C_1=C_1(m,s)>0$ the estimates
\begin{align}
    \left\|\Psi[u]\right\|_{X^s}&\le \left\|e^{it\partial_x^4}u_0\right\|_{X^s}+\sum_{k=0}^m\left|C_{k}\right|\left\|\mathcal{I}\left[\partial_x^{\gamma}\left(u^k\overline{u}^{m-k}\right)\right]\right\|_{X^s}\notag\\
    &\le C_0\|u_0\|_{H^s(\R)}+C_s\sum_{k=0}^m\left|C_{k}\right|\left\|\partial_x^{\gamma}\left(u^k\overline{u}^{m-k}\right)\right\|_{Y^s}\notag\\
    &\le C_0\varepsilon +C_1\left\|u\right\|_{X^s}^m\le C_0\varepsilon+C_1(2C_0\varepsilon)^m\le 2C_0\varepsilon\label{8-8-5}
\end{align}
hold. Here we take $\varepsilon>0$ such as $\varepsilon\le  \left(\frac{1}{C_1C_0^{m-1}2^{m}}\right)^{\frac{1}{m-1}}$ to obtain the last inequality of (\ref{8-8-5}). This implies that the nonlinear mapping $\Psi$ is well defined from $X^s(\R,2C_0\varepsilon)$ to itself. Let $u,w\in X^s(\R;2C_0\varepsilon)$. By a simple computation, the identity
\begin{align}
\label{8-9-4}
u^k\overline{u}^{m-k}-w^k\overline{w}^{m-k}=(u-w)\overline{u}^{m-k}\sum_{i=1}^ku^{k-i}w^{i-1}+(\overline{u-w})w^k\sum_{i=1}^{m-k}\overline{u}^{m-k-i}\overline{w}^{i-1}
\end{align}
holds for any $u,w\in \C$.
By the identity (\ref{8-9-4}), in the similar manner as the proof of the estimate (\ref{8-8-5}), the estimates
\begin{align}
    &\left\|\Psi[u]-\Psi[w]\right\|_{X^s}\notag\\
    &\le \sum_{k=0}^m\left|C_{k}\right|\left\|\mathcal{I}\left[\partial_x^{\gamma}\left(u^k\overline{u}^{m-k}-w^k\overline{w}^{m-k}\right)\right]\right\|_{X^s}\notag\\
    &\le C_s \sum_{k=0}^m\left|C_{k}\right|\left\|\partial_x^{\gamma}\left(u^k\overline{u}^{m-k}-w^k\overline{w}^{m-k}\right)\right\|_{Y^s}\notag\\
    &\le
    C_s\sum_{k=0}^m\left|C_k\right|\left[\sum_{i=1}^k\left\|\partial_x^{\gamma}\left\{(u-w)\overline{u}^{m-k}u^{k-i}w^{i-1}\right\}\right\|_{Y^s}+\sum_{i=1}^{m-k}\left\|\partial_x^{\gamma}\left\{(\overline{u-w})\overline{u}^{m-k-i}w^k\overline{w}^{i-1}\right\}\right\|_{Y^s}\right]
\notag\\
    &\le C_s\sum_{k=0}^m\left|C_{k}\right|\sum_{i=1}^m\|u\|_{X^s}^{m-i}\|w\|_{X^s}^{i-1}\|u-w\|_{X^s}\notag\\
    &\le C_1(2C_0\varepsilon)^{m-1}\|u-w\|_{X^s}\le \frac{1}{2}\|u-w\|_{X^s}
    \label{8-10-6}
\end{align}
hold. Here we take $\varepsilon$ such as $\varepsilon \le \frac{1}{2C_0}\left(\frac{1}{2C_1}\right)^{\frac{1}{m-1}}$ to obtain the last inequality of (\ref{8-10-6}). This implies that the nonlinear mapping $\Psi$ is a contraction mapping. Thus by the contraction mapping principle, we see that there exists a unique solution $u\in X^s(\R;2C_0\varepsilon)$ to (\ref{1-1})-(\ref{nonl_sc_2}). 

Next we prove that the global solution $u\in X^s(\R)$ scatters in $H^s(\R)$ as $t\rightarrow\pm\infty$. We only consider the positive time direction, since the negative time direction can be treated in the similar manner. Let $t_2>t_1>0$. We claim that if $F\in Y^s$, then the relation 
\begin{equation}
\label{inhomo}
\|F\|_{Y^s(t_1,t_2)}\rightarrow0.
\end{equation}
holds as $t_2>t_1\rightarrow\infty$. Indeed, we note that for any $N\in 2^{\Z}$, the relation 
\begin{equation}
\label{inhomoge}
\|P_NF\|_{Y_N(t_1,t_2)}\rightarrow 0
\end{equation}
holds as $t_2>t_1\rightarrow\infty$ due to the definition of the $Y_N$-norm. By the embedding $Y^s\hookrightarrow Y^s(t_1,t_2)$ and the relation (\ref{inhomoge}), the relation (\ref{inhomo}) holds. We note that the nonlinear term $G$ belongs to $Y^s$ due to the estimates (\ref{8-8-5}). By the linear estimates ((\ref{XN_Linest_g1}) for $\gamma=1$ and (\ref{XN_Linest_g2}) for $\gamma=2$) and the relation (\ref{inhomo}), the relations
\begin{align*}
    \left\|e^{-it_2\partial_x^4}u(t_2)-e^{-it_1\partial_x^4}u(t_1)\right\|_{H^s(\R)}&=\left\|\int_{t_1}^{t_2}e^{-it'\partial_x^4}G(t')dt'\right\|_{H^s(\R)}
    \le \sup_{t\in [t_1,t_2]}\left\|\int_{t_1}^{t}e^{-i(t-t')\partial_x^4}G(t')dt'\right\|_{H^s(\R)}\\
    &\le C_s\|G\|_{Y^s(t_1,t_2)}\rightarrow 0
\end{align*}
hold as $t_2>t_1\rightarrow\infty$, which implies that $\left\{e^{-it\partial_x^4}u(t)\right\}_{t>0}$ satisfies the Cauchy condition in $H^s(\R)$. Since $H^s(\R)$ is complete and the operator $e^{it\partial_x^4}$ is unitary, there exists $u_+\in H^s(\R)$ such that the identity (\ref{scatte}) holds, which completes the proof of the theorem.
\end{proof}
\begin{remark}
Theorem~\ref{lwp2-2} 
can be proved in the similar manner as above.
\end{remark}
Next we give a proof of large data local well-posedness at the scaling critical regularity $s=s_c$ in Theorem \ref{lwp2-3}. For $R\ge \epsilon>0$ and $\phi_1\in B_R(H^{s_c}(\R))$, we introduce a closed ball $\phi_1+B_{\epsilon}(H^{s_c}(\R))$ in $H^{s_c}(\R)$ given by
\[
\phi_1+B_{\epsilon}(H^{s_c}(\R))
:=\left\{\phi \in H^{s_c}(\R)|\ 
\phi =\phi_1+\phi_2,\ 
\|\phi_2\|_{H^{s_c}(\R)}\le \epsilon\right\}. 
\]
For any $\phi\in H^{s_c}(\R)$, the identity $\lim_{M\rightarrow\infty}\|P_{\ge M}\phi\|_{H^{s_c}(\R)}=0$ holds, which enables us to define a map $M:H^{s_c}(\R)\rightarrow 2^{\N}$ given by
\[
   M(\phi)=M\left(\phi,\epsilon\right):=\min\left\{M\in 2^{\Z}:\ \|P_{\ge M}\phi\|_{H^{s_c}(\R)}\le \epsilon\ \right\}.
\]
For $T\in (0,1)$, $\rho>0$ and $a>0$, we introduce a closed ball $X^{s_c}\left(T;\rho,a\right)$ in $X^{s_c}(T)$ defined by
\[
   X^{s_c}\left(T;\rho,a\right):=\left\{u\in X^{s_c}(T):\ \|u\|_{X^{s_c}(T)}\le \rho,\ \|P_{\ge M}u\|_{X^{s_c}(T)}\le a\right\}.
\]
In the following, we prove that the nonlinear mapping $\Psi$ given by (\ref{8-1-7}) is a contraction mapping on the complete metric space $X^{s_c}(T;\rho,a)$. The main tools for the proof are the multilinear estimates (\ref{multilin_mhigh_inhom_g1})-(\ref{multilin_mhigh_hom_12_lowf4}) if $\gamma=1$ or (\ref{multilin_mhigh_inhom_12})-(\ref{multilin_mhigh_hom_12_lowf}) if $\gamma=2$. 

The precise statement of the large data local well-posedness part of Theorem \ref{lwp2-3} at the critical regularity case $s=s_c$ is as follows.
\begin{theorem}[Large data local well-posedness at the critical regularity $s=s_c$]
\label{thm8-4}
Let $\gamma \in \{1,2\}$ and $m\ge 4$.
We assume that the nonlinearity $G=G_{\gamma}^{m,m}$ is the form of (\ref{nonl_sc_2}). Then the following statements hold: 
\begin{itemize}
\item (Existence): There exist a constant $\epsilon=\epsilon(m)>0$ dependent only on $m$ and a map $M:H^{s_c}(\R)\rightarrow 2^{\N}$ such that the following holds: For any $R>0$ and $u_0^*\in B_{R}(H^{s_c}(\R))$, there exists $T=T\left(R,M(u_0^*)\right)>0$ such that for any $u_0\in u_0^*+B_{\epsilon}(H^{s_c}(\R))$, there exists a solution $u\in X^{s_c}\left(T\right)$ to the problem (\ref{1-1})-(\ref{nonl_sc}) on the time interval $I_T=(-T,T)$.
\item (Uniqueness): Let $u\in X^{s_c}(T)$ be the solution obtained in the Existence part. Let $T_1\in (0,T]$ and $w\in X^{s_c}\left(T_1\right)$ be another solution to (\ref{1-1})-(\ref{nonl_sc}). If the identity $u_0=w(0)$ holds, then the identity $u=w$ holds on $[-T_1,T_1]$.
\item (Continuity of the flow map): For any $R>0$, $u_0^*\in B_R(H^{s_c}(\R))$ and $T>0$ given above, the flow map $u_0\mapsto u\in X^{s_c}(T)$ is Lipschitz continuous.
\end{itemize}
\end{theorem}

\begin{proof}[Proof of Theorem \ref{thm8-4}]
We only consider the case of $\gamma=1$, since the case of $\gamma=2$ can be treated in the similar manner.\\
(Existence): The proof is based on the argument of the proof of \cite[Theorem 6.2]{IKO16}. Let $a\in \left(0, \left(\frac{1}{8C_1}\right)^{\frac{1}{m-1}}\right]$ and $\epsilon\in \left(0,\frac{a}{4C_0}\right]$ be positive numbers, where $C_0$ is given by (\ref{XN_Linest_g1}) if $\gamma=1$ and by (\ref{XN_Linest_g2}) if $\gamma=2$, and $C_1$ is defined by (\ref{8-8-5}).
Let $R>0$ be an arbitrary positive number and $u_0^{*}\in B_R\left(H^{s_c}(\R)\right)$. We set $M\in 2^{\Z}$ such as $M:=M\left(u_0^*,\epsilon\right)$. We define $\rho$ and $T$ as
\[
   \rho=\rho(R):=\max\left(4C_0R,a\right),\ \ \ T=T(R,M):=\min\left(1, \left(\frac{a}{8C_2M^{\kappa}\rho^m}\right)^{\frac{1}{\delta}}\right),
\]
where $\kappa$ is given by (\ref{multilin_mhigh_hom_12_lowf4}) if $\gamma=1$ or by (\ref{multilin_mhigh_hom_12_lowf2}) if $\gamma=2$, and $C_2=C_2(m)>0$ is a positive constant depending only on $m$, which is given by (\ref{8-14-5}) below. Then we note that the estimates
\[
   C_0R\le \frac{\rho}{4},\ \ C_0\epsilon\le \frac{a}{4}\le \frac{\rho}{4},\ \ C_1a^m\le \frac{a}{8}\le \frac{\rho}{8},\ \ C_2T^{\delta}M^{\kappa}\rho^m\le \frac{a}{8}\le \frac{\rho}{8}
\]
hold. Let $u_0\in u_0^*+B_{\epsilon}(H^{s_c}(\R))$ be an arbitrary initial data. Then the estimates 
\[
   \|u_0\|_{H^{s_c}(\R)}\le R+\epsilon,\ \ \ \|P_{\ge M}u_0\|_{H^{s_c}(\R)}\le \|P_{\ge M}u_0\|_{H^{s_c}(\R)}+\|P_{\ge M}(u_0-u_0^*)\|_{H^{s_c}(\R)}\le 2\epsilon
\]
hold. We apply the multilinear estimates ((\ref{multilin_mhigh_inhom_g1})-(\ref{multilin_mhigh_hom_12_lowf4}) if $\gamma=1$ or (\ref{multilin_mhigh_inhom_12})-(\ref{multilin_mhigh_hom_12_lowf}) if $\gamma=2$). Thus by the linear estimates ((\ref{XN_Linest_g1}) for $\gamma=1$ or (\ref{XN_Linest_g2}) for $\gamma=2$), the estimates
\begin{align}
    \left\|\Psi[u]\right\|_{X^{s_c}(T)}&\le \left\|e^{it\partial_x^4}u_0\right\|_{X^{s_c}(T)}
    +\sum_{k=0}^m\left|C_{k}\right|\left\|\mathcal{I}\left[\partial_x^{\gamma}\left(u^k\overline{u}^{m-k}\right)\right]\right\|_{X^{s_c}(T)}\notag\\
    &\le C_0\|u_0\|_{H^{s_c}(\R)}+C_{s_c}\sum_{k=0}^m\left|C_{k}\right|\left\|\partial_x^{\gamma}\left(u^k\overline{u}^{m-k}\right)\right\|_{Y^{s_c}(T)}\notag\\
    &\le C_0\|u_0\|_{H^{s_c}(\R)}+C_{s_c}\sum_{k=0}^m\left|C_{k}\right|\left\|\partial_x^{\gamma}\left\{(P_{\ge M}u)^k(P_{\ge M}\overline{u})^{m-k}\right\}\right\|_{Y^{s_c}(T)}\notag\\
    &\ \ \ \ \ \ \ \ \ \ \ \ \ \ \ \ \ \ \ \ \ \ \ \  +C_{s_c}\sum_{k=0}^m\left|C_{k}\right|\left\|\partial_x^{\gamma}\left\{u^k\overline{u}^{m-k}-\left(P_{\ge M}u\right)^k\left(P_{\ge M}\overline{u}\right)^{m-k}\right\}\right\|_{Y^{s_c}(T)}\notag\\
    &\le C_0\left(R+\epsilon\right)
    +C_1\left\|P_{\ge M}u\right\|_{X^{s_c}(T)}^m+C_2 T^{\delta}M^{\kappa}\left\|u\right\|_{X^{s_c}(T)}^m\notag\\
    &\le C_0\left(R+\epsilon\right)+C_1a^m+C_2 T^{\delta}M^{\kappa}\rho^m\le \frac{3}{4}\rho
    \label{8-14-5}
\end{align}
hold. In the similar manner as the proof of the above estimates (\ref{8-14-5}), the inequalities
\begin{equation}
\label{8-15-5}
   \left\|P_{\ge M}\Psi[u]\right\|_{X^{s_c}(T)}
   \le 2C_0\epsilon+C_1a^m+C_2 T^{\delta}M^{\kappa}\rho^m\le \frac{3}{4}a
\end{equation}
hold. The estimates (\ref{8-14-5})-(\ref{8-15-5}) imply that the nonlinear mapping $\Psi$ is well defined from $X^{s_c}(T;\rho,a)$ to itself. By the identity (\ref{8-9-4}), in the similar manner as the proof of the estimate (\ref{8-14-5}), the inequalities
\begin{align}
     &\left\|\Psi[u]-\Psi[w]\right\|_{X^{s_c}(T)}\notag\\
    &\le \sum_{k=0}^m\left|C_{k}\right|\left\|\mathcal{I}\left[\partial_x^{\gamma}\left(u^k\overline{u}^{m-k}-w^k\overline{w}^{m-k}\right)\right]\right\|_{X^{s_c}(T)}\notag\\
    &\le C_s \sum_{k=0}^m\left|C_{k}\right|\left\|\partial_x^{\gamma}\left(u^k\overline{u}^{m-k}-w^k\overline{w}^{m-k}\right)\right\|_{Y^{s_c}(T)}\notag\\
    &\le
    C_s\sum_{k=0}^m\left|C_k\right|\left[\sum_{i=1}^k\left\|\partial_x^{\gamma}\left\{(u-w)\overline{u}^{m-k}u^{k-i}w^{i-1}\right\}\right\|_{Y^{s_c}(T)}+\sum_{i=1}^{m-k}\left\|\partial_x^{\gamma}\left\{(\overline{u-w})\overline{u}^{m-k-i}w^k\overline{w}^{i-1}\right\}\right\|_{Y^{s_c}(T)}\right]
\notag\\
    &\le C_s\sum_{k=0}^m\left|C_{k}\right|\sum_{i=1}^m\left(T^{\delta}M^{\kappa}\|u\|_{X^{s_c}(T)}^{m-i}\|w\|_{X^{s_c}(T)}^{i-1}+\|P_{\ge M}u\|_{X^{s_c}(T)}^{m-i}\|P_{\ge M}w\|_{X^{s_c}(T)}^{i-1}\right)\|u-w\|_{X^{s_c}(T)}\notag\\
    &\le \left(C_1a^{m-1}+C_2T^{\delta}M^{\kappa}\rho^{m-1}\right)\|u-w\|_{X^{s_c}(T)}\le \frac{1}{2}\|u-w\|_{X^{s_c}(T)}
    \label{8-16-9}
\end{align}
hold. This implies that the nonlinear mapping $\Psi$ is a contraction mapping. Thus by the contraction mapping principle, we see that there exists a unique solution $u\in X^{s_c}\left(T;\rho,a\right)$ to (\ref{1-1})-(\ref{nonl_sc_2}).\\
(Uniqueness): We note that for any $w\in X^{s_c}(T)$, the identity $\lim_{M\rightarrow \infty}\|P_{\ge M}w\|_{X^{s_c}(T)}=0$ holds. Thus by the similar argument as the proof of the uniqueness part of Theorem \ref{lwpcomp2} and the estimate (\ref{8-16-9}), we can prove the uniqueness part.
\\
(Continuity of the flow map): Let $u_0, w_0\in u_0^*+B_{\epsilon}(H^{s_c}(\R))$ and $u,w\in X^{s_c}(T)$ be the corresponding solutions given by the Existence part. In the similar manner as the proof of the estimate (\ref{8-16-9}), the estimate
\[
   \|u-w\|_{X^{s_c}(T)}\le C_0\|u_0-w_0\|_{H^{s_c}(\R)}+\frac{1}{2}\|u-w\|_{X^{s_c}(T)}
\]
hold, which implies that the flow map $u_0^*+B_{\epsilon}(H^{s_c}(\R))\mapsto X^{s_c}(T)$,\ $u_0\mapsto u$ is Lipschitz continuous.
\end{proof}

\appendix
\section{Derivation of an important 4NLS model with third order derivative nonlinearities}

\label{Aderi}
In this appendix, we derive the important 4NLS model with third order derivative nonlinearities ($\gamma=3$), that is,  (\ref{1-1})-(\ref{nonli}), from the second ($n=2$) of the derivative nonlinear Schr\"odinger (DNLS) hierarchy (\ref{1-6-a}). To describe the DNLS hierarchy (\ref{1-6-a}) more precisely, we give the definitions of several notations. For a complex-valued function $u=u(x)$ on $\R$, we introduce a $\C^2$-valued function $U=U(x)$ on $\R$ defined by $U:=(u,\overline{u}){}^{\text{T}}$. Let $\sigma_3$ be the third Pauli matrix given by
\[
   \sigma_3:=\begin{pmatrix}
1 & 0 \\
0 & -1 \\
\end{pmatrix}.
\]
For a $\C^2$-valued smooth function $(v,w){}^{\text{T}}$ on $\R$, we introduce a first order differential operator $\mathfrak{D}_1$ defined by
\[
     \mathfrak{D}_1\begin{pmatrix}
v \\
w \\
\end{pmatrix}(x)
:=\sigma_3\frac{d}{dx}\begin{pmatrix}
v \\
w \\
\end{pmatrix}(x)
=\begin{pmatrix}
v_x(x) \\
-w_x(x)\\
\end{pmatrix}.
\]
Moreover, for a $\C^2$-valued smooth function $(v,w){}^{\text{T}}$ decaying $0$ as $|x|\rightarrow\infty$, we introduce a linear operator $\mathfrak{D}_2$ defined by
\[
     \mathfrak{D}_2\begin{pmatrix}
v \\
w \\
\end{pmatrix}(x)
:=-U(x)\int_x^{\infty}U(y)^*\frac{d}{dy}\begin{pmatrix}
v(y) \\
w(y) \\
\end{pmatrix}dy
=-\begin{pmatrix}
u(x) \\
\overline{u}(x)\\
\end{pmatrix}
\int_x^{\infty}\left\{\overline{u(y)}v_y(y)+u(y)w_y(y)\right\}dy.
\]
We note that for a smooth function $u=u(x)$ decaying $0$ as $|x|\rightarrow\infty$, the identity
\[
    \mathfrak{D}_2U(x)
=\begin{pmatrix}
|u(x)|^2u(x) \\
|u(x)|^2\overline{u(x)}\\
\end{pmatrix}
\]
holds for any $x\in\R$. Indeed, this identity follows from the following identities
\[
   -\int_x^{\infty}\left\{\overline{u(y)}u_y(y)+u(y)\overline{u_y(y)}\right\}=-\int_x^{\infty}\frac{d}{dy}|u(y)|^2dy=|u(x)|^2.
\]
For a smooth $\C^2$-valued function $(v,w){}^{\text{T}}$ decaying $0$ as $|x|\rightarrow\infty$, we define the recursion operator $\Lambda$ given by
\begin{equation}
\label{A-1}
    \Lambda\begin{pmatrix}
v \\
w\\
\end{pmatrix}
    :=\frac{i}{2}\left(\mathfrak{D}_1+i\mathfrak{D}_2\right)\begin{pmatrix}
v \\
w\\
\end{pmatrix}.
\end{equation}
By using this operator, we can write the $n$-th of the derivative nonlinear Schr\"odinger hierarchy as
\begin{equation}
\label{dnlsh}
    i\partial_t U(t,x)+\partial_x\left\{(-2i\Lambda)^{2n-1}U\right\}(t,x)=0,\ \ \ (t,x)\in \R\times\R,
\end{equation}
where $U=U(t,x)=\left(u(t,x),\overline{u(t,x)}\right){}^{\text{T}}$ is a smooth solution decaying $0$ as $|x|\rightarrow\infty$ and $n\in \N$.

In the following, we only consider the case of $n=2$. By a simple calculation, the identity
\begin{align*}
    (-2i\Lambda)^3=(\mathfrak{D}_1+i\mathfrak{D}_2)^3=\mathfrak{D}_1^3-\mathfrak{D}_1\mathfrak{D}_2^2-\mathfrak{D}_2(\mathfrak{D}_1\mathfrak{D}_2+\mathfrak{D}_2\mathfrak{D}_1)+i\left\{\mathfrak{D}_1(\mathfrak{D}_1\mathfrak{D}_2+\mathfrak{D}_2\mathfrak{D}_1)+\mathfrak{D}_2\mathfrak{D}_1^2-\mathfrak{D}_2^3\right\}
\end{align*}
holds. By a simple calculation and taking the first component of the equation (\ref{dnlsh}), we can derive the equation (\ref{1-1}) with (\ref{nonli}).

\subsection*{Acknowedgements}
{\rm
\ \ The authors express deep gratitude to Professor Hervert Koch for many useful suggestions and comments. They also deeply grateful to Professor Kenji Nakanishi and Dr. Yohei Yamazaki for pointing out the completely integrable structure for the fourth order Schr\"odinger equation with third-order derivative nonlinearities.
The first author is supported by Grant-in-Aid for Young Scientists Research (B) No.17K14220 and Program to Disseminate Tenure Tracking System from the Ministry of Education, Culture, Sports, Science and Technology. The second author is supported by JST CREST Grant Number JPMJCR1913, Japan and Grant-in-Aid for Young Scientists Research (B) No.15K17571 and Young Scientists Research (No.19K14581), Japan Society for the Promotion of Science. The third author was supported by RIKEN Junior Research Associate Program.
}



\end{document}